\documentclass[12pt]{amsart}
\usepackage{amsthm, amsfonts, amssymb,pinlabel,graphicx}
\usepackage[usenames,dvipsnames]{xcolor}
\usepackage{comment}
\usepackage{pinlabel}
\usepackage{bbm}
\usepackage{bm}
\usepackage{tikz-cd}
\usepackage[pdftex,%
  final,%
  colorlinks=true,%
  linkcolor=NavyBlue,%
  citecolor=NavyBlue,%
  filecolor=NavyBlue,%
  menucolor=NavyBlue,%
  urlcolor=NavyBlue,%
  bookmarks=true,%
  bookmarksdepth=3,%
  bookmarksnumbered=true,%
  bookmarksopen=true,%
  bookmarksopenlevel=2,%
]{hyperref}
\usepackage{autonum}
\usepackage{soul}
\usepackage[foot]{amsaddr}

\numberwithin{equation}{section}
\graphicspath{{Figures/}}

\theoremstyle{plain}
\newtheorem{thm}{Theorem}[section]

\newtheorem{lem}[thm]{Lemma}

\newtheorem{prop}[thm]{Proposition}

\newtheorem{assump}[thm]{Assumption}

\theoremstyle{definition}
\newtheorem{defn}[thm]{Definition}

\theoremstyle{remark}
\newtheorem{rem}[thm]{Remark}
\newtheorem{ex}[thm]{Example}

\newcommand{\R}{\ensuremath{\mathbb{R}}}
\newcommand{\Z}{\ensuremath{\mathbb{Z}}}

\newcommand{\cP}{\ensuremath{\mathcal{P}}}

\newcommand{\ms}{\ensuremath{\mathcal{M}}}  
\newcommand{\crit}{\ensuremath{\mathrm{Crit}}} 
\newcommand{\ind}{\ensuremath{\mathrm{ind}}} 
\newcommand{\find}{\ensuremath{\mathrm{Ind}}} 
\newcommand{\loc}{\ensuremath{\mathrm{loc}}} 
\newcommand{\coker}{\ensuremath{\mathrm{coker}}} 
\newcommand{\hess}{\ensuremath{\mathrm{Hess}}} 
\newcommand{\supp}{\ensuremath{\mathrm{supp}}}
\newcommand{\p}{\ensuremath{\partial}} 
\newcommand{\cH}{\ensuremath{\mathcal{H}}}
\newcommand{\cB}{\ensuremath{\mathcal{B}}}
\newcommand{\cI}{\ensuremath{\mathcal{I}}}
\newcommand{\cR}{\ensuremath{\mathcal{R}}}
\newcommand{\tpsi}{\ensuremath{\psi^\tau}}
\newcommand{\Ruz}{\ensuremath{u_0^{R_- + R_0^-}}}
\newcommand{\Rup}{\ensuremath{u_+^{R_- + R_0 + R_+}}}
\newcommand{\tu}{\ensuremath{u^{\tau}}}

\newcommand{\ob}{\ensuremath{\mathcal{O}}} 
\newcommand{\os}{\ensuremath{\mathfrak{s}}} 

\DeclareMathOperator{\im}{Im}

\newcommand{\ol}{\overline}

\begin{document}
\title[Obstruction Bundle Gluing on Morse Flowlines]{An Invitation to Obstruction Bundle Gluing through Morse Flowlines}

\begin{abstract}
We adapt ``Obstruction Bundle Gluing (OBG)" techniques from \cite{Hutchings-Taubes_2007} and \cite{Hutchings-Taubes_2009} to Morse theory.
We consider Morse function-metric pairs with gradient flowlines that have nontrivial yet well-controlled cokernels (i.e., the gradient flowlines are not transversely cut out). We investigate (i) whether these nontransverse gradient flowlines can be glued to other gradient flowlines and (ii) the bifurcation of gradient flowlines after we perturb the metric to be Morse-Smale.
For the former, we show that certain three-component broken flowlines with total Fredholm index $2$ can be glued to a one-parameter family of flowlines of the given metric if and only if an explicit (essentially combinatorial, and straightforward to verify) criterion is satisfied.  For the latter, we provide a similar combinatorial criterion of when certain $2$-level broken Morse flowlines of total Fredholm index $1$ glue to index $1$ gradient flowlines after perturbing the metric.
Our primary example is the ``upright torus," which has a flowline between the two index-$1$ critical points.
\end{abstract}

\date{}


\author[Ipsita Datta]{Ipsita Datta $^{1}$}
\address{$^1$Department of Mathematics ETH Z\"{u}rich, Switzerland;\linebreak\tiny{ipsita.datta@math.ethz.ch}}
\author[Yuan Yao]{Yuan Yao $^{2}$}
\address{$^2$The University of Texas at Austin Department of Mathematics, USA;\linebreak \tiny{yuan.yao.yy.1995@gmail.com}}

 \thanks{\emph{Subjclass[2025]}:	57R58, 	53D40.}

\maketitle

\tableofcontents

\section{Introduction}
Morse homology is now a widely studied homology theory. Apart from being very visual, it provides a sound stage for testing new techniques for studying Floer theories and moduli spaces of $J$-holomorphic curves. This paper's primary goal is to adapt the ``Obstruction Bundle Gluing (OBG)'' techniques from \cite{Hutchings-Taubes_2009} and \cite{Hutchings-Taubes_2007}, which were developed in the context of Embedded Contact Homology (ECH), to the setting of Morse theory. The OBG techniques were introduced into symplectic geometry by Hutchings and Taubes to show that the ECH differential squares to zero, even though not all the holomorphic curves that appear in the boundary of the $1$-dimensional moduli spaces under consideration are ``transversely cut-out". In particular, they glue holomorphic buildings to holomorphic curves even when some components of the building are not transversely cut out. Our primary motivation is to provide an expository account of the OBG techniques in the Morse setting in detail, without the complications that arise in the $J$-holomorphic curve context, with the hope that the symplectic geometry community will widely use it.\footnote{There has been some earlier discussion of how to use obstruction bundle techniques in the Morse setting, see the blogpost \cite{Hutchings_blog} outlining how to use this technique to study bifurcations of gradient flow trajectories for circle-valued Morse functions. The blog post only examines the linearized obstruction section and does not provide the analytic details for analyzing the full obstruction section. We can view the present article as fleshing out how one would need to fill in those analytical details.}

\begin{figure}[h]
        \centering
        \includegraphics[scale=0.5]{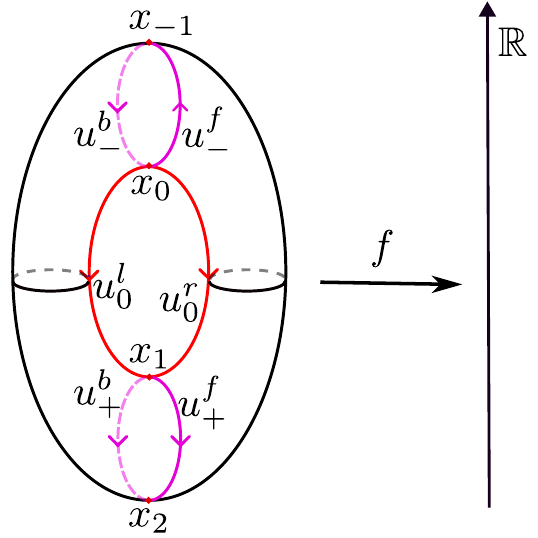}
        \caption{Height function on the upright torus and flowlines for the metric induced by restricting the standard Euclidean metric on $\R^3$.}
        \label{fig: height function on the torus}
    \end{figure}

Examples of non-transversely cut out flowlines in Morse theory are, in fact, not hard to find. For Morse functions on surfaces, nontrivial cokernels naturally arise when the Morse function is invariant under a $\Z_2$-symmetry. Our primary example of the ``upright torus" exploits this symmetry. Consider the height function $f$ on the torus embedded in $\R^3$ such that the embedding is symmetric with respect to the $\Z_2$ action sending $x \mapsto -x$. Endow the torus with the metric $g$ coming from restricting the standard metric on $\R^3$. The height function is a Morse function, but the pair $(f,g)$ is not Morse-Smale - there are flowlines between the two index $1$ critical points.  

Most of this paper is devoted to what we informally call ``$0$-gluing", which refers to the gluing of total Fredholm index $2$ broken flowlines with three components, one of which is not transversely cut out and has Fredholm index $0$, without perturbation of the metric. In particular, consider a pair $(f,g)$ of a Morse function and a metric\footnote{ To keep the analysis to manageable levels, we impose some simplifying assumptions on the form of the metric near the critical points, in Assumptions \ref{assumption on setup}. We explain how we expect to get rid of the simplifying assumptions in Remark \ref{rem:nonlinear_metric}.} on a closed surface with a broken flowline, $(u_-, u_0, u_+)$ with $3$ continuous pieces, $u_-, u_0, u_+$, of indices $1, 0, 1$, resp. The broken flowline $(u_-, u_0, u_+)$ has a unique family $\{u_t\}_{t \in (0, \epsilon)}$ of flowlines with respect to the same metric $g$,  converging to it when the components of the broken flowline satisfy the required asymptotic conditions. These asymptotic conditions are very easy to read off. If these asymptotic conditions are not satisfied, there is no such limiting family, that is, there is no ``$0$-gluing". 

The heart of the argument is the construction of the ``obstruction bundle" and an ``obstruction section" $\os$ such that there is a glued flowline every time the obstruction section vanishes. We study the zero set of the obstruction section by studying the zero set of an approximation called the ``linearized" obstruction section $\os_0$. We show that the linearized section is ``$C^1$-close'' to the obstruction section; hence, understanding the zero set of $\os_0$ is sufficient to understand the zero set of the original section $\os$. 

The linearized section is easy to understand as it is a linear combination of exponential functions. 
We show (under assumed conditions) that the zero set of the obstruction section $\os^{-1}(0)$, which describes the moduli space of flowlines, is well-behaved (i.e. a manifold) even though we can have non-transversely cut-out flowlines. For each $R_0 \gg 0$, the linearized obstruction section has a graph like one of the graphs in Figure~\ref{fig: linearized section}; therefore, it has either a unique solution (corresponding to unique gluing) or no solution (no possible gluing).

\begin{figure}[h]
    \centering
    \includegraphics[scale=0.5]{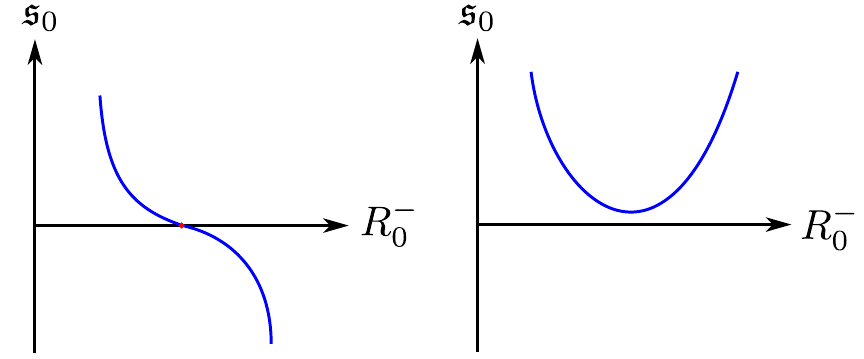}
    \caption{Linearized obstruction sections depending on the asymptotic of $u_-, u_0$, and $u_+$. For a fixed parameter $R_0$ large enough, the $x$-axis plots one of the gluing parameters $R_0^-$ and the $y$-axis plots the ``linearized" obstruction section $\os_0$ computed with parameters $(R_0^-, R_0 - R_0^+)$.}
    \label{fig: linearized section}
\end{figure}

Additionally, we describe gluing of total Fredholm index one, 2-level broken flowlines containing an index $0$ nontransverse component after a perturbation of the metric. We refer to this gluing informally as ``$t$-gluing" where $t$ denotes the perturbation parameter. A broken flowline with total Fredholm index $1$ is called $t$-gluable for a chosen perturbation $\{g_t\}_{t \in (0, \epsilon)}$ such that $(f, g_t)$ is Morse-Smale for each $t$, if there exists a $1$-parametric family 
\begin{align}
    u_t \in \ms(x_{-1}, x_1; g_t)
\end{align} 
converging to it.  Just as in the $0$-gluing case, whether a broken flowline can be $t$-glued or not depends on the asymptotics of the components and the choice of perturbation. 

 We illustrate both $0$-gluing and $t$-gluing on the upright torus. In Example~\ref{ex: 0-gluing on upright torus}, we show that a broken flowline is $0$-gluable if and only if all the components of the broken flowline lie on the same side of the plane $\{z=0\}$. This is an instance of the asymptotic conditions mentioned, and we explicitly work these out. For example, in the notation of Figure~\ref{fig: height function on the torus}, the broken flowline $(u_-^f,u_0^r,u_+^f)$ is 0-gluable, while $(u_-^b,u_0^r,u_+^f)$ is not.
 
 This type of gluing also appears in Kronheimer--Mrowka's \cite{Kronheimer-Mrowka_2007} where they describe Morse functions on manifolds with boundary. For example, we can view the torus in Figure~\ref{fig: height function on the torus} as the double of an annulus, as in Figure~\ref{fig:KM example}. Restricting the height function to the annulus gives us an example of the Kronheimer--Mrowka setup, and we observe the phenomenon of "boundary-obstructed" flowlines. More details are in Example~\ref{Kronheimer-Mrowka example}.
\begin{figure}[h]
    \centering
    \includegraphics[scale=0.5]{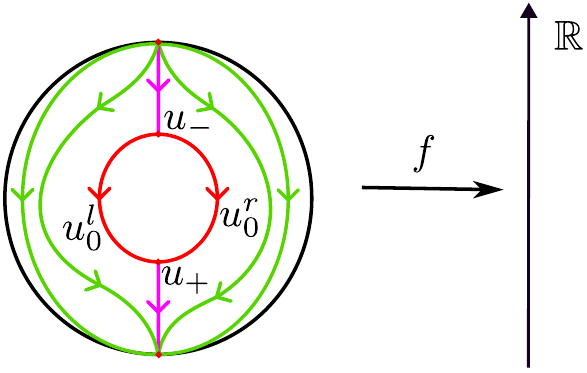}
    \caption{Height function on annulus in $\R^2$ is a Morse function with flowlines (with respect to the restriction of the standard Euclidean metric on $\R^2$) that are either entirely contained in the boundary or disjoint from the boundary. The flowlines $u_0^r$ and $u_0^l$ are ``boundary obstructed" as in \cite{Kronheimer-Mrowka_2007}. The three component flowlines $(u_-, u_0^l, u_+)$ and $(u_-, u_0^r, u_+)$ are $0$-gluable.}
    \label{fig:KM example}
\end{figure}

For $t$-gluing we can imagine ``tilting" the torus in $\R^3$ so that the $\Z_2$-symmetry is broken, refer Figure~\ref{fig: t-gluing on torus}, Example~\ref{ex: t-gluing on torus}. Figure~\ref{fig: t-gluing on torus} shows how one such tilting and the resulting glued flowlines. Note that $t$-gluing tells us that we could define the Morse complex even when the setup is not Morse-Smale. The differential is defined by counting all the (possibly broken) flowlines of total index $1$ that either survive under an a priori choice of perturbation or can be $t$-glued for the same perturbation. Theorem~\ref{thm: t-gluing} and Theorem~\ref{thm: multi level t-gluing} imply that this complex is precisely equal to the Morse complex for a nearby Morse-Smale pair. 
\begin{figure}[h]
    \centering
    \includegraphics[scale = 0.5]{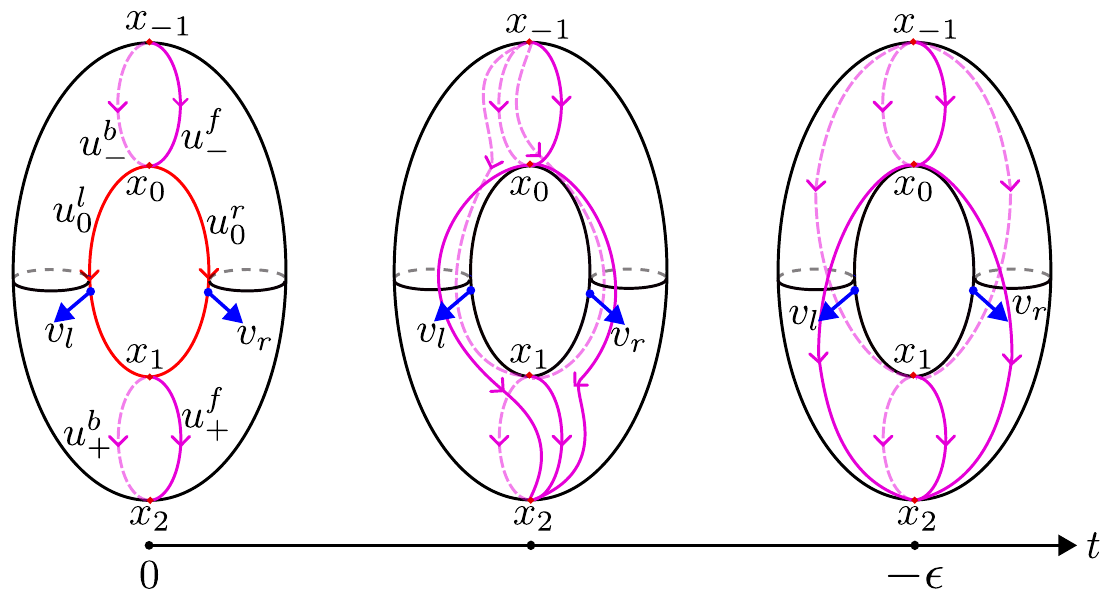}
    \caption{$t$-gluing on the torus via perturbing the metric the torus. Refer to Theorem \ref{thm: t-gluing} for how to read off which trajectories glue and which do not.}
    \label{fig: t-gluing on torus}
\end{figure}

\subsection{OBG overview} The obstruction bundle gluing technique can be boiled down to a sequence of steps.

Let us briefly explain these steps in $0$-gluing a $3$-component index $2$ broken flowline.
The setup involves two transversely cut-out flowlines, $u_\pm$, with a non-transversely cut-out flowline, $u_0$, in the middle, as shown in Figure~\ref{fig: height function on the torus}. 
\begin{enumerate}
\item We {\bf preglue} $(u_-, u_0, u_+)$, using pregluing parameters $(R_-, R_0^-, R_0^+, R_+)$ and a choice of two constants $h, \gamma > 0$. It will become clear later how $h$ and $\gamma$ should be chosen. Later, we will reduce the number of gluing parameters to two by setting $R_-$ and $R_+$ to be multiples of $R_0^-$ and $R_0^+$, respectively. It is easier to compute with four gluing parameters initially. This is Section \ref{subsec:preglue}.
\item We let $\psi_\pm$ be vector fields over $u_\pm$ and $\psi_0$ be a vector field over $u_0$. We {\bf deform the pregluing} by patching together these vector fields. This is Section \ref{sec: deformation of pregluing}.
\item We want to now find which perturbations of the pregluing satisfy the flowline equation. To do this, we {\bf split the equation} into three parts, $\Theta_\pm (\psi_\pm, \psi_0)$ as equations over the domains of $u_\pm$ and $\Theta_0 (\psi_-, \psi_0, \psi_+)$ as an equation over the domain of $u_0$. In the setup of this paper, all of these domains are $\R$, but it is still useful to remember the association. This is Section \ref{sec: equation for deformation to be a flowline}.
\item The flowlines $u_\pm$ are transversly cut out, and so we can {\bf solve the equations $\Theta_\pm$} in relatively straight forward manner. We use the chosen right inverses of the linearized differentials $D_\pm$ to construct contraction maps, whose fixed points give us the required solutions. We get $\psi_\pm$ as functions of $\psi_0$. This is Section \ref{sec: solving for psi + and psi -}.
\item Since $u_0$ is not transversely cut out, we have to work harder to solve $\Theta_0$, and it is here that the main OBG techniques show up. We further {\bf split $\Theta_0$} into its projection onto the image of $D_0$ and the cokernel of $D_0$. For this, we have to fix a projection onto the image. We can {\bf solve the projection to the image} in a similar way to $\Theta_\pm$ as $D_0$ is obviously surjective onto its image. This is Section \ref{sec: solving for psi 0}.
\item We plug in the $\psi_\pm$ and $\psi_0$ obtained from the previous steps into the projection of $\Theta_0$ onto the cokernel and {\bf view it as the obstruction} to gluing. In effect, we produce a finite-dimensional reduction of the original problem. We produce a finite dimensional manifold (in this case an open set in $\mathbb{R}^2$ parametrized by $(R_0^\pm)$) as a base space, a vector bundle $\ob$ over this base (in this case the fiber being $\coker D_0$), and a {\bf section $\os$} of this bundle, whose zeroes are in bijection with solutions of $\Theta_0$ and so also in bijection with gluings. This is Section \ref{sec: obstruction section and gluing map}.
\item Using analytic techniques, we {\bf show that $\os$ is ``$C^1$-close'' to another section $\os_0$} of the same bundle, that we refer to as the ``linearized" obstruction section. The section $\os_0$ is not an honest linearization but has the relevent properties of a linearization, namely, it consists of the ``largest" terms and is a good enough approximation. This step includes the trickiest analysis and relies on careful decay estimates of the flowlines $u_\pm$ and $u_0$, but also of the chosen cokernel element $\sigma_0$. This is explained in Sections \ref{sec: C0 closeness of linearized section} and \ref{sec: C1 estimates}.
\item One can by hand {\bf count of the number of zeroes of $\os_0$}\footnote{To be completely precise, the zero set $\os^{-1}(0)$ is a collection of $1$-dimensional manifolds even after modding out by the reparametrization of the domain. We are in effect counting the number of connected components of this 1-manifold near the broken trajectory. We do this by counting the number of zeroes of $\os_0$ after fixing a gluing parameter $R_0$.} and thus count the number of gluings. We note that the count of zeroes is easy to compute and, to a large extent, fully combinatorial, once the hard analysis work has been done. 
\end{enumerate}

The key to the success of this strategy lies in the following two points:
\begin{enumerate}
\item A good understanding of the cokernel of $u_0$. In this case, we were able to completely describe it via the identification to a specific vector field on the flowline $u_0$ (Equation~\ref{eqn: identification of coker with perp}).
\item Being able to identify which terms in $\os$ make the largest contributions, and how they vary on the base of the obstruction bundle, that is, with varying gluing parameters. This includes a careful understanding of asymptotic decays of not only $u_\pm$ and $u_0$, but also of the vector fields $\psi_\pm$, $\psi_0$, and the cokernel element $\sigma_0$.
\end{enumerate}
\subsubsection{More technical remarks}
In this subsection, we compare our construction to that of \cite{Hutchings-Taubes_2009, Hutchings-Taubes_2007} and highlight features of our construction that differ from the \cite{Hutchings-Taubes_2009, Hutchings-Taubes_2007} one.

While it is true our proof follows the same strategy as in \cite{Hutchings-Taubes_2009, Hutchings-Taubes_2007}, 
our analysis has one additional complication.
In \cite{Hutchings-Taubes_2009, Hutchings-Taubes_2007}, they glue three-level buildings, where the middle levels are branched covers of trivial cylinders. As the middle levels are (branched covers of) trivial cylinders, they do not contribute a pregluing error. 

In our case, the middle segment is not a trivial cylinder and hence contributes a pregluing error. However, it turns out that this new extra pregluing error does not substantially change the obstruction section.  
The two main technical challenges from this new contribution are:
\begin{itemize}
\item The new pregluing error changes the form of the equations in a way that makes it difficult to invoke exponential decay estimates directly from \cite{Hutchings-Taubes_2009, Hutchings-Taubes_2007};
\item The new pregluing error
needs to be shown to be small compared to the linearized section.
\end{itemize}
The main difficulty with the first point is that the exponential decay estimates \cite{Hutchings-Taubes_2009, Hutchings-Taubes_2007} require the equations $\Theta_\pm$ and $\Theta_0$ to be ``autonomous'' in certain regions of the domain for us to see exponential decay (See Proposition~\ref{prop:autonomous} and the surrounding discussion for an elaboration). However, the new pregluing error introduced by the middle segment loses this autonomous behaviour. 
We get around this challenge by assuming that the metric is Euclidean in a Morse chart around the critical point, which makes the equation ``autonomous'' in the correct regions to invoke the exponential decay estimates of \cite{Hutchings-Taubes_2009, Hutchings-Taubes_2007}. This exponential decay over the autonomous regions is essential to the analysis of the obstruction section. Dropping the assumption on the metric would imply a much more careful pregluing to reduce the pregluing error (or rather, reduce the support of the pregluing error). We explain how we expect this to be done in Remark~\ref{rem:nonlinear_metric}. We expect a similar construction to work in the pseudoholomorphic case.

For the second point, even with the exponential estimates obtained, there is still considerable work to be done as the new pregluing error is not a priori small compared to the other terms in the obstruction section. However, the new pregluing error itself does not appear directly in the obstruction section, but instead appears implicitly through the vector fields $\psi_\pm$. 
We capitalize on this implicit dependence through a careful pregluing construction -- we choose asymmetrical gluing profiles that depend explicitly on the sizes of the different eigenvalues of the Hessian at each critical point, so that the effects introduced by the undesired pregluing error have enough room to ``decay away''. Consequently, the estimates in Sections \ref{sec: C0 closeness of linearized section} and \ref{sec: C1 estimates} are slightly more involved than the analogous estimates appearing in \cite{Hutchings-Taubes_2009, Hutchings-Taubes_2007}.

Another difference from \cite{Hutchings-Taubes_2009, Hutchings-Taubes_2007} lies in how we count the zeroes of the obstruction section.
In \cite{Hutchings-Taubes_2009, Hutchings-Taubes_2007}, they show that the (after restricting to a ``slice'' of the domain) obstruction section has the same number of zeroes as the linearized obstruction section over $\mathbb{Z}$, and then count the number of zeroes of the linearized section. This is partly because the base of their obstruction bundle has a highly complex topology. However, we can show directly that the linearized obstruction section and obstruction section are ``$C^1$-close'' to each other, which provides an explicit description of the zero set of the obstruction section.

\section{Acknowledgements}
The authors would like to thank Yasha Eliashberg, Helmut Hofer, Michael Hutchings, Jo Nelson, and Josh Sabloff for fruitful discussions. During part of this project, the first author was at the Institute for Advanced Study, Princeton, and was supported by NSF grant DMS-1926686, and is currently supported by the FIM at ETH Z\"urich. The second author was partially supported by  ERC Starting Grant No. 851701 and ANR COSY ANR-21-CE40-0002.

\section{Preliminaries}

Consider a Morse function $f: M \to  \R$ on a closed compact manifold $M$. Let $\crit (f) = \{x \in M| df = 0\}$ be the set of critical points of $f$. For a point $x \in \crit(f)$, denote the Hessian by $\hess_x(f)$ and the index of $x$ by $\ind(x)$. 
For a Riemannian metric $g$, denote the gradient vector field by $\nabla f$
and the negative gradient flow by $\phi_f$, that is, 
\begin{align}
\varphi_f : \R \times M \to M, \quad \text{ satisfying}\quad \frac{\p \varphi_f}{\p t}(t, x) = -\nabla f(\varphi_f(t, x)).
\end{align}
A {\bf flowline} is a map $u: \R \to M$  such that $\frac{du}{dt}(t) = -\nabla f(u(t))$.
Denote the limits of a flowline (they exist as $M$ is closed) by
$u(\pm\infty) : = \lim_{t \to \pm \infty} u(t)$. 
For $x \in \crit(f)$, the {\bf stable} submanifold of $x$ is given by
\begin{align}
W^s(x) = \{y \in M | \lim_{t \to + \infty} \varphi_f(t, y) = x \},
\end{align}
and the {\bf unstable} manifold of $x$ is given by
\begin{align}
W^u(x) = \{y \in M | \lim_{t \to - \infty} \varphi_f (t, y) = x \}.
\end{align}
The pair $(f,g)$ is said to be {\bf Morse-Smale} if for any two critical points $x, y \in \crit(f)$, $W^s(x) \cap W^u(y)$ intersect transversely. We will consider a slight relaxation of the Morse-Smale condition, namely, allowing clean intersections instead of only transverse intersections.

Consider a pair $(f,g)$ of a Morse function $f: M \to \R$ and a metric $g$ on $M$.
For two critical points $x, y \in \crit(f)$, let the {\bf moduli space of parametrized flowlines} be
\begin{align}
\widehat \ms(x,y) = \{u \in C^\infty(\R, M) \,:\, \dot{u} + \nabla f \circ u = 0, u(-\infty) = x ,  u(\infty) = y\}
\end{align}
endowed with the topology induced by convergence in $C^\infty_{\rm loc}$ on compact subsets of $\R$.
We get the {\bf moduli space of unparametrized flowlines} by quotienting $\widehat \ms$ by the (free) action of $\R$ by translation on the domain,
\begin{align}
\ms(x,y) : = \widehat\ms(x,y)/\R.
\end{align}
Then the topology on $\ms(x,y)$ is induced by $C^\infty_{\rm  loc}$ convergence on the representative flowlines up to translation in the domain. We will abuse notation and denote elements of $\ms$ by $u: \R \to M$ or $u$, even though we mean an equivalence class.
Let the {\bf moduli space of broken flowlines} between $x$ and $y$ be
\begin{align}
\ol \ms(x,y) = \cup_{c_i \in \crit(f)} \ms(x,c_1) \times \ms(c_1, c_2) \times \dots \ms(c_j, y).
\end{align}
We consider $\ol \ms(x,y)$ again with the topology of $C^\infty_{\rm loc}$ convergence on compact sets up to translation. We call an element of $\ol \ms$ a {\bf broken flowline} and each $u_i$ a component of $\mathbf{u} = (u_1, \dots, u_k)$.

\begin{lem}\label{lem: compactification}
The space $\ol \ms(x, y)$ is compact with respect to the $C^\infty_{\rm loc}$ convergence.
\end{lem}
The compactness proof is the same as that found in various places, for example, see \cite[Section 3.2.b]{Audin-Damian_2014}. We note that unless we assume that $(f,g)$ is Morse-Smale, $\ol\ms(x,y)$ may not be a manifold of the right dimension. For example, there may be broken flowlines in $\ol\ms(x,y)$ that are isolated even when there exist components of $\ol \ms(x,y)$ that have dimensions greater than or equal to one. 
To set up the moduli spaces of flowlines and the required Fredholm theory, we include only the necessary definitions and properties here. We refer the reader to \cite[Section 2.1]{Schwarz_1993} for details. 

We compactify $\R$ as $\ol \R = \R \cup \{ \pm \infty\}$ equipped with the structure of a manifold with boundary by the requirement that 
\begin{align}
h:\ol \R \to [-1, 1], \quad t \mapsto \frac{t}{\sqrt{1 + t^2}}
\end{align}
be a diffeomorphism.
Given arbitrary points $x,y \in M$ we define the set of smooth, compact curves $C^\infty_{x,y}$ as
\begin{align}
C^\infty_{x,y} := C^\infty_{x,y}(\ol \R, M) = \{ u \in C^\infty(\ol \R, M) \, |\, u(-\infty) = x, u(+\infty) = y\}.
\end{align}
Fix a complete metric $g$ on $M$; denote the exponential map by
\begin{align}
\exp : TM \supset \mathcal{D} \to M
\end{align}
where $\mathcal{D}$ is an open and convex neighbourhood of the zero section in the tangent bundle. For any smooth, compact curve $u \in C^\infty(\ol \R, M)$, we denote the pull-back bundles by $u^* \mathcal{D} \subset u^* TM$.
We get a well-defined map
\begin{align}
\exp_u : H^{1,2}_{\R} (u^* \mathcal{D}) &\to C^0 (\ol \R, M)\\
s &\mapsto \exp \circ \psi, \text{ where }(\exp \circ \psi)(t) = \exp_{u(t)} (\psi(t)).
\end{align}
So, we can define the space of curves
\begin{align}
\cP^{1,2}_{x,y} = \cP^{1,2}_{x,y} (\R , M) = \{ \exp \circ \psi \in C^0(\ol \R, M)|\psi \in H^{1,2}_\R(u^*\mathcal{D}), u \in C^\infty_{x,y}(\ol \R, M)\}.
\end{align}

The space of curves $\cP^{1,2}_{x,y} \subset C^0_{x,y}(\ol \R, M)$ is equipped with a Banach manifold structure via the atlas of charts
\begin{align}
\left\{ H^{1,2}_\R(u^* \mathcal{D}), \exp_u \right\}_{u \in C^\infty_{x,y}(\ol \R, M)}.
\end{align}
We represent the tangent space of $\cP^{1,2}_{x,y}$ as
\begin{align}
T\cP^{1,2}_{x,y} = H^{1,2}_\R ({\cP^{1,2}_{x,y}}^*TM) = \bigcup_{\psi \in \cP^{1,2}_{x,y}} H^{1,2}_\R(\psi^* TM).
\end{align} 
This is a Banach bundle on $\cP^{1,2}_{x,y}$ with $H^{1,2}(\R, \R^n)$ as the characteristic fiber. Similarly, we can define the $L^2_\R({\cP^{1,2}_{x,y}}^* TM)$ as
\begin{align}
L^{2}_\R ({\cP^{1,2}_{x,y}}^*TM) = \bigcup_{\psi \in \cP^{1,2}_{x,y}} L^{2}_\R(\psi^* TM).
\end{align} 

\begin{prop}\cite[Proposition 2.8]{Schwarz_1993}
Let $f \in C^\infty(M, \R)$ be an arbitrary smooth real function on $M$. Then, given critical points $x,y \in \crit f$ and a metric $g$, the gradient $\nabla f$ with respect to $g$ induces a smooth 
section in the $L^2$-Banach bundle,
\begin{align}\label{eqn: defn of section F}
F: \cP^{1,2}_{x,y} &\to L^2_\R ({\cP^{1,2}_{x,y}}^* TM)\\
u &\mapsto \dot{u} + \nabla f \circ u.
\end{align}
The zeroes of the section $F$ are exactly the flowlines from $x$ to $y$. 
\begin{align}
\widehat \ms (x,y) = F^{-1}(0) \subset \cP^{1,2}_{x,y}.
\end{align}
\end{prop}
For a zero $u \in F^{-1}(0)$, we can look at the projection of the differential of $F$ to the fibre $L^2(u^* TM)$, referred to as the {\bf linearization} of $F$ and denoted as $D_u$. With the choice of a metric $g$ and in a local chart $H^{1,2}_\R (u^*\mathcal{D})$ around $u \in \cP^{1,2}_{x,y}$, the linearization of $F$ at $u$ is of the form
\begin{align}
D_u : H^{1,2} (u^* TM) \to L^2 (u^* TM)\\
D_u (\psi) = \nabla_s \psi + (\hess_f \circ u)\psi.
\end{align} 
Here $\hess_f$ is the Hessian of $f$ with respect to the metric $g$. At $\pm \infty$, $\hess_f \circ u$, are independent of the metric $g$, non-degenerate, and self-adjoint on $u^* TM|_{\pm \infty}$ with respect to $g$. The linearization $D_u F$ is a Fredholm map with Fredholm index
\begin{align}
\find D_u = \mu(\hess_f(u(- \infty))) - \mu(\hess_f(u(+\infty))),
\end{align}
where $\mu$ denotes the number of negative eigenvalues counted with multiplicity. Note that $\mu(\hess_f(x)) = \ind(x)$ for any $x \in \crit f$.
For a flowline $u$, we refer to the $\find D_u$ as the {\bf Fredholm index} of $u$, that is,
\begin{align}
\find (u) := \find D_u = \ind (x) - \ind (y).
\end{align} 
Let the {\bf total Fredholm index} of a broken flowline $\mathbf{u} = (u_1, \dots, u_k) \in \ol \ms(x,y)$ be the sum of the Fredholm indices of the components $u_i$. In particular, we get
\begin{align}
    \find(\mathbf{u}) = \sum_{i =1}^k \find(u_i) = \ind(x) - \ind(y). 
\end{align}

\section{Transversality and Cokernels}
When the pair $(f,g)$ is assumed to be Morse-Smale, $\ms (x,y)$ is a manifold of dimension $\ind(x) - \ind(y)$ and the linearized operator $D_u$ along a gradient flowline $u$ is surjective with empty cokernel. As explained in the introduction, we relax the Morse-Smale condition to include flowlines with nontrivial cokernels.

We begin by understanding the cokernel of Morse flowlines. As we shall see, having precise control of the cokernel elements will be essential for understanding the obstruction section. Most of the section below is taken from Proposition 10.2.8 of \cite{Audin-Damian_2014}.

We begin by describing the notion of a ``resolvent'' of the linear differential operator.  Let $u$ denote a gradient flowline between critical points $x$ and $y$. Let $D_u$ denote the linearization of the gradient flow equation and let $D_u^*$ denote its formal adjoint. 
If $s$ and $t$ are two real numbers, then let
\begin{align}
\Psi_{(s,t)} : T_{u(s)} M \to T_{u(t)} M
\end{align}
be the {\bf resolvent} of the linear differential equation $D_u = 0$. This means that the map sends a vector $Y \in T_{u(s)} Y$ to the value $\tilde{Y}(t)$ when $\tilde{Y}: \R \to \R^n$ is a solution $D_u \tilde{Y} = 0$ with $\tilde{Y}(s) = Y$. 

Let $W^u(x)$ denote the unstable manifold of $x$ and let $W^s(y)$ denote the stable manifold of $y$.
We also let \[E^u(s):= \{\tilde{Y}\in T_{u(s)}M | \lim_{t\rightarrow -\infty } \Psi _{s,t} \tilde{Y} =0\} \]
and
\[
E^s(s):= \{\tilde{Y}\in T_{u(s)}M | \lim_{t\rightarrow \infty } \Psi _{s,t} \tilde{Y} =0.
\]
Then it is not hard to see\footnote{For instance, we can see this by realizing $W^u(x)$ as the set of maps $u:[0,\infty) \rightarrow M$ that satisfy $u(-\infty) =x$ and $u'+\nabla f(u(s))=0$. See for instance Section 8 of \cite{frauenfelder2020modulispacegradientflow}.} that 
\[
E^u(s) =TW^u_{u(s)}(x), \quad E^s(s)= TW^s_{u(s)}(y).
\]
From which we can deduce the following proposition.
\begin{prop}[Proposition 10.2.8 in \cite{Audin-Damian_2014}]Consider a flowline $u: \R \to M$ for $(f,g)$.
For any $s \in \mathbb{R}$ we have
\[
ker D_u \cong T_{u(s)}W^u(x) \cap T_{u(s)}W^s(y)
\]
\end{prop}
The identification is given as follows: given $Y_s \in  T_{u(s)}W^u(x) \cap T_{u(s)}W^s(y)$, the corresponding kernel element is given by the vector field $Y(s)$ that uniquely solves $D_uY=0$ satisfying the initial condition $Y(s) = Y_s$.

Similarly, studying the resolvent of the adjoint operator $D_u^*$ gives us the following.
\begin{prop} [Proposition 10.2.8 in \cite{Audin-Damian_2014}]
Consider a flowline $u: \R \to M$ for $(f,g)$.
We have
\begin{align}\label{eqn: identification of coker with perp}
\coker D_u \cong \ker D_u^* \cong (T_{u(s)} W^u(x) + T_{u(s)} W^s(y))^\perp.
\end{align}
\end{prop}
In future sections, we will use this discussion to pick cokernel elements with properties we prefer.

\begin{ex}\label{ex: choosing cokernel elements}
\begin{figure}[h]
    \centering
    \includegraphics[scale=0.5]{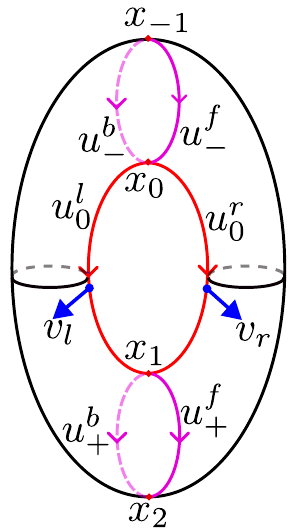}
    \caption{Vectors $v_l$ and $v_r$ generate $\coker D_{u_0^l}$ and $\coker D_{u_0^r}$, resp.}
    \label{fig:generators of cokernels in upright torus}
\end{figure}
Consider the height function on the torus $T^2 \subset \R^3$, refer Figure~\ref{fig: height function on the torus}. We consider an embedding of $T^2$ that is symmetric about the reflection $(x,y,z) \mapsto (-x,y,z)$ across the $(y,z)$-plane. Then we have a maximum at $x_{-1} := (0,0,2)$, two critical points, $x_0 := (0,0,1)$ and $x_1 := (0,0,-1)$, of index $1$, and one minimum at $x_2 := (0,0,-2)$. Let $g$ denote the restriction of the standard Euclidean metric in $\R^3$ to $T^2$.\footnote{Later, we will make some modifications to the metric $g$ so that it is the standard Euclidean metric near a Morse chart. This is so that some of the technical estimates in the gluing analysis become easier; however, the discussion here remains unaffected: the modification can be made in such a way that the non-transversely cut-out gradient flowlines persist and have cokernels described in the same way. }

    Let us call $T^2 \cap \{x \geq 0\}$ the {\bf front side} of the torus, and $T^2 \cap \{x \leq 0\}$ the {\bf back side}. For the metric $g$ the negative gradient $-\nabla f$ has two flowlines from $x_{-1}$ to $x_0$. One of these lies entirely on the front side and one on the back. Let us denote them as $u_-^f$ and $u_-^b$, respectively. Similarly, there are two flowlines, $u_+^f$ and $u_+^b$, from $x_1$ to $x_2$. 

   There are two flowlines from $x_0$ to $x_1$, both lying on $\{x = 0\}$, and both with one-dimensional cokernels. Let us call them $u_0^l$ and $u_0^r$. 
 Let $p_r = u_0^r(0)$.
    The vector $(1,0,0) \in T_{p_r} M$ is a non-zero vector in 
    \begin{equation}
        v_r := (1,0,0) \in (T_{p_r}W^u(x_0) + T_{p_r} W^s(x_1))^\perp. 
   \end{equation}
    So, $\sigma_0^r$ defined by $\sigma_0^r(t) = \Psi_{(0, t)} (1,0,0)$ gives us a generator of $\coker D_{u_0^r}$. Note that for any point $p \in \im u_0^r$, $(1,0,0) \in (T_{p}W^u(x_0) + T_{p} W^s(x_1))^\perp$ and we could have chosen any of these as the ``initial value" for defining a nontrivial element of $\coker D_{u_0^r}$. One-dimensionality implies that this other element would be a positive multiple of $\sigma_0^r$. This means that the function given by $t \mapsto \langle \sigma_0^r(t), (1,0,0)\rangle$ is a non-vanishing function because of the uniqueness of the solution of a differential equation, and so, $t \mapsto \mathrm{sign} \langle \sigma_0^r(t), (1,0,0)\rangle$ is a constant function.

    For $u_0^l$, we do an analogous construction with $p_l = u_0^l(0)$ and $v_l = (1,0,0) \in T_{p_l} M$ to get $\sigma_0^l \in \coker D_{u_0^l}$.
\end{ex}

\begin{rem}
The above upright torus is an example where the stable and unstable manifolds intersect \emph{cleanly} instead of transversely. The computations of this paper concern gluing of flowlines in particular cases, but more generally, we expect these methods can be used to study gluing of flowlines in the case of cleanly intersecting stable/unstable submanifolds.
\end{rem}


\section{Asymptotic estimates}
In this section, we analyze the asymptotics of Morse flowlines and vector fields along Morse flowlines. These are used in multiple ways in the estimates for the gluing construction. In particular, we will use them crucially to show that the linearized obstruction section $\os_0$ is ``$C^1$-close'' to the obstruction section $\os$.

 We first begin by stating our assumptions on the Morse function and our metric.

Consider a pair $(f, g)$ of a Morse function $f: M \to \R$ on an $n$-dimensional smooth manifold and a metric $g$ on $M$.
For any critical point $x\in M$, fix Morse neighbourhood $U$ of $x$ with coordinates $(p_1, \dots, p_n)$. We identify the critical point $x$ itself with $(p_1,..,p_n)=(0,\dots,0).$ We assume the Morse function $f$ is given by 
\begin{align}
f(p_1, \dots, p_n) = f(x) - \frac{1}{2}\sum_{j=1}^{\ind(x)} \lambda^{(j)}  p_j^2 + \frac{1}{2}\sum_{j=\ind(x_0) + 1}^{n} \lambda^{(j)} p_{j}^2.
\end{align}

Here, the positive numbers $\lambda^j$ are the eigenvalues of the Hessian of $f$ at $x$.

At this point, we make one major assumption on the function metric pair $(f,g)$ that simplifies the analysis. This assumption will be used for the rest of the paper. We note this assumption does \underline{not} occur generically, but examples satisfying this assumption exist in great abundance.
\begin{assump}\label{assumption on setup}
 Assume that the metric $g$ is the standard Euclidean metric with respect to these coordinates within the Morse neighbourhoods around all critical points. This means that the exponential map with respect to this metric is simply vector addition within the Morse neighbourhoods. \footnote{This assumption is similar to the assumption of tame $J$ made in the paper \cite{Bao}. See also \cite{rooney, avdek2023}. The assumption simplifies the nonlinear equation to a linear one near the critical points/Reeb orbits to make certain parts of the obstruction bundle analysis easier.}
\end{assump}

From this, the gradient flow equation becomes linear near the critical points. In particular, let $u$ be a solution to 
    \begin{align}
        \frac{d}{ds}u + \nabla f(u) = 0.
    \end{align}
Assume for $s>s_0$, $u$ is near the critical point $u(\infty)=x$.
Then we can write
\begin{align}
 u =  \sum_{j=\ind(x)+1}^{n} v_j e^{-\lambda^{(j)} s}, \quad  s> s_0.
\end{align}
Here $v_j$ is an the eigenvector of $\textup{Hess}_x f$ with eigenvalue $\lambda^{(j)}$.

We also need exponential decay estimates for the kernel of the linearized operator and its adjoint.

\begin{prop}\label{Prop:decay_of_kernel}
Let $D_u$ denote the linearization of the gradient flow equation. Suppose $\psi \in \ker D_u$. Assume for $s>s_0$, $u(s)$ is contained in a Morse neighbourhood containing the critical point $x=u(\infty)$. Then for $s>s_0$ we can write
\begin{equation}\label{eqn:exponential decay of psi}
    \psi= \sum_{j=\ind(x)+1}^{n} v_j e^{-\lambda^{(j)} s}, \quad  s> s_0.
\end{equation}
Here $v_j$ an the eigenvector of $Hess_x f$ with eigenvalue $\lambda^{(j)}$.

If we fix our conventions to be $|\lambda^{ind(x)+1}|\leq |\lambda^{ind(x)+2}|\leq \dots |\lambda^{n}|$, then as a consequence of this, we have
\[
|\psi(s)|\leq |\psi(s_0)|e^{-|\lambda^{ind(x)+1}(s-s_0)|}
\]
A similar expression holds for $\psi$ near the negative end of $u$.
\end{prop}

Equation~\ref{eqn:exponential decay of psi} is valid because, in the Morse neighbourhood, $D_u = \frac{d}{ds} + A$, where $A$ is the constant matrix that is given by the Hessian of $f$ at the critical point. This ODE can be essentially solved by a Fourier series; that is, the solution at any point $s$ is expressed as a linear combination of eigenvectors of the linear operator $A$. That this solution has the form given in Equation~\ref{eqn:exponential decay of psi} comes from the fact that the vector field decays to $0$ at $s=\infty$.

A similar expression holds for the cokernel of $D_u$.

\

\begin{prop}
Let $D_u^*$ denote the adjoint of the operator $D_u$ with respect to the ambient metric. Suppose $\sigma \in ker D_u^*$, for $s>s_0$, we have
\[
 \sigma =\sum_{j=1}^{\ind(x)} v_j e^{-\lambda^{(j)} s}
\]
Here $v_j$ an the eigenvector of $Hess_x f$ with eigenvalue $\lambda^{(j)}$.
Consequently, if we assume $|\lambda^{(1)}|\leq |\lambda^{(2)}|\leq\dots\leq |\lambda^{(ind x)}|$, then
\[
|\sigma(s)|\leq |\sigma(s_0)|e^{-|\lambda^{(1)}(s-s_0)|}
\]
for $s>s_0$.
\end{prop}
To see this, we observe $D^*_u=d/ds-A$ near the critical point.
\begin{rem}
A slightly more complicated expression holds without the assumption that the metric is Euclidean near the critical points.
\end{rem}

\section{Obstruction bundle gluing without perturbation}\label{sec: obg 0-gluing}

In this section, we glue $3$-component broken flowlines where the central component's linearized operator has a one-dimensional cokernel. Namely, we consider broken flowlines of the type $(u_-, u_0, u_+)$ when the total Fredholm index is
\begin{align}
    \find D_{u_-} + \find D_{u_0} + \find D_{u_+} = 2.
\end{align}
Additionally, $\find D_{u_0} = 0$ and has a one-dimensional cokernel. We refer to this gluing informally as {\bf ``$0$-gluing"}. \footnote{The ``$0$'' is to emphasize we don't perturb the metric, to be contrasted with our later ``$t$''-gluing.}

Consider a pair $(f, g)$ of a Morse function $f: M \to \R$ on an $n$-dimensional smooth manifold and a metric $g$ on $M$.
Fix Morse neighbourhoods $U_-$ and $U_+$ of $x_0$ and $x_1$, respectively, such that the Morse function $f$ is given by 
\begin{align}
f(p_1, \dots, p_n) = f(x_0) - \frac{1}{2}\sum_{j=1}^{\ind(x_0)} \lambda_0^{(j)}  p_j^2 + \frac{1}{2}\sum_{j=\ind(x_0) + 1}^{n} \lambda_{0}^{(j)} p_{j}^2, \\
f(q_1, \dots, q_n) = f(x_1) -\frac{1}{2}\sum_{j=1}^{\ind(x_1)} \lambda_1^{(j)} q_j^2  + \frac{1}{2}\sum_{j=\ind(x_1) + 1}^{n} \lambda_0^{(j)} q_{j}^2 ,
\end{align}
for $(p_1, \dots, p_n)$ coordinates on $U_-$ and $(q_1, \dots, q_n)$ on $U_-$. We assume the critical point $x_0$ is identified with $(p_0,..,p_n)=(0,\dots,0)$ and $x_1$ is identified with $(q_1,\dots,q_n)=(0,\dots,0).$

At this point, we remind the reader of the standing assumption Assumption~\ref{assumption on setup} on the function metric pair $(f,g)$, which simplifies the analysis. This assumption does \underline{not} occur generically. 
We are now ready to state the setup of our main theorem.

For $x_{-1}, x_0, x_1, x_2 \in \crit(f)$ with 
$$\ind(x_{-1}) = k+1, \ind(x_0)= \ind(x_1)=k, \text{ and }  \ind(x_2) = k-1$$ 
let 
\begin{align}
u_- \in \ms(x_{-1}, x_0), u_0 \in \ms( x_0, x_1), u_+ \in \ms(x_1, x_2).
\end{align}

We assume 
\begin{align}
u_-(s) \in U_- \text{ for } s > 1,\quad
u_0(s) \in U_- \text{ for } s < -1,\\
u_0(s) \in U_+ \text{ for } s > 1,\quad
u_+(s) \in U_+ \text{ for } s < -1.
\end{align}

Let $\lambda_0^+$ be the smallest positive eigenvalue of the Hessian $\hess_{x_0} f$ and $\lambda_1^-$ be the largest negative eigenvalue (that is, the smallest absolute value) of $\hess_{x_1}f$.

We assume $\dim \coker D_{u_0} =1$, and fix a generator, $\sigma_0 \in \coker D_{u_0}$. We also assume 
there exists $0 \neq b_- \in T_{x_0}M$, $b_+ \neq 0 \in T_{x_1} M$, such that 
\begin{align}\label{eqn: expansion of sigma 0}
    \sigma_0 = \begin{cases}
         e^{\lambda_0^+ s}b_- + \sum_{\lambda_+,v_+} e^{\lambda_+ s} v_+ &s < -1,\\
         e^{\lambda_1^- s}b_+ + \sum_{\lambda_-,v_-} e^{\lambda_- s} v_- & s> 1,
    \end{cases}
\end{align}
where the summation over  $(\lambda_+,v_+)$ denotes any of the positive eigenvalues of $\hess_{x_0} f$ not equal to (hence greater than) $\lambda_0^+$, and $v_+ \in T_{x_0} M$ are the eigenvectors of $\hess_{x_0}f$ with eigenvalue $\lambda_+$. The summation over $(\lambda_-,v_-)$ is similarly defined.

Similarly, we may assume that
\begin{align}
 u_-& =   e^{-\lambda_0^+ s}a_- + \sum_{v_-} e^{-\lambda_+ s} v_- & s> 1,\\
         u_+ & = e^{-\lambda_1^- s}a_+   + \sum_{v_+} e^{-\lambda_- s} v_+ &s < -1,
     \end{align}
     for some nonzero eigenvectors $a_- \in T_{x_0} M$, $a_+ \in T_{x_1} M$ of the Hessians $\hess_{x_0}f$ and $\hess_{x_1}f$ with eigenvalues $\lambda_0^+$ and $\lambda_1^+$.
     The subsequent summation over
     $\lambda_\pm$ and $v_\pm$ similarly defined as in Equation~\ref{eqn: expansion of sigma 0}.
\begin{thm}\label{thm: gluing}
Suppose that
\begin{align}
    \langle a_-, b_- \rangle \neq 0, \quad \langle a_+, b_+ \rangle \neq 0.
\end{align}
Then, if 
\begin{align}\label{eqn: sign condition for gluing}
     \mathrm{sign} \langle a_-, b_- \rangle =  \mathrm{sign} \langle a_+, b_+\rangle,
\end{align}
there exists a unique one parameter family $\{(u_\nu)\}_{\nu \in \R_+} \subset \ms(x_{-1}, x_2)$ that converges to $(u_-, u_0, u_+)$ in the $C^\infty_\loc$-convergence. Otherwise, that is if,
\begin{align}
    \mathrm{sign} \langle a_-, b_- \rangle \neq  \mathrm{sign} \langle a_+, b_+\rangle,
\end{align}
no such family exists, that is, $(u_-, u_0, u_+)$ is not a limit point of flowlines. Refer Figure~\ref{fig: obg gluing}.
\begin{figure}[h]
    \centering
    \includegraphics[scale=0.5]{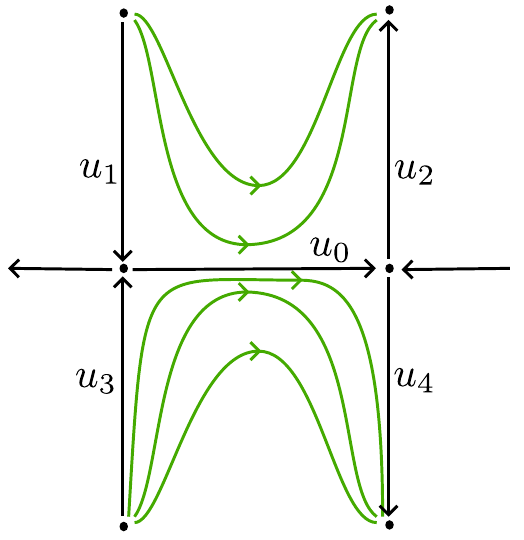}
    \caption{Three-component broken flowlines that can and cannot be glued. As drawn, $(u_1, u_0, u_2)$ and $(u_3, u_0, u_5)$ are $0$-gluable, and $(u_1, u_0, u_4)$ and $(u_3, u_0, u_2)$ are not $0$-gluable.}
    \label{fig: obg gluing}
\end{figure}
\end{thm}
\begin{rem}
    The general form of the gradient flowline and the cokernel element follows from our assumptions on the metric. The real assumption we are making here is the nonzero pairing of the eigenvectors associated with the largest terms appearing in the asymptotic expansion of $u_i$ and $\sigma_0$. In some sense, that this pairing is nonzero is what happens ``generically''. If this pairing is zero, the analysis needs to be done more carefully. For an example of this, see \cite{rooney}. 
\end{rem}

Even though Theorem~\ref{thm: gluing} feels very abstract, it can be applied concretely, especially on surfaces.
We recommend that the reader have the following example in mind throughout the proof of Theorem~\ref{thm: gluing}.
\begin{ex}\label{ex: 0-gluing on upright torus}
Recall the setup in Example~\ref{ex: choosing cokernel elements}. As the ambient dimension is two, the cokernel elements take on a simple form as we saw in Example~\ref{ex: choosing cokernel elements}. Namely, there exist vectors $b_-^r , b_-^l \in T_{x_0} T^2$ and $b_+^r , b_+^l \in T_{x_1} T^2$ such that
\begin{align}
\sigma_0^l &=  e^{-\lambda_0^+ s}b_-^l \quad s < -1, &\sigma_0^l = e^{-\lambda_1^- s}b_+^l \quad s > 1,\\
\sigma_0^r &=  e^{-\lambda_0^+ s}b_-^r \quad s < -1, &\sigma_0^r =  e^{-\lambda_1^- s}b_+^r 
\quad s > 1.
\end{align}
The asymptotic vectors $b_\pm^r$ and $b_\pm^l$  have positive inner products \begin{align}\langle b_\pm^r, (1,0,0)\rangle > 0 \text{ and } \langle b_\pm^r, (1,0,0)\rangle > 0.
\end{align}
    This positivity implies that the sign conditions \ref{eqn: sign condition for gluing} can be reduced to determining whether a flowline is on the front or back side. Namely, 
    \begin{align}
        (u_-^f, u_0^r, u_+^f), (u_-^f, u_0^l, u_+^f), (u_-^b, u_0^r, u_+^b), \text{ and } (u_-^b, u_0^l, u_+^b)
    \end{align}
    are gluable, that is, there exists a unique one-dimensional family of flowlines in $\ms(x_{-1}, x_2)$ that limit to each of these. In contrast, the other combinations,
    \begin{align}
        (u_-^f, u_0^r, u_+^b), (u_-^f, u_0^l, u_+^b), (u_-^b, u_0^r, u_+^f), \text{ and } (u_-^b, u_0^l, u_+^f)
    \end{align}
    are not gluable.
\end{ex}

\begin{ex}\label{Kronheimer-Mrowka example}
    In \cite{Kronheimer-Mrowka_2007} Chapter 2, Kronheimer and Mrowka consider manifolds with boundary and Morse functions that have flowlines tangential to the boundary. They define two different complexes: $\hat C$ generated by interior critical points and critical points on the boundary where the normal to the boundary is an unstable direction for the Hessian, and $\check{C}$ generated by interior critical points and boundary critical points where the normal to the boundary is a stable direction. The differentials count appropriate, possibly broken, flowlines of total index $1$.

    \cite[Theorem 2.4.5]{Kronheimer-Mrowka_2007} states that these actually define homology groups. To prove this, one needs to show that the differentials square to zero, where we can apply Theorem~\ref{thm: gluing}.
    The geometry selects only the gluable flowlines in the manifold with the boundary case. Thus, Theorem~\ref{thm: gluing} implies the gluing counterpart of \cite[Lemma 2.4.3]{Kronheimer-Mrowka_2007}, that is, together they show the following: Suppose $a$ and $c$ are interior critical points with indices $k$ and $k-2$, respectively. Then, the boundary of the moduli space $\ms(a,c)$ consists of all two-component broken flowlines
    \begin{align}
        (u_1, u_2) \in \ms(a,b) \times \ms(b,c),
    \end{align}
    for $b$ interior critical point of index $k-1$ and all the three-component broken flowlines
    \begin{align}
        (u_1, u_2, u_3) \in \ms(a,b_1) \times \ms(b_1, b_2) \times \ms(b_2, c),
    \end{align}
    for boundary critical points $b_1$ and $b_2$ of index $k-1$. 

    As an example, consider the annulus in Figure~\ref{fig:KM example} with Morse function given by projection to the $y$-coordinate. Then the two flowlines along the inner boundary, namely $u_0^l$ and $u_0^r$, have non-trivial ($1$-dimensional) cokernels. We can identify these with the inner pointing normals $\nu^l$ and $\nu^r$ at arbitrary points $p_0^l$ and $p_0^r$, respectively. Then notice that all the $3$-component flowlines satisfy Equation~\ref{eqn: sign condition for gluing}.
\end{ex}

The rest of Section~\ref{sec: obg 0-gluing} is devoted to the proof of Theorem~\ref{thm: gluing}. 

\subsection{Pregluing} \label{subsec:preglue}
This subsection is the analogue of Section 5.2 in \cite{Hutchings-Taubes_2009}.
 Choose four gluing parameters, $R_-, R_+, R_0^-,$ and $R_0^+ > 0$. It will become clear later how these parameters are related. In fact, we can make $R_-$ and $R_+$ depend on $R_0^\pm$, but we keep them separate for now, as it makes the computations easier to understand. 

We now define the {\bf $\mathbf {(R_-, R_0^-, R_0^+,  R_+)}$-pregluing}, $u_{\#}: \R \to M$ for the parameters $(R_-, R_0^-, R_0^+,  R_+)$. Even though all the maps defined in this section depend on $(R_-, R_0^-, R_0^+,  R_+)$, we will not include $(R_-, R_0^-, R_0^+,  R_+)$ in the notation for ease of reading. Denote $R_0 = R_0^- + R_0^+$.
We first need three cutoff functions. 
\begin{figure}[h]
    \centering
    \includegraphics[width=\linewidth]{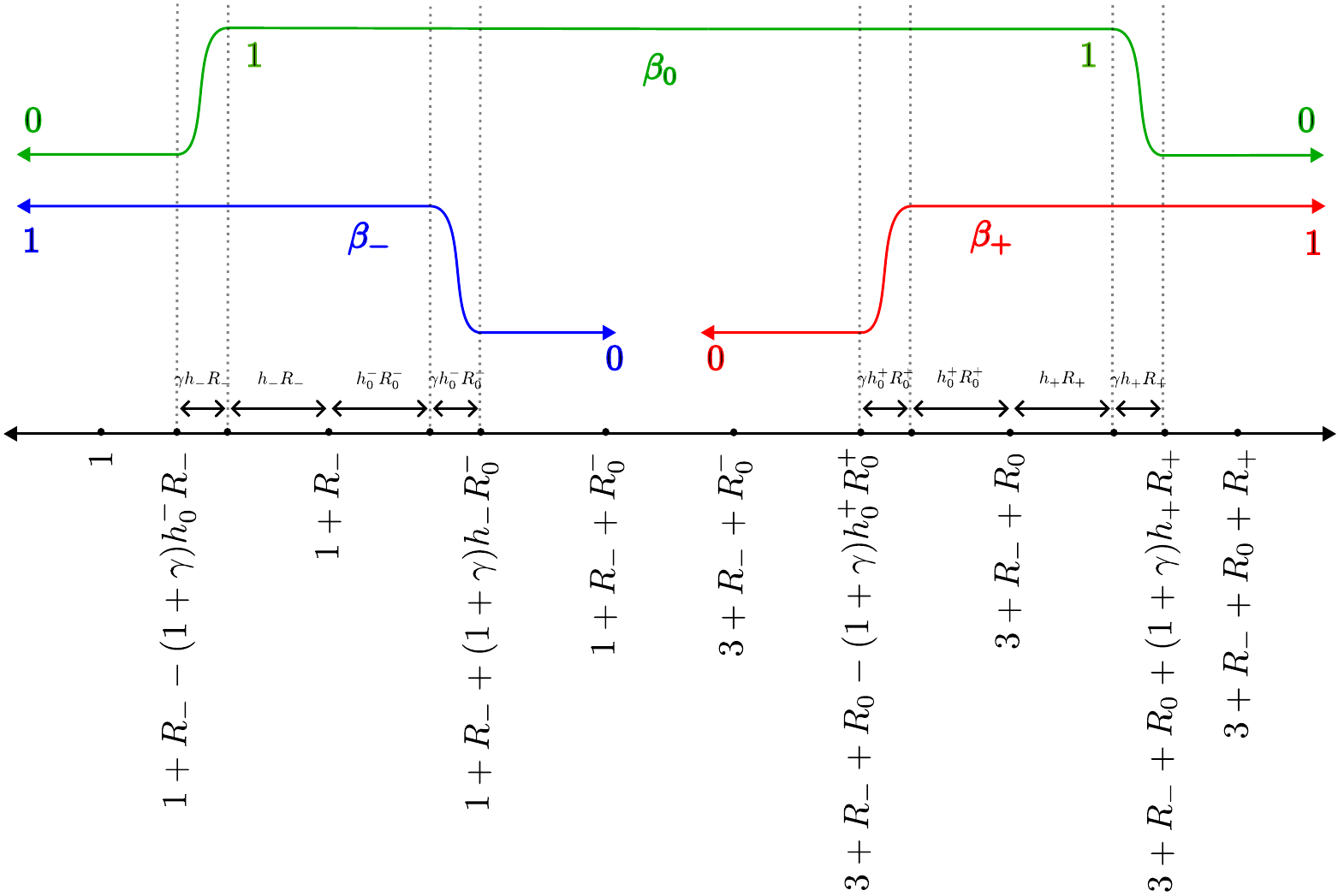}
    \caption{Cutoff functions for gluing parameters $(R_-, R_0^-, R_0^+, R_+)$}
    \label{fig: cutoff functions}
\end{figure}
\begin{defn}\label{defn: cutoff functions}
Fix a smooth function $\beta: \R \to [0,1]$ which is non-decreasing, equal to $0$ on $(-\infty, 0]$, and equal to $1$ on $[1, \infty)$. Fix $0< h< 1$ and $0< \gamma \ll 1$. It will become clear later that we need to pick $h > 1/2$, and that $h$ and $\gamma$ must satisfy conditions depending on the eigenvalues of the Hessians at $x_0$ and $x_1$.
Define three cutoff functions as follows, refer to Figure~\ref{fig: cutoff functions}.
\begin{align}
    \beta_-(s) & = \beta\left(\frac{-s + \left(1 + R_-  + h_0^-(1+\gamma)R_0^- \right)}{\gamma h_0^-R_0^-}\right),\\
    \beta_0(s) & = \begin{cases}
        \beta\left(\frac{s-\left(1+R_- - h_-(1+\gamma)R_-\right)}{\gamma h_- R_-}\right) & s < 2 + R_- + R_0^-,\\
        \beta\left(\frac{-s + \left(3 + R_- + R_0 + h_+(1 + \gamma)R_+\right)}{\gamma h_+ R_+}\right) &s \geq 2 + R_- + R_0^-
    \end{cases} , \text{ and }\\
     \beta_+(s) & = \beta\left( \frac{s - \left(3+ R_- + R_0 -h_0^+(1 + \gamma)R_0^+\right)}{\gamma h_0^+ R_0^+}\right).
\end{align} 
 The main point of note in the definition of these cutoff functions is the supports of $\beta_*$'s and the supports of their $s$-derivatives $\beta'_*$ that are as follows.
\begin{align}
    \supp \beta_- &= ( -\infty, 1 + R_- + h_0^-(1+\gamma)R_0^-], \\
    \supp \beta'_- & = [1 + R_- +h_0^-R_0^-, 1 + R_- + h_0^-(1 + \gamma) R_0^-],\\
    \supp \beta_0 &= [1 + R_- -h_-(1+\gamma)R_-, 3 + R_- +R_0 + h_+(1 + \gamma)R_+],\\ \quad \supp \beta'_0 &= [1 + R_- -h_-(1+\gamma)R_-, 1 + R_- -h_-R_-]\\
    & \quad \cup [3 + R_- +R_0 + h_+ R_+, 3 + R_- +R_0 + h_+(1 + \gamma)R_+] ,\\
    \supp \beta_+ &= [3+R_-+R_0 - (1+\gamma)h_0^+R_0^+, \infty),\\
    \supp \beta'_+ &= [3+R_-+R_0 - (1+\gamma)h_0^+R_0^+, 3+R_-+R_0 - h_0^+R_0^+].
\end{align}
\end{defn}
Next, we define the following translates of $u_0, u_+$. We will not translate $u_-$.
\begin{align}
u_0^{R_- + R_0^-} (s) &:= u_0\left(s  - \left( 2 + R_- + R_0^-\right)\right),\\ \quad u_+^{R_- + R_0 + R_+} &:= u_+\left(s -\left( 4 + R_- + R_0 + R_+\right)\right).
\end{align}
Then, we define the map $u_\# : \R \to M$ by
\begin{align}\label{eqn:definition of pregluing}
u_\# (s) = 
\beta_-(s) u_-(s) + \beta_0(s)u_0^{R_- + R_0^-}(s) + \beta_+(s) u_+^{R_- + R_0 + R_+} (s).
\end{align}
This definition makes sense because outside the intervals 
\begin{align}
\supp \beta_- \cap \supp \beta_0 
\text{ and } 
\supp \beta_0 \cap \supp \beta_+ 
\end{align}
only one out of $\beta_-, \beta_0$, and $\beta_+$ is non-zero. So, the right-hand side of Equation~\ref{eqn:definition of pregluing} is equal to 
\begin{align}
u_- &\quad\text{on}\quad (-\infty, 1 + R_- -h_-(1+\gamma)R_-],\\ 
u_0^{R_- + R_0^-} &\quad\text{on}\quad \left[1 + R_- + h_0^-(1+\gamma)R_0^-, 3+R_-+R_0 - (1+\gamma)h_0^+R_0^+\right], \\  
u_+^{R_- + R_0 + R_+} &\quad\text{on}\quad \left[1 + R_- -h_-(1+\gamma)R_-, \infty\right).\end{align}
On the interval $\supp \beta_- \cap \supp \beta_0$, $\beta_+$ vanishes. Additionally, $\supp \beta_- \cap \supp \beta_0$ gets mapped to $U_-$ by $u_-$ and $u_0^{R_- + R_0^-}$. So, we can add 
$$\beta_-(s) u_-(s) + \beta_0(s)u_0^{R_- + R_0^-}(s) \quad \text{ for } s \in \supp \beta_- \cap \supp \beta_0$$ 
using the identification of $U_-$ to the neighbourhood of $0$ in $\R^n$. Similarly, for $s \in \supp \beta_0 \cap \supp \beta_+$, $\beta_-$ vanishes and we can add 
\begin{align}
\beta_0(s) u_0^{R_- + R_0^-}(s) + \beta_+(s) u_+^{R_- + R_0 + R_+} (s)  \text{ for } s \in \supp \beta_0 \cap \supp \beta_+.
\end{align}

\begin{rem}
    We use the superscript $\tau$ to denote sections over the translated flowlines; refer Section~\ref{sec: shifting by global translation} for a relevant discussion. That is, we denote $\tu_0 := \Ruz$ and $\tu_+ := \Rup$, and similarly for other sections. In general, we use $\tu_*$ to mean an appropriately translated $u_*$ where the translation parameter is clear from the context. We will always have $u_-^\tau =u_-$, but sometimes we put an extra $\tau$ for convenience.
\end{rem}
\subsection{Deforming the pregluing}\label{sec: deformation of pregluing}
In this section, we define deformations of the pregluing $u_\#$. This section is analogous to Section~5.3 in \cite{Hutchings-Taubes_2009}. We will then search for solutions to the gradient flow equations among these deformations. 

To get the deformations of the flowlines, consider three pullback tangent bundles on $\R$, namely, 
\begin{align}
u_-^* (TM), u_0^* (TM), \text{ and } u_+^*(TM).
\end{align}
We can translate these bundles by translating the functions $u_-, u_0$, and $u_+$, and then glue the three together to make a single bundle $E$ on $\R$ over the preglued curve $u_\#$ given by
\begin{align}
u_-^*(TM) &\text{ for } s \in \left(-\infty, 1 + R_- \right],\\
 (u_0^{R_- + R_0^-})^*(TM) &\text{ for } s \in \left[1 + R_- , 3 + R_- + R_0^- + R_0^+ \right],\text{ and } \\ (u_+^{R_- + R_0 + R_+})^*(TM) &\text{ for } s \in \left[3 + R_- + R_0^- + R_0^+ , \infty\right).
\end{align}
This gives us a smooth bundle as, near $1 + R_- $ and $3 + R_- + R_0^- + R_0^+ $, the respective translated flowlines map to Morse neighbourhoods of $x_0$ and $x_1$, where the tangent bundle is identified with $\R^n$. So, the pull-back bundles can be identified in neighbourhoods of $1 + R_- $ and $3 + R_- + R_0^- + R_0^+$ in the domain. 

Now, pick $\psi_-$, $\tpsi_0$, and $\tpsi_+$ be sections of the bundles $u_-^*(TM), (u_0^{R_-+R_0^-})^*(TM)$, and $(u_+^{R_-+R_0 + R_+})^*(TM)$, respectively. The sections $\psi_-$, $\tpsi_0$ and $\tpsi_+$ give deformations of $u_-, \Ruz,$ and $\Rup$. 
Then, we can define a deformation of $u_\#$ by
\begin{align}\label{eqn:deformation of pregluing}
\R &\to M\\
s &\mapsto \exp_{u_\#(s)} (\beta_- \psi_- + \beta_0 \tpsi_0 + \beta_+ \tpsi_+)(s).
\end{align}
Note that we can formally write 
\begin{align}
    \exp_{u_\#} &(\beta_- \psi_- + \beta_0 \tpsi_0 + \beta_+ \tpsi_+) \\
    &= \beta_- \exp_{u_-} \psi_- + \beta_0 \exp_{\left(u_0^{R_- + R_0^-}\right)} \tpsi_0 + \beta_+ \exp_{\left(u_+^{R_- + R_0 + R_+}\right)} \tpsi_+.
\end{align}
Note that the addition in the above formula makes sense in a similar way to the addition in Definition~\ref{eqn:definition of pregluing}.

\subsection{Equation for the deformation to be a gradient flowline}\label{sec: equation for deformation to be a flowline}
Let us temporarily denote the vector field $X : = \nabla f$. Then, the Morse flow equation is given by
\begin{equation}
F = \frac{d}{ds} + X .
\end{equation}
We want to rewrite $F(\exp_{u_\#}\beta_-\psi + \beta_0 \tpsi_0 + \beta_+ \tpsi_+) = 0$ to have the form
\begin{equation}\label{eqn: deformation is gradient flow}
\beta_- \Theta_-(\psi_-, \tpsi_0) + \beta_0\Theta_0(\psi_-, \tpsi_0, \tpsi_+) + \beta_+ \Theta_+(\tpsi_0, \tpsi_+) =0
\end{equation}
for appropriate operators $\Theta_*$'s.

We fix some notation at this stage.
Denote the $s$-derivative of a function or a section $\alpha$ by $\alpha'$. Denote the Sobolev norm by $\| \cdot \|$ and the pointwise norm (of a vector) by $| \cdot |$. 

Let us expand $F(\exp_{u_\#} (\beta_- \psi_- + \beta_0 \tpsi_0 + \beta_+ \tpsi_+))$, for $\|\psi_+\|, \|\tpsi_-\|,$ and $\|\tpsi_0\| < \epsilon$ for a suitably small $\epsilon > 0$. We note we are implicitly using the fact that near each of the critical points we work in charts where the critical point is at the origin, and the metric is Euclidean.
\begin{align}
F&(\exp_{u_\#} (\beta_- \psi_- + \beta_0 \tpsi_0 + \beta_+ \tpsi_+)\\
&=  \beta_- u'_- + \beta_- \psi'_- + \beta'_- u_- + \beta'_- \psi_-  +\beta_0 (\tu_0)' + \beta_0 (\tpsi_0)' + \beta'_0 \tu_0+ \beta'_0 \tpsi_0  \\
& \quad +\beta_+ (\tu_+)' + \beta_+ (\tpsi_+)' + \beta'_+ \tu_+ + \beta'_+ \tpsi_+ \\
& \quad + X(\exp_{u_\#} (\beta_- \psi_- + \beta_0 \tpsi_0 + \beta_+ \tpsi_+)) \\
& = \beta_- u'_- + \beta_- \psi'_- + \beta'_- u_-+ \beta'_- \psi_- + \beta_0 (\tu_0)' + \beta_0 (\tpsi_0)' + \beta'_0 \tu_0+ \beta'_0 \tpsi_0 \\
& \quad + \beta_+ (\tu_+)' + \beta_+ (\tpsi_+)' + \beta'_+ \tu_+ + \beta'_+ \tpsi_+ \\
& \quad + \beta_-X(u_-) +\beta_- \partial_{u_-}X(u_-)(\psi_-) + \beta_- Q_-(\psi_-) \\
& \quad + \beta_0X(\tu_0) +\beta_0 \partial_{\tu_0}X(\tu_0)(\tpsi_0) + \beta_0 Q_0(\tpsi_0) \\
& \quad + \beta_+X(\tu_+) +\beta_+ \partial_{\tu_+}X(\tu_+)(\tpsi_+) + \beta_+ Q_+(\tpsi_+) \\
& = \beta_- \bigg( \psi'_- + \beta'_0 \tu_0 + \beta'_0 \tpsi_0 + \partial_{u_-}X(u_-)(\psi_-)  + Q_-(\psi_-) \bigg)\\
&\quad+ \beta_0 \bigg((\tpsi_0)' + \beta'_- u_- + \beta'_- \psi_- + \beta'_+ \tu_+ + \beta'_+ \tpsi_+ +\partial_{\tu_0}X(\tu_0)(\tpsi_0)  \\
& \quad \quad\quad\quad + Q_0(\tpsi_0) \bigg)\\
& \quad+\beta_+ \bigg( (\tpsi_+)' + \beta'_0 \tu_0 + \beta'_0 \tpsi_0 +\partial_{\tu_+}X(\tu_+)(\tpsi_+)   + Q_+(\tpsi_+) \bigg)
\end{align}
For the second equality, we use the Taylor expansion of $X$ about $u_\#$. The new functions $Q_*(\psi_*)$ depend on $u_*$ and satisfy the bounds
\begin{align}\label{eqn: form of Q}
|Q(\psi_*)| \leq C |\psi_*|^2, \quad \|Q(\psi_*)\| \leq \|\psi_*\|^2 \text{ for } \|\psi_*\| < \epsilon, 
\end{align}
for suitable small $\epsilon > 0$. To be more specific, we can write 
\[
Q(\psi_*) =q(\psi_*)q_2(\psi_*),
\]
where $q$ is a smooth function with uniformly bounded derivatives; $q_2$ is a smooth function with uniformly bounded derivatives that vanishes at $0$ and whose first derivative also vanishes at $0$.
For the last equality, we have used that the $u_*$'s are flowlines and therefore
\begin{align}
    u'_* + X(u_*) = 0.
\end{align}We have also used that
\begin{align}
\beta'_- = \beta_0 \beta'_-,\quad \beta'_0 = (\beta_- + \beta_+)\beta'_0, \quad\beta'_+ = \beta_0 \beta'_+.
\end{align}

Notice that the linearization of the gradient flow operator at $u_*$ given by
\begin{align}\label{eqn: linearization of gradient flowline}
    D_{u_*} \psi_* = \nabla_s\psi_*  + \nabla_{\psi_*} X(u_*),
\end{align}
appears within each of the coefficients of the $\beta_*$'s. So, using notation $D^\tau_* := D_{u^\tau_*}$, we can rewrite $F$ of the deformed pregluing as a ``linear" combination of the following operators.
\begin{align}
\Theta_-(\psi_-, \tpsi_0) & := D_-(\psi_-) + \beta'_0 \tu_0 + \beta'_0 \tpsi_0 + Q_-(\psi_-) , \label{eqn:theta minus}  \\
\Theta^\tau_0(\psi_-, \tpsi_0, \tpsi_+) & := D^\tau_0 \tpsi_0 + \beta'_- u_- + \beta'_- \psi_- \label{eqn:theta 0} \\
&\,\quad\quad\quad\quad + \beta'_+ \tu_+ + \beta'_+ \tpsi_+  + Q_0(\tpsi_0),\\
\Theta^\tau_+(\tpsi_+, \tpsi_0) & := D^\tau_+(\tpsi_+) + \beta'_0 (\tu_0 + \tpsi_0 ) + Q_+(\tpsi_+) .\label{eqn:theta plus}
\end{align}
 We now formulate the above computation as a Lemma.\footnote{See Section 5.4 of \cite{Hutchings-Taubes_2009}.}
\begin{lem}
 There exist functionals $\Theta^\tau_-, \Theta^\tau_0$, and $\Theta^\tau_+$, of the form \ref{eqn:theta minus}, \ref{eqn:theta 0}, and \ref{eqn:theta plus} respectively, such that the map \ref{eqn:deformation of pregluing} is a flowline for $\nabla f$ if and only if equation \ref{eqn: deformation is gradient flow} holds.
 \end{lem}

 Our strategy for solving equation~\ref{eqn: deformation is gradient flow} is to solve the three equations
 \begin{align}
 \Theta^\tau_-(\psi_-, \tpsi_0) = 0, \label{eqn:theta minus zero on interval}\\
 \Theta^\tau_+(\tpsi_+, \tpsi_0) = 0, \label{eqn:theta plus zero on interval}\\
 \Theta^\tau_0(\psi_-, \tpsi_0, \tpsi_+) = 0 ,\label{eqn:theta zero zero on interval}
 \end{align}
 iteratively.

 We carefully choose the spaces where the perturbations $\psi_*$'s can belong, avoiding the redundancy that comes from adding elements of the kernels and ensuring the injectivity of the gluing.
 Let $\cH^\tau_0$ denote the $L^2$-orthogonal complement of $\ker(D^\tau_0)$ in $W^{1,2}(u^{R_- *}_0TM)$, $\cH_-$ denote the orthogonal complement of $\ker(D_-)$ in $W^{1,2}(u_-^*TM)$, and $\cH^\tau_+$ denote the orthogonal complement of $\ker(D^\tau_+)$ in $W^{1,2}(u_+^{R_- + R_+*}TM)$. We will solve the above equations for $\tpsi_\pm \in \cH^\tau_\pm$ and $\tpsi_0 \in \cH_0$. To find solutions to all three Equations~\ref{eqn:theta minus zero on interval}, \ref{eqn:theta plus zero on interval}, and \ref{eqn:theta zero zero on interval}, simultaneously, we first solve Equations~\ref{eqn:theta minus zero on interval} and \ref{eqn:theta plus zero on interval} for a fixed $\psi_0^\tau$ to get $\psi_-$ and $\psi_+^\tau$, respectively, as functions of $\psi_0$. We then plug these results into equation~\ref{eqn:theta zero zero on interval} to view \ref{eqn:theta zero zero on interval} as an equation of $\psi_0^\tau$, and then solve for $\psi_0^\tau$.

 \subsection{Shifting by the global translation}\label{sec: shifting by global translation}
This subsection provides a brief digression to explain how to think about changing the pregluing parameters $(R_-, R_0^-, R_0^+, R_+)$. This will be most relevant for solving the middle equation $\Theta^\tau_0=0$ where we will study the obstruction bundle.

Recall we have chosen $\psi_-^\tau$, $\psi_0^\tau$, and $\psi_+^\tau$ be sections \footnote{We will sometimes abuse notation and write $\psi_-^\tau := \psi_-$.} of the bundles $u_-^*(TM)$, $(u_0^{R_-+R_0^-})^*(TM)$, and $(u_+^{R_-+R_++R_0})^* (TM)$, resp., and written $\Theta^\tau_\pm$ and $\Theta^\tau_0$ as equations for vector fields over the bundles $u_-^*(TM), (u_0^{R_-+R_0^-})^*TM$ and $(u_+^{R_-+R_++R_0})^*TM$. It is sometimes helpful to translate the vector fields $\psi_0^\tau$ and $\psi^\tau_+$ back to be sections of $(u_0)^*(TM)$ and $u_+^*TM$, respectively. We shall refer to the vector fields that we translated back as $\psi_0$ and $\psi_\pm$, respectively\footnote{As a sanity check, we have $\psi_0(s)=\psi_0^\tau(s+2+R_-+R_0^-)$}. Then we can rewrite the equations $\Theta^\tau_\pm$, $\Theta^\tau_0$ as equations over $u_0^*TM, u_\pm^*TM$. Namely, in the coordinates of $u_0^*TM, u_\pm^*TM$, they take the following forms. On the domain of $u_-$, with $s$ denoting the variable that parametrizes $u_-(s)$:
\begin{align}
\Theta_- & = D_-\psi_- +\beta_0' \psi_0(s-(2+R_- +R_0^-)) +\beta_0' u_0^{R_-+R_0^-}+ Q_-(\psi_-),
\end{align}
For the middle portion, if $s\in \mathbb{R}$ denotes the domain variable of $u_0(s)$:
\begin{align}
\Theta_0 &= D_0 \psi_0+ \beta_-'(s+R_0^-+R_-+2)(u_-(s+R_0^-+R_-+2) \\
& +\psi_-(s+R_0^-+R_-+2)) \\
&\quad + \beta_+'(s-(R_0^++R_+))(u_+(s-(R_0^++R^++2))+\psi_+(s-R_0^+-R^++2)) \\
&\quad + Q_0(\psi_0) ,
\end{align}
where the $s$ coordinate is on the domain of $u_0$. Lastly we have
\begin{align}
\Theta_+ &= D_+(\psi_+) + \beta'_0(s+R_0^++R_++2) (u_0(s+R_0^++R_++2) \\&\quad+ \psi_0(s+R_0^++R_++2) ) + Q_+(\psi_+) ,
\end{align}
for $s$ coordinate on the domain of $u_+$. Note that, in the above equations, all the $\beta_*$ have been translated. For brevity of notation, we will write $\beta_*^\tau$ for the translated cut-off functions.

From this viewpoint, when we vary the pregluing by varying the gluing parameters $R_-, R_0^-, R_0^+, R_+$, we are varying how the vector fields over \emph{the unchanging} domains are coupled via a system of PDEs. This will be particularly important when we try to understand how the obstruction section varies with varying pregluing parameters. 

 We note that the vector fields $\psi_*$ as well as the base curves $u_\pm$ have also been translated, depending on the pregluing parameters. For convenience, we may sometimes omit the translations from the vector fields and flowlines when they are not relevant. 
Hence, we will sometimes write the equations above as 
\begin{align} 
\Theta_-&=D_-\psi_- +\beta_0'(\psi_0+u_0) + Q_-(\psi_-)\\
\label{eqn:simplified_form} \Theta_0&=D_0\psi_0+\beta_-^{\tau\prime}(u_++\psi_+)+\beta_+^{\tau \prime}(u_-+\psi_-) + Q_0(\psi_0) \\
\Theta_- &=D_+\psi_- +\beta_0^{\tau \prime}(\psi_0+u_0)+ Q_+(\psi_+)
\end{align}
even though as it appears in the equation $\psi_*,u_*$ have been translated we omit that.
 
 \subsection{Solving for $\psi_-$ and $\psi_+$ in terms of $\psi_0$}\label{sec: solving for psi + and psi -}
 In this section, we do the first step in solving for the $\psi$'s. We fix the pregluing parameters. We fix a $\tpsi_0$ with $\|\tpsi_0\| < \epsilon$ for small enough $\epsilon$. It will become clear what the constraints on $\epsilon$ are. We will now solve for $\psi_\pm^\tau$ as functions of $\psi_0^\tau$.
 
\begin{rem}
For inequalities, we use the symbol ``$\lesssim$" to mean ``less than or equal to up to multiplication by some positive constants." We use the term constant to refer to any quantity that depends only on the fixed flowlines $u_\pm$ and $u_0$. We hope that this will make the exposition clearer.

 Additionally, whenever we say ``for $\epsilon >0$", we mean ``for an $\epsilon > 0$, sufficiently small." Usually, how small $\epsilon$ needs to be is contained in the proofs or computations of inequalities.
 
We will switch between $\psi_*^\tau$ and $\psi_*$, depending on which coordinates are easier. The reader is reminded of the dictionary between the two as explained in Section~\ref{sec: shifting by global translation}.
 \end{rem}
 Let $\cB_{\epsilon,*}$ for $* \in \{+,-, 0\}$ denote the $\epsilon$-ball\footnote{Here $\cH_*$ is simply the untranslated version of $\cH_*^\tau$.} in $\cH_*$ and $\cB_{\epsilon,*}^\tau$ the $\epsilon$-ball in $\cH^\tau_*$.
 \begin{prop}\label{prop: solving psi + and psi - as functions of psi 0}
 \footnote{Proposition~5.6 in \cite{Hutchings-Taubes_2009} is the analogous proposition.}
 For $\epsilon > 0$, and $R_-$, and $ R_+$ large enough, 
 the following hold:
 \begin{enumerate}
 \item Given any $\psi_0^\tau \in \cB_{\epsilon, 0}^\tau$, there exits vector fields $\tpsi_\pm \in \cB_{\epsilon, \pm}^\tau$ depending on $\psi_0^\tau$, such that $\psi_- = \psi_-(\psi_0^\tau)$ solves \ref{eqn:theta minus zero on interval} and $\tpsi_+ = \tpsi_+(\psi_0^\tau)$ solves \ref{eqn:theta zero zero on interval}.
Analogously given $\psi_0\in \mathcal{H}_0$, we get $\psi_\pm(\psi_0)$ solving the equations \ref{eqn:simplified_form}.
 \item We get bounds on the Sobolev norm of $\psi_\pm$
 \begin{align}\label{eqn: improved norm estimates on psi pm}
     \|\tpsi_\pm\| \lesssim R_\pm^{-1}\left(\|\tpsi_0\|_{\supp \beta'_0 \cap \supp \beta_\pm} + \|\tu_0\|_{\supp \beta'_0 \cap \supp \beta_\pm}\right).
 \end{align}
Here, by $\|\tu_0\|_{\supp \beta'_0 \cap \supp \beta_\pm}$, we mean the following: near $\supp \beta'_0 \cap \supp \beta_\pm$ we have chosen coordinate neighbourhoods for which the critical point is at the origin. We take the distance of $\tu_0(s)$ from the origin and measure it with respect to the $W^{1,2}$-norm in this coordinate neighbourhood.
 
 \item The derivative of $\tpsi_\pm$ at a point $\tpsi_0 \in \cB_\epsilon^\tau$ defines a bounded linear functional $\mathcal{D}_\pm^\tau: \cH_0^\tau \to \cH_\pm^\tau$ satisfying
 \begin{align}
 \|\mathcal{D}_\pm^\tau \eta\| \lesssim R_\pm^{-1} \|\eta\|.
 \end{align}
 \item 
 The untranslated solutions $\psi_\pm(\psi_0) \in \cH_\pm$ depend implicitly on the gluing parameters $\{R_-,R^-_0, R_+,R_0^+\}$. If we want to make this dependence explicit, then we should write $\psi_\pm(R_*,\psi_0)$.
 The derivative of $\psi_\pm(R_*,\psi_0)$ with respect to $R_*\in \{R_-,R^-_0, R_+,R_0^+\}$ 
 satisfy
 \[
 \left \| \frac{\partial\psi_\pm}{\partial R_*}\right\|\lesssim \frac{1}{R_\pm}\left(\|\tpsi_0\|_{\supp \beta'_0 \cap \supp \beta_\pm} + \|\tu_0\|_{\supp \beta'_0 \cap \supp \beta_\pm} \right).
 \]
 \end{enumerate}
 \end{prop}
 \begin{proof}
 We prove the Proposition for $\tpsi_-$, a completely identical proof works for $\tpsi_+$.
 \begin{enumerate}
 \item 
  We expand $\Theta^\tau_-$ in \ref{eqn:theta minus zero on interval} as in \ref{eqn:theta minus} to get
 \begin{align}\label{eqn:theta minus on interval 2}
 D^\tau_- \tpsi_- + \beta'_0 (\tu_0 + \tpsi_0) + Q_-(\psi_-) = 0.
 \end{align}
 To solve \ref{eqn:theta minus on interval 2}, we will apply the contraction mapping theorem to an operator $\mathcal{I}_-$ defined as follows. Our assumption that $u_-$ is cut out transversely implies that there exists a bounded right inverse $D_-^{-1} : L^2(u_-^*TM) \to \cH_-$ of $D_-$. Consequently, for fixed $\tpsi_0$ satisfying $\|\tpsi_0\| < \epsilon$ for $\epsilon > 0$ small enough, the assignment
 \begin{align}
 \psi_- \mapsto \mathcal{I}_-(\psi_-) := - D_-^{-1} \left(Q_- (\psi_-) + \beta'_0(\tu_0 + \tpsi_0) \right)
 \end{align}
 defines a continuous map from the $\epsilon$-ball $\cB_\epsilon^-$ in $\cH_-$ to $\cH_-$.
 
 \noindent{\it Claim:} If $\epsilon > 0$ is sufficiently small and $R_-$ and $R_+$ used to define $\beta_0$ are sufficiently large, the map $\mathcal{I}_-$ sends $\cB_\epsilon^-$ to itself as a contraction mapping satisfying
 \begin{align}\label{eqn: contraction inequality of psi minus}
 \|\mathcal{I}_-(\psi^1_-) - \mathcal{I}_-(\psi^2_-)\| \leq \frac{1}{2} \|\psi^1_- - \psi^2_-\| \text{ for } \psi^1_-, \psi^2_- \in \cB_\epsilon^-.
 \end{align}
 {\it Proof of Claim:} Let $R : = \min \{R_-, R_+\}$. The definition of $\beta_0$ implies that there exists a constant $c_1 > 0$ such that $|\beta'_0| < c_1 R^{-1}$. This implies 

 \begin{align}
     \|\beta'_0(u_0 + \psi_0)\| \leq c_1 R^{-1}(\|u_0\|_{\supp\beta_0'\cap \supp \beta_-} + \|\psi_0\|_{\supp\beta_0'\cap \supp \beta_-}).
 \end{align}
 As $Q_-$ has the form \ref{eqn: form of Q}, there exists constant $c_2 > 0$ such that for $\|\psi_-\|, \|\psi_0\| < \epsilon$,  
 \begin{align}
 \|Q_-(\psi_-)\| \leq c_2\|\psi_-\|^2.
  \end{align}
 So, for $R$ large and $\|\psi_-\|, \|\psi_0\|<\epsilon$, $\mathcal{I}_-$ is continuous and maps $\cB_\epsilon^-$ to itself.
 
 
 To see the contraction property, expand $\mathcal{I}_-$ and use the linearity of $D_-^{-1}$ to get,
 \begin{align}
 \frac{\|\mathcal{I}_-(\psi^1_-) - \mathcal{I}_-(\psi^2_-)\|}{\|\psi^1_- - \psi^2_-\|} 
 & =  \frac{\|D^{-1}_-(Q_-(\psi^1_-)  - Q_-(\psi^2_-) \| }{\|\psi^1_- - \psi^2_-\|}\\
 & \leq  \|D^{-1}_-\| \frac{\|Q_-(\psi^1_-) - Q_-(\psi^2_-)\| }{\|\psi^1_- - \psi^2_-\|}.
  \end{align}
 Using the form of $Q$, we can show that the right-hand side is less than
 $$\|D_-^{-1}\|(\|\psi^1_-\| + \|\psi^2_-\|) < \frac{1}{2}, \text{ for } \|\psi^1_-\|, \|\psi^2_-\|<\epsilon.$$
 This concludes the proof of the claim. Part (1) of the proposition now follows from the contraction mapping theorem applied to $\mathcal{I}_-$ restricted to $B_\epsilon^-$.
 \item If $\psi_-$ is a fixed point of $\mathcal{I}_-$ as above, 
 \begin{align}
\|\psi_-\| \leq c\|\psi_-\|^2 + c\|\psi_-\|\|\psi_0\| + cR^{-1} \|u_0 + \psi_0\|_{\supp\beta_0'\cap \supp \beta_-}.
 \end{align}
 If we choose $\epsilon < {c^{-1}}/{4}$ such that $c\epsilon \|\psi_0\| < \epsilon/4$, then inputting $\|\psi_-\| < \epsilon$ into the first two terms of the above inequality and using triangle inequality on the second term gives
 \begin{align}
 \|\psi_-\|  \leq \frac{\|\psi_-\|}{2} + cR^{-1} \|u_0 \| + cR^{-1}\|\psi_0\|,
 \end{align}
 which implies the desired inequality.

 \item To get the bounds on the derivative of $\psi_-$ as a function of $\tpsi_0$, we regard $\mathcal{I}_- = -D_-^{-1} (Q_-(\psi_-) + \beta'_0(u_0 + \tpsi_0) )$ as a function of both $\tpsi_-$ and $\psi_0$. Let the differential of this function with respect to $\psi_-$ and $\tpsi_0$ be denoted by $\mathcal{D}_-^{\tau'}$ and $\mathcal{D}_0^{\tau'}$, respectively. Denote the derivative of $\psi_-$ with respect to $\tpsi_0$ by $\mathcal{D}_-^\tau$. Then differentiating $\psi_- - \mathcal{I}_-$ with respect to $\tpsi_0$, we get
 \begin{align}
  (1 - \mathcal{D}_-^{\tau'})\mathcal{D_-^\tau} = \mathcal{D}_0^{\tau'} \quad \text{ or }\quad 
 \mathcal{D}_-^\tau = (1 - \mathcal{D}_-^{\tau'})^{-1} (\mathcal{D}_0^{\tau'}).
 \end{align}
 By inequality~\ref{eqn: contraction inequality of psi minus}, $\mathcal{D}_-^{\tau'}$ has norm less than $1/2$. Additionally, for a fixed $\psi_-$, the partial derivative $\mathcal{D}_0^{\tau'}$ satisfies 
 $
 \|\mathcal{D}_0^{\tau'} \eta\| \lesssim R^{-1} \|\eta\| $.
 Putting these together completes the proof.
 \item 
 The existence of $\frac{\partial \psi_\pm (R_*,\psi_0)}{\partial R_*}$ as an $W^{1,2}$ vector field follows the same way as the previous step. To estimate its norm,
 we look at the equations $\Theta_\pm, \Theta_0$ as described in \ref{eqn:simplified_form}. We see that for fixed $\psi_0$, we are looking at the fixed point equations, 
 \[
 \psi_\pm = -(D_\pm)^{-1} \circ (\beta_0^{\tau \prime} (u_0+\psi_0) +Q_\pm),
 \]
 where we remind ourselves $\beta_*^\tau$ above has been translated by factors of $R_*$, the same is true for any occurances of $\psi_0$.
 Then we may differentiate both sides w.r.t. to $R_*$ to obtain
 \[
 \left \|\frac{d}{dR_*}\psi_\pm \right \| \leq \frac{C}{R_\pm}\left(\|\tpsi_0\|_{\supp \beta'_0 \cap \supp \beta_\pm} + \|\tu_0\|_{\supp \beta'_0 \cap \supp \beta_\pm} \right).
 \]
 The above follows from the following observations. First, since $u_0$ is translated by factors of $R_*$, when we differentiate, we see the derivatives of $u_0$; however, because of the form of exponential decay of $u_0$ near its ends, the $W^{1,2}$-norm of the derivative of $u_0$ is approximately the same size as that of $u_0$.
 
 Secondly, note that when we differentiate by $R_*$, we pick up an extra derivative of $\psi_0$ because it is translated, so $\psi_\pm'(s)$ is in $L^2$. Furthermore, we are applying $D_\pm^{-1}$, a smoothing operator of order $1$ to the derivative of $\psi_0$, so this ensures $\frac{d\psi_\pm}{dR_*}$ lands back in $W^{1,2}(u_\pm^*TM)$. The norm estimates follow from standard computations as above.

 \end{enumerate}
 \end{proof}
 
 \subsection{Solving for $\tpsi_0$}\label{sec: solving for psi 0}
 Our next step is to solve Equation~\ref{eqn:theta zero zero on interval} for $\tpsi_0$. Our naive hope would be to input the obtained $\tpsi_\pm$ from Proposition~\ref{prop: solving psi + and psi - as functions of psi 0} and then use the same method as we did for solving $\Theta^\tau_\pm = 0$ by constructing a contraction. Unfortunately, this does not work entirely as $D_0^\tau$ is not surjective. The goal of this section is to split the equation into two: first, one equation that is solvable by creating a contraction mapping and finding a fixed point, and second, an ``obstruction" to finding solutions to $\Theta^\tau_0 = 0$. 

 Let $R_+$, and $R_-$ be large as in Proposition \ref{prop: solving psi + and psi - as functions of psi 0}. We want to solve the Equation \ref{eqn:theta zero zero on interval} for $\tpsi_0$ after substituting the values of $\tpsi_\pm$ we obtained in Proposition~\ref{prop: solving psi + and psi - as functions of psi 0}. 
 Let us rewrite \ref{eqn:theta zero zero on interval} as
\begin{align}\label{eqn:solving theta}
     D^\tau_0 \tpsi_0 + F^\tau_0 \tpsi_0 = 0
 \end{align}
 where $F^\tau_0(\tpsi_0)$ denotes the sum of all terms other than $D^\tau_0 \tpsi_0$ on the right hand side of Equation \ref{eqn:theta 0}. Namely, 
\begin{align}\label{eqn:definition of F0}
     F^\tau_0 \tpsi_0 &= \beta'_- u_- + \beta'_- \psi_- + \beta'_+ \tu_+ + \beta'_+ \tpsi_+  + Q_0(\tpsi_0) ,
 \end{align}
 where we use Proposition~\ref{prop: solving psi + and psi - as functions of psi 0} to write $\tpsi_+$ and $\psi_-$ as functions of $\tpsi_0$.

We no longer have a right inverse for $D^\tau_0$, and therefore, cannot solve for $\tpsi_0$ in a manner identical to solving for $\tpsi_\pm$ in Proposition~\ref{prop: solving psi + and psi - as functions of psi 0}. To deal with this, we split $\Theta^\tau_0$ into two parts: its $L^2$ projection onto the image of $D^\tau_0$ and the rest. 
 Using the metric, we introduce a $L^2$-orthogonal projection $\Pi^\tau$ from $L^2((\tu_0)^* TM)$ onto $\ker (D^{\tau *}_0) (\cong \coker (D^\tau_0)$). We hence have a splitting $L^2((\tu_0)^*TM) = \im D^\tau_0 \oplus \coker D^\tau_0$. Then, solving Equation~\ref{eqn:solving theta} is equivalent to simultaenously solving
 \begin{align}
     D^\tau_0 \tpsi_0 + (1-\Pi^\tau) F^\tau_0(\tpsi_0) &= 0 \label{eqn:psi 0 CMT part}, \text{ and }\\
     \Pi^\tau F^\tau_0(\tpsi_0) &=0, \label{eqn:psi 0 cokernel part}
 \end{align}
as $D^\tau_0\tpsi_0$ lies in image of $D^\tau_0$ which is orthogonal to image of $\Pi^\tau$ by definition. 
This analysis holds analogously for the untranslated equation over $u_0$, i.e. with equation $\Theta_0$, linear operator $D_0$ and vector field $\psi_0$. So, we have
\begin{align}
     D_0 \psi_0 + (1-\Pi) F_0(\psi_0) &= 0 \label{eqn:psi 0 CMT part untranslated}, \text{ and }\\
     \Pi F_0(\psi_0) &=0. \label{eqn:psi 0 cokernel part untranslated}
 \end{align}

We first solve the first of these two equations in a manner similar to solving $\Theta^\tau_\pm$ for $\tpsi_\pm$ as $D^\tau_0$ is surjective onto its image.
\begin{prop}\label{prop:solving for theta}
\footnote{This is analogous to Proposition~5.7 in \cite{Hutchings-Taubes_2009}}
    The following are true for $\epsilon>0$ small enough and $R_+, R_-, R_0^\pm$ large enough.
\begin{enumerate}
        \item There exists a unique $\tpsi_0 \in \cB_\epsilon^\tau$, the $\epsilon$-ball in $\cH_0^\tau$, satisfying Equation~\ref{eqn:psi 0 CMT part}.
        
        \item This $\tpsi_0$ satisfies the following bound for $\tpsi_\pm$ obtained as in Proposition~\ref{prop: solving psi + and psi - as functions of psi 0}
\begin{align}\label{ineq: improved norm estimates for psi_0}
        \|\tpsi_0\| & \lesssim (R_0^-)^{-1}(\|u_-\|_{\supp \beta'_- \cap \supp \beta_0} + \|\psi_-\|_{\supp \beta'_-\cap \supp \beta_0})   \\
        & \quad + (R_0^+)^{-1}(\|\tu_+\|_{\supp \beta'_+ \cap \supp \beta_0} + \|\tpsi_+\|_{\supp \beta'_+ \cap \supp \beta_0})  .
        \end{align}
            \item This $\tpsi_0$ defines a smooth section of $(\tu_0)^* TM$. Additionally, $\tpsi_\pm(\tpsi_0)$ from Proposition~\ref{prop: solving psi + and psi - as functions of psi 0} for this $\psi_0$ are also smooth sections of $(\tu_\pm)^* TM$.
            \item The vector field $\psi_0$, which is a translation of $\tpsi_0$ so that it is a vector field over the gradient flowline $u_0$, depends implicitly on the gluing parameters $(R_-,R_0^-,R_0^+, R_+)$. This dependence is smooth. Additionally, $\psi_\pm(R_*,\psi_0)$, which are translates of $\tpsi_\pm(\tpsi_0)$ in (3), also depend smoothly on $R_* \in \{R_-,R_0^-,R_0^+, R_+\}$.
            
            For $R_* \in \{R_-,R_0^-,R_0^+, R_+\}$, we have the norm estimate
            \begin{align}
            \left \| \frac{d\psi_0}{dR_*}\right \| &\lesssim (R_0^-)^{-1}\bigg(\|u_-\|_{\supp \beta'_-\cap \supp \beta_0} + \|\psi_-\|_{\supp \beta'_-\cap \supp \beta_0} \\
            &\quad\quad\quad\quad\quad\quad\quad\quad\quad\quad\quad\quad+ \left \|\frac{\partial \psi_-}{\partial R_*} \right\|_{\supp \beta'_-\cap \supp \beta_0}\bigg) \\
            & + (R_0^+)^{-1}\bigg(\|u_+^\tau\|_{\supp \beta'_+ \cap \supp \beta_0} + \|\tpsi_+\|_{\supp \beta'_+ \cap \supp \beta_0}\\
            & \quad\quad\quad\quad\quad\quad\quad\quad\quad\quad\quad\quad\quad + \left \|\frac{\partial \psi_+^\tau}{\partial R_*} \right\|_{\supp \beta'_+ \cap \supp \beta_0}\bigg).
            \end{align}
    \end{enumerate}
\end{prop}
\begin{proof}

\begin{enumerate}
    \item We apply the contraction mapping theorem to \begin{align}
            \mathcal{I}_0(\psi_0) := -(D^\tau_0)^{-1} (1 - \Pi^\tau) F^\tau_0(\psi^\tau_0),
        \end{align}
        where $(D^\tau_0)^{-1}$ denotes a right inverse of $D^\tau_0: \cH_0^\tau \to \im(D^\tau_0)$.

        It follows from the estimates established in Proposition \ref{prop: solving psi + and psi - as functions of psi 0} that if  $\epsilon > 0$ is sufficiently small and $\|\psi_-\|, \|\psi_+\| < \epsilon$, and $R_\pm$ are sufficiently large, then $\cI_0$ maps $\cB_\epsilon^\tau$  to itself.

        For distinct elements $\psi^{\tau, 1}_0$ and $\psi^{\tau, 2}_0$ of the $\epsilon$-ball of $\cH_0^\tau$, using Proposition~\ref{prop: solving psi + and psi - as functions of psi 0}(c), (d), and assuming $\epsilon$ sufficiently small, 
        \begin{align}
            \frac{\|\cI (\psi^{\tau,1}_0) - \cI(\psi^{\tau,2}_0)\|}{\|\psi^{\tau,1}_0 - \psi^{\tau,2}_0\|} \lesssim C\epsilon + R_+^{-1} + R_-^{-1}.
        \end{align}
        So, $\cI_0$ is a contraction mapping on $\cB_0^\tau$ provided $\epsilon > 0$ is sufficiently small and $R_\pm$ are sufficiently large. Then $\cI_0$ has a unique fixed point in $\cB_\epsilon^\tau$, which by definition will satisfy \ref{eqn:psi 0 CMT part}. This concludes the proof of part (1).
\item Part (2) follows from the above provided $\epsilon$ is sufficiently small and $r$ is sufficiently large.
\item  We show for fixed $\{R_+,R_-,R_0^-,R_0^+\}$, the functions $\tpsi_0, \tpsi_\pm$ are smooth as functions of $s$. It follows from bootstrapping using the specific forms of $\Theta^\tau_+$, $\Theta^\tau_-$, and $\Theta^\tau_0$. 
For example, consider $\Theta^\tau_0$, which gives us the equation
\begin{align}
D^\tau\tpsi_0 = -\left (\beta'_- u_- + \beta'_- \psi_- + \beta'_+ \tu_+ + \beta'_+ \tpsi_+  + Q_0(\tpsi_0)\right ).
\end{align}
The left-hand side takes the form $(\frac{d}{ds} +A(s))\tpsi_0$, the right-hand side has no $s$ derivatives on the vector fields $\tpsi_0,\tpsi_\pm$. This means $\tpsi_0$ is in $W^{2,2}(u_0^{\tau*}TM)$. Looking at the equations $\Theta_\pm=0$ will give us $\psi_\pm \in W^{2,2}(u_\pm^*TM)$. Repeating this process gives us that they are smooth.
\item The bounds in Part (4) follow in the same way as Part (4) of Proposition~\ref{prop: solving psi + and psi - as functions of psi 0}.
We now show that the derivatives of $\psi_0, \psi_\pm$ are smooth with respect to $R_*$ (note that to consider derivatives, we are examining the untranslated vector fields). We look at 
\[
\frac{d\psi_0}{dR_*} =-D_0^{-1} \circ \frac{d}{dR_*} \left (\beta_+^{\tau \prime}(\psi_++u_+)+\beta_-^{\tau \prime}(\psi_-+u_-)+ Q_0(\psi_0) \right )
\]
in our abbreviated notation.
We need to make the dependence of $\psi_\pm(R_*,\psi_0(R_*))$ on $R_*$ on the right-hand side of the equation more precise. After taking the $R_*$-derivative, the right-hand side of the equation consists of the application of $D_0^{-1}$ to
\begin{itemize}
\item $\beta_\pm^{\tau \prime \prime}(u_\pm)$ and $\beta_\pm^{\tau \prime } \frac{d}{dR_*}(u_\pm)$ (note that we've left implicit that $u_\pm$ has also been translated by $R_*$);
\item \[\frac{d\psi_\pm(R_*,\psi_0(R_*))}{dR_*} = \frac{\partial \psi_\pm}{\partial R_*} + \mathcal{D}_\pm \frac{d\psi_0}{dR_*}\];
\item From differentiating $Q(\psi_0)$, terms of the form \[ \psi_0 \frac{d\psi_0}{dR_*}\].
\end{itemize}
All of these are shown to be in $W^{1,2}$ in Proposition~\ref{prop: solving psi + and psi - as functions of psi 0}. Since $D_0^{-1}$ is a smoothing operator, this shows that $\frac{d\psi_0}{dR_*}$ is in $W^{2,2}$. Iterating this process to $R_*$ derivatives of $\psi_\pm$ by looking at the equations $\Theta_\pm$ shows that the $R_*$ derivatives of $\psi_0,\psi_\pm$ are smooth. 

The Sobolev norm bound follows by inspecting the right-hand side.
This concludes the proof.
        
        \end{enumerate}
\end{proof}

\subsection{The obstruction section and the gluing map}\label{sec: obstruction section and gluing map}
 In this section, we define the ``obstruction section," which is essentially the projection of $\Theta_0$ to the cokernel of $D_0$. 
 By inputting the $\psi_\pm$ and $\psi_0$ obtained uniquely in Propositions~\ref{prop: solving psi + and psi - as functions of psi 0} and \ref{prop:solving for theta}, we can view the obstruction section as a function of only the gluing parameters. Then, we show that for large enough gluing parameters satisfying some relations, we will always have a unique solution if (and only if) the signs on the asymptotics are as in Theorem~\ref{thm: gluing}, thus giving us a $1$-dimensional ``parametrization space", namely $\os^{-1}(0)$. We will then define the ``gluing map", namely $G$ in Definition~\ref{defn: gluing map}, on this space. 
The gluing map will give us a parametrization of a $1$-parametric family of ``glued" flowlines limiting to our chosen $(u_-, u_0, u_+)$. 

Note there are redundancies in the pregluing parameters\linebreak $(R_-,R_0^-,R_0^+,R_+)$. To define the obstruction bundle, we eliminate this redundancy. In particular, we eliminate $R_-$ and $R_+$ and keep $R_0^\pm$ as independent variables. To be specific, we choose a sufficiently large integer $A$ (how large it needs to be will be specified in the analysis of the obstruction section) and set
\[
R_\pm = R_0^\pm/A.
\]
From this point onwards, specifying only the parameters $(R_0^-,R_0^+)$ determines a pregluing with $R_\pm$ as above.

Let $r$ always denote a real number greater than the minimum value $R_\pm$ can take as per Propositions~\ref{prop: solving psi + and psi - as functions of psi 0} and \ref{prop:solving for theta}. Let $\ob \to [r, \infty)^2$, referred to as the {\bf obstruction bundle}, denote the trivial bundle where the fiber over any $(R_0^-, R_0^+) \in [r, \infty)^2$ is 
\begin{align}
    \ob_{(R_0^-, R_0^+)} = \hom (\coker(D_{u_0}), \R).
\end{align}

We are now ready to define the obstruction section, a different way of looking at Equation~\ref{eqn:psi 0 cokernel part}, whose zero set will be the space parametrizing the ``gluing." We now begin working with the untranslated bundles $u_0^*(TM)$ and $u_\pm^*(TM)$. Recall that, when we refer to $\psi_*$ obtained from Proposition~\ref{prop: solving psi + and psi - as functions of psi 0} or \ref{prop:solving for theta} without the superscript $\tau$, we are referring to sections of $u_0^*(TM)$ and $u_\pm^*(TM)$ that correspond to $\tpsi_*$'s as described in Section~\ref{sec: shifting by global translation}.

\begin{defn}[Definition 5.9 \cite{Hutchings-Taubes_2009}]\label{defn: obstruction section}
Define a section $\os: [r, \infty)^2 \to \ob$, called the {\bf obstruction section}, as  
    \begin{align}\label{eqn: obstruction section inner product}
        \os(R_0^-, R_0^+)(\sigma) := \langle \sigma, F_0(\psi_0(R_-, R_0^-, R_0^+, R_+))\rangle \text{ for } \sigma \in \coker(D_{u_0}).
    \end{align}
    Here, as earlier, $F_0(\psi_0)$ appears as in Equation~\ref{eqn:psi 0 CMT part untranslated} and 
    $\psi_0(R_-, R_0^-, R_0^+, R_+)$
    is the solution to Equation~\ref{eqn:psi 0 CMT part untranslated} we found in Proposition~\ref{prop:solving for theta}.
\end{defn}
We will use $\os$ to define a parametrizing space for the flowlines limiting to the broken flowline $(u_-, u_0, u_+)$. To do this, we need $\os$ to be a smooth section and to intersect the $0$-section transversely.

\begin{prop}\label{prop:smootheness of obstruction section}
    The section $\os:[r, \infty)^2  \to \ob$ is smooth.
\end{prop}
\begin{proof}

The proof follows from the smoothness of $\psi_*$ with respect to $s \in \R$ and the gluing parameters $(R_-,R_0^-,R_0^+, R_+)$, see Proposition~\ref{prop:solving for theta}, and observing that $\Pi$ does not depend on the pregluing parameters.\footnote{The smoothness property is much more complicated in \cite{Hutchings-Taubes_2009} because their base of the obstruction bundle is much more complicated. See Section 6 of\cite{Hutchings-Taubes_2009} for their case.}
\end{proof}

The following lemma will be a consequence of Section~\ref{sec: C1 estimates}, which contains a detailed analysis of the obstruction section. In Section~\ref{sec: C1 estimates}, we will show that the obstruction section $\os$ is ``$C^1$-close'' to the linearized obstruction section $\os_0$. Clearly, $\os_0$ is ``$C^1$-close to $\os_{00}$'' and $\os_{00}$ is transverse to the zero section\footnote{The term $C^1$-approximation is somewhat of a loaded word in this context, because we are talking about $C^1$-approximating a function whose $C^1$-norm at various points in the domain can be very close to zero, and our required notion of $C^1$ approximation is not uniform in $R_0 = R_0^- + R_0^+$. However, we will be very precise about what it means to be $C^1$ close, and we shall see all of the claimed properties follow as desired.}. Therefore, $\os$ is transverse to the zero section. We note that our approach to showing that $\os$ is transverse to the zero section differs from that taken in \cite{Hutchings-Taubes_2009}.
\begin{lem}\label{lem: transversality of obstruction section}
    The obstruction section $\os: [r, \infty)^2 \to \ob$ is transverse to the zero section. 
\end{lem}

Proposition~\ref{prop:smootheness of obstruction section} and Lemma~\ref{lem: transversality of obstruction section} together show $\os^{-1}(0)$ is a manifold. Therefore, we can define the ``gluing" on $\os^{-1}(0)$ to obtain our candidate parametrization of the space of flowlines limiting to $(u_-, u_0, u_0)$.
\begin{defn}\label{defn: gluing map}
Given $R_-, R_+ > r$ and $R_0^\pm$, define the {\bf$\mathbf{(R_0^-, R_0^+)}$-gluing}, denoted by $u( R_0^-, R_0^+)$, to be the deformed pregluing \ref{eqn:deformation of pregluing}, where $\psi_\pm$ are given by Proposition~\ref{prop: solving psi + and psi - as functions of psi 0} as functions of $\psi_0$ and $\psi_0$ is given by Proposition~\ref{prop:solving for theta} for the gluing parameters $(R_-, R_0^-, R_0^+, R_+)$ defined by setting
    \begin{align}
       R_0 = R_0^- + R_0^+, \quad R_\pm = \frac{R_0^\pm}{A},
    \end{align} for the same large $A \in \Z$ as in Definition~\ref{defn: obstruction section}.
    Define the {\bf gluing map} as \begin{align}\label{eqn:definition of gluing map}
        G: \os^{-1} (0) \to \ms_{x_{-1}, x_{2}}, \quad (R_0^-, R_0^+) \mapsto u(R_0^-, R_0^+).
    \end{align}
\end{defn}
Under the identification $\hom(\coker(D_0), \R) \cong \coker(D_0)$ given by the inner product, we have
\begin{align}
    \os(R_0^-, R_0^+) = \Pi F_0(\psi_0).
\end{align}
Thus, by \ref{eqn:psi 0 cokernel part}, $u(R_0^-, R_0^+)$ is a flowline if and only if $\os(R_0^-, R_0^+) = 0$. 
Our next goal is to show that the gluing map is a parametrization of all curves in $\ms_{x_{-1}, x_2}$ that are ``close to breaking" into $(u_-, u_0, u_+)$, that we state in Theorem~\ref{thm:gluing map is homeo} below. We need to define what it means to be ``close to breaking." \footnote{See Definition 7.1 in \cite{Hutchings-Taubes_2009}.}
\begin{defn}\label{defn:open nbhd around broken u}
Fix $A$ as in Definition~\ref{defn: gluing map}.
    For $\delta > 0$ small enough so that all the required exponentiations are defined, define {\bf space of paths close to } $\mathbf{(u_-, u_0, u_+)}$, $ \mathbf{\widetilde G_\delta(u_+, u_0, u_-)}$, to be the set of paths in $M$ that can be decomposed to $v_- \star v_0 \star v_+$, where $\star$ denotes concatenation, such that there exists  $S_- \in \mathbb{R}$ and $(S_0^+,S_0^-)\in \mathbb{R}_{>0}^2$ such that
    \begin{itemize}
        \item there exists a section $\eta_-$ of the normal bundle of $u_-$ with $\|\eta_-\|_\infty < \delta$ such that 
        $$v_- : \left(-\infty, 1 + \frac{1}{\delta} +S_-\right] \to M$$ is given by $v_-(s+S_-) = \exp_{u_-(s)} \eta_-(s)$;
        \item A section $\eta_0$ of the normal bundle of $u_0$ with $\|\eta_0\|_\infty < \delta$ such that 
        $$v_0 : \left[1 + \frac{1}{\delta} +S_- , 3  +\frac{1}{\delta} + S_0^++S_0^-+S_-\right] \to M$$ 
        is given by $v_0\left(s + \left(2+  S_- +S_0^-+\frac{1}{\delta}\right)\right) = \exp_{u_0} \eta_0(s) $;
        \item There exists a a section $\eta_+$ of the normal bundle of $u_+$ with $\|\eta_+\|_\infty < \delta$ such that 
        $$v_+: \left[3 +  \frac{1}{\delta} +S_0^-+S_0^++S_-, \infty\right) \to M$$ is given by $v_+\left(s + \left(4+ \frac{2}{\delta} + S_0^- + S_0^++S_- \right) \right) = \exp_{u_+(s)} \eta_+(s)$; this conditions is saying we should see a small perturbation of $u_+(s)$ restricted to $s\in(-\frac{1}{\delta},\infty)$ suitably translated in $v_-\star v_0\star v_+$.
        \item The concatenation makes sense, that is, the endpoints match:
        \begin{align} 
        v_-\left(1 + \frac{1}{\delta} +S_-\right) &= v_0\left(1 + \frac{1}{\delta} +S_- \right)\\
        v_0\left(3 +  \frac{1}{\delta} +S_0^-+S_0^++S_-\right) &= v_+\left(3 +  \frac{1}{\delta} +S_0^-+S_0^++S_-\right).
        \end{align}
    \end{itemize}
    Let the {\bf space of flowlines close to } $\mathbf{(u_-, u_0, u_+)}$, $\mathbf{G_\delta(u_-, u_0, u_+)}$, be the set of paths $v \in \widetilde G_\delta(u_-, u_0, u_+)$ such that $v$ is a flowline. Note that, for $\delta >0$ small enough, any element of $G_\delta(u_-, u_0, u_0)$ is in $\ms_{x_{-1}, x_2}$ and has index $2$.
\end{defn}
\begin{defn}    
    Given $\delta > 0$, define the space of paths that are close to breaking into $(u_-, u_0, u_+)$ that we obtain in the image of the gluing map, $\mathbf{U_\delta} \subset [r, \infty)^2$, to be the set of $(R_0^-, R_0^+) \in [r, \infty)^2$ such that $u(R_0^-, R_0^+) \in \widetilde G_\delta(u_+, u_0, u_-)$.
\end{defn}
We are now ready to state the promised parametrization result.
\begin{thm}[Theorem 7.3 in \cite{Hutchings-Taubes_2009}]\label{thm:gluing map is homeo}
     If $r$ is sufficiently large with respect to $\delta > 0$, then
     \begin{enumerate}
        \item[(a)]
        the entire base space 
        $
            [r, \infty)^2  \subset U_\delta,
        $ and
        \item[(b)]the gluing map \ref{eqn:definition of gluing map} restricts to a homeomorphism
        \begin{align}\label{eqn:definition of gluing isomorphism}
            G: \os^{-1}(0) \cap U_\delta \to G_\delta(u_+, u_0, u_-)/\mathbb{R}.
        \end{align}
        Here, the quotient in $\mathbb{R}$ refers to the reparametrization of the domain.
    \end{enumerate}
\end{thm}
\begin{proof}
    Part (a) follows from the norm and derivative estimates in Propositions~\ref{prop:solving for theta} and \ref{prop: solving psi + and psi - as functions of psi 0}.
    We divide the proof of Part (b) into two lemmas. We prove injectivity of $G$ in Lemma~\ref{lem:injectivity of gluing isomorphism} and surjectivity in Lemma~\ref{lem:surjectivity of the gluing isomorphism}. Continuity of $G$ follows from Proposition~\ref{prop: solving psi + and psi - as functions of psi 0}(2) together with Proposition~\ref{prop:smootheness of obstruction section}.
\end{proof}

\subsection{The linearized section $\os_0$}\label{sec: linearized obstruction section}
Now that we have that the gluing map is a homeomorphism between $\os^{-1}(0) \cap U_\delta$ and the flowlines close to $(u_-, u_0, u_+)$, we would like to ``count" the zeroes of the obstruction section so that we can ``count" the number of gluings of a broken flowline. This will conclude the proof of Theorem~\ref{thm: gluing}.
Directly counting the zeroes of $\os$ is difficult. So, in this section, we introduce a ``linearized obstruction section" $\os_0$ whose zeroes are easier to track. Despite the name, this is not strictly the linearization of $\os$ but instead is a $C^1$-approximation of $\os$ and therefore, has the same count of zeroes over the base $[r, \infty)^2$. \footnote{This section is analogous to Section~8.1 of \cite{Hutchings-Taubes_2009}.}

 \begin{defn}
 To define an element of $\ob$, it is enough to give how it pairs with the element $\sigma_0$.
      Let $s$ denote the domain variable of $u_0(s)$, let $\sigma_0$ denote a cokernel element of $D_{0}$. Recall that $\langle-,-\rangle$ is the $L^2$ pairing of the space $W^{1,2}(u_0^*TM)$. Define the {\bf linearized section} $\os_0: [r, \infty)^2 \to \ob$ as follows.
     \begin{align}\label{eqn: defn of s0}
         \os_0(R_0^-, R_0^+)(\sigma_0) &:= \langle \beta_-'(s+R_0^-+R_-+2)(u_-(s+R_0^-+R_-+2), \sigma_0 \rangle \\
        &+ \langle \beta_+'(s-(R_0^+ + R_+))(u_+ (s-(R_0^+ + R^+ + 2)), \sigma_0 \rangle
\end{align}
\end{defn}
\begin{rem}
The shifts in $\beta_\pm$ by factors of $R_*$ is because we have chosen to think about domains of $u_0$ instead of $u_0^\tau$. If we worked in the domain of $u_0^\tau$ we could have similarly defined an obstruction section $\mathfrak{s}^\tau$ and the linearized obstruction section $\mathfrak{s}_0^\tau$. That is, if we let $s$ denote the coordinate of $u_0^\tau$, the said linearized obstruction section would have been
\[
\mathfrak{s}_0^\tau(R_0^-,R_0^+)(\sigma_0^\tau) : = \langle \beta_-'(s) u_-(s), \sigma_0^\tau \rangle + \langle \beta_+'(s)u_+(s),\sigma_0^\tau\rangle.
\]
\end{rem}

     The advantage of $\os_0$ over $\os$ is that its zeroes are easy to compute. Let us elaborate what we mean and conclude the proof of Theorem~\ref{thm: gluing}. 
     
     \begin{proof}[Proof of Theorem~\ref{thm: gluing}]
    \noindent\textbf{Step 1} We first recall the asymptotic expansion of the gradients flowlines $u_\pm$ near the critical points 
    \begin{align}
 u_-& =   e^{-\lambda_0^+ s}a_- + \sum_{v_-} e^{-\lambda_+ s} v_- & s> 1,\\
         u_+ & = e^{-\lambda_1^- s}a_+   + \sum_{v_+} e^{-\lambda_- s} v_+ &s < -1,
     \end{align}
    and similarly in Equation~\ref{eqn: expansion of sigma 0} the asymptotic form of $\sigma_0$
    \begin{align}
    \sigma_0 = \begin{cases}
         e^{\lambda_0^+ s}b_- + \sum_{\lambda_+,v_+} e^{\lambda_+ s} v_+ &s < -1,\\
         e^{\lambda_1^- s}b_+ + \sum_{\lambda_-,v_-} e^{\lambda_- s} v_- & s> 1.
    \end{cases}
\end{align}
    We note in both cases they consist of a largest term, for example $u_-$ this is $e^{-\lambda_0^+s}a_-$, and smaller higher order terms. 
    By taking the pairings with $\sigma_0$ only, the largest terms in the asymptotic expansion, we expand $\os_0$ as
     \begin{align}
         \os_0(R_0^-, R_0^+)(\sigma_0) = -\langle b_-, a_-\rangle e^{-\lambda_0^+ ( R_- + R_0^- )} +  \langle b_+, a_+ \rangle e^{-|\lambda_1^-| (R_0^+ + R_+)} + E,
     \end{align}
     where $E$ is the pairing of the higher order terms of $\sigma_0$ with the higher order terms of $u_\pm$: 
     \begin{align}
         E = \os_0(R_0^-, R_0^+) (\sigma_0) -\left (-\langle b_-, a_-\rangle e^{-\lambda_0^+ ( R_- + R_0^- )} +  \langle b_+, a_+ \rangle e^{-|\lambda_1^-| (R_0^+ + R_+)}\right).
     \end{align}
    \noindent\textbf{Step 2}
    The fact that $\lambda_0^+$ is the smallest positive eigenvalue, and $\lambda_1^-$ is the largest  negative eigenvalue implies that $E$ is ``$C^1$-small" with respect to 
     \begin{align}
         \os_{00}(R_0^-, R_0^+) := - \langle b_-, a_-\rangle e^{-\lambda_0^+ ( R_- + R_0^- )} + \langle b_+, a_+ \rangle e^{-|\lambda_1^-| (R_0^+ + R_+)},
     \end{align}
     for $R_0^\pm$ large enough. We put ``$C^1$-small" in quotations because this is comparing the $C^1$-norms of two sections whose $C^1$-norms are themselves going to zero for large values of $R_0^\pm$. 
    To be more precise, by ``$C^1$-small" we mean, \[E \leq e^{-f_1(R_0^-,R_0^+)}e^{-\lambda_0^+ ( R_- + R_0^- )} + e^{-f_2(R_0^-,R_0^+)} e^{-|\lambda_1^-| (R_0^+ + R_+)}\]
    in $C^1$-norm. Here $f_i$ are functions of bounded derivative that go to $\infty$ as $R_0^+,R_0^- \rightarrow \infty$. To be a bit more precise, there exists $a_i,b_i> 0$ such that 
    \[
    f_i \geq \min \{a_i R_0^+,b_iR_0^-\}.
    \]
    That $E$ satisfies the above inequality follows directly from the sizes of the eigenvalues $\lambda_\pm$.

    We will sometimes write
    $E \ll \mathfrak{s}_{00}$ or $E\ll e^{-\lambda_0^+ ( R_- + R_0^- )} +  e^{-\lambda_1^- (R_0^+ + R_+)} $
    to denote the above for convenience. When we later speak of $C^1$ or even $C^0$ proximity of one term to another, we will always mean it in this sense.
    
    \noindent\textbf{Step 3}.
     We can split the enumeration of zeroes of $\mathfrak{s}_0$ into two situations.
     First suppose if $\langle a_-, b_- \rangle$ and $\langle a_+, b_+ \rangle$ have the opposite signs, then 
        $\os_{00}(R_0^-, R_0^+)$ never vanishes. For $R_0^+,R_0^->r$, with $r$ sufficiently large, $E\ll\mathfrak{s}_{00}$, and so $\mathfrak{s}_0$ also does not have zeroes.

     Suppose, $\langle a_-, b_- \rangle$ and $\langle a_+, b_+ \rangle$ have the same sign. We first observe that 
     \[
     \mathfrak{s}_{00}: [r,\infty)^2 \rightarrow \mathbb{R}
     \]
     has a nonempty zero set. We observe that as $\langle a_-, b_- \rangle$ and $\langle a_+, b_+ \rangle$ have the same sign, the directional derivative of $\mathfrak{s}_{00}$ in the direction $(1,-1)$ is never zero, and takes the form
     \[
     \lambda_0^+\langle  b_-,a_-\rangle e^{-\lambda_0^+(R_-+R_0^-)} +|\lambda_1^-| \langle b_+,a_+\rangle e^{-|\lambda_1^-|(R_0^2+R_+)}.
     \]
     This implies that the zero set of $\os_{00}$ is transversely cut out, and hence, a smooth 1-manifold.
     It follows from ``$C^1$-closeness'' that the $(1,-1)$ direction derivative of $\os_0$ is also never zero, hence the zero set of $\os_0$ is also a smooth 1-manifold. We are using the fact that the derivative of $E$ with respect to $R_0^-$ is exponentially smaller than the derivative of $\mathfrak{s}_{00}$ with respect to $R_0^-$, which is nonzero. So that the zero of $\mathfrak{s}_0$ is both unique and transverse.

     We may better understand the zero set of $\os_{00}$ as follows.
     If we fix 
     \begin{align}
             R_0^- + R_0^+ = R_0
         \end{align}
     and restrict to $R_0^-,R_0^+>r$, then we may view $\mathfrak{s}_{00}$ as a function of $R_0^- \in (r,R_0-r)$, which we write as
     \[
     \mathfrak{s}_{00}(R_0^-,R_0-R_0^-).
     \]
     We see that $\os_{00}(R_0^-,R_0-R_0^-)$ has a unique zero that is transversely cut out.
     It follows from the fact that $E\ll\mathfrak{s}_{00}$ that the same is true for $\mathfrak{s}_0$. Refer Figure~\ref{fig: linearized section}
     
    \textbf{Step 4}
          We have shown that if $\langle a_-, b_- \rangle$ and $\langle a_+, b_+ \rangle$ have the same sign, then for large enough fixed $R_0$, we have a unique zero $(R_0^-,R_0-R_0^-)$ of the linearized obstruction section $\os_0$, and if they have the opposite sign, there are no zeroes. To conclude the proof of Theorem~\ref{thm: gluing}, we need to argue that the same is true for the nonlinear obstruction section $\mathfrak{s}$.    We accomplish this by showing $\os_0$ and $\os$ are ``$C^1$-close'' in the sense specified in Step 2. The definition of ``$C^1$-close'' is exactly so that:
          \begin{itemize}
          \item If $\langle a_-, b_- \rangle$ and $\langle a_+, b_+ \rangle$ have opposite signs, $\os-\os_0$ is so small in $C^0$ norm compared to $\os_0$ such that adding $\os-\os_0$ to $\os_0$ will not introduce any new zeroes;
          \item If $\langle a_-, b_- \rangle$ and $\langle a_+, b_+ \rangle$ have the same signs, the $(1,-1)$ directional derivative of $\os-\os_0$ is so small compared to the $(1,-1)$-directional derivative of $\os_0$ such that the $(-1,1)$ directional derivative of $\os$ always has the same sign as the $(1,-1)$-directional derivative of $\os_0$. 

          \end{itemize}
          The claims about the zeroes of $\os$ immediately follow from the above and the properties of the zeroes of $\os_0$.
          
           Showing $\os$ and $\os_0$ are ``$C^1$-close'' is much more technical, and is discussed in the following two sections (Lemmas~\ref{lem: error term is C0 small} and \ref{lem: linearized obs section is C1 close}). With that, combining with Theorem~\ref{thm:gluing map is homeo}, we complete the proof of Theorem~\ref{thm: gluing}.
      \end{proof}

\begin{rem}
We note if $\langle a_-, b_- \rangle$ and $\langle a_+, b_+ \rangle$ have the same sign, then the $1$-manifold $\mathfrak{s}^{-1}(0)$ is parametrized simply by $R_0$.
\end{rem}

\subsection{$C^0$-estimates}\label{sec: C0 closeness of linearized section}
In this section, we show that the two sections $\os$ and $\os_0$ are $C^0$-close to each other. 
The linearized section $\os_0$ appears as part of the original section $\os$, as follows. By Equation~\ref{eqn: obstruction section inner product}, we can write
 \begin{align}\label{eqn: linearized obstruction section}
     \os(R_-, R_+)(\sigma_0) = 
     \langle \sigma_0, \mathcal{E}_0 + \mathcal{R}(\psi_0)\rangle,
 \end{align}
 where
\begin{align}
   \mathcal{E}_0 &:= \beta_-'(s+R_0^-+R_-+2)(u_-(s+R_0^-+R_-+2)\\
   &\quad + \beta_+'(s-(R_0^++R_+))(u_+(s-(R_0^++R^++2)))
\end{align}
while $\mathcal{R}(\psi_0)$ denotes the sum of all the other terms in \ref{eqn:psi 0 cokernel part untranslated} that enter into $F_0(\psi_0)$ \footnote{We can also define the translated version of these terms as $\mathcal{R}^\tau$.}. Note that $\mathcal{E}_0$ is supported only in $U_-$ and $U_+$. Then the linearized obstruction section is equal to
\begin{align}
    \os_0(R_-, R_+)(\sigma_0) = \langle \sigma_0,  \mathcal{E}_0 \rangle.
\end{align}
Now, showing that $\os$ and $\os_0$ are $C^0$-close reduces to proving the following lemma.
     \begin{lem}\label{lem: error term is C0 small}
     For parameters $h_\pm, h_0^\pm > 1/2$, and 
         $R_0^\pm > A R_\pm$
     for some large enough $A \in \Z$,
     the error term $\mathcal{R}$ satisfies exponential 
     \begin{align}
         \langle \cR, \sigma_0\rangle \ll \os_{00},
     \end{align}
     where $\ll$ means
     \begin{equation}
     |\langle \cR, \sigma_0\rangle|  \leq e^{-f_1(R_0^-,R_0^+)}e^{-\lambda_0^+ ( R_- + R_0^- )} + e^{-f_2(R_0^-,R_0^+)} e^{-|\lambda_1^-| (R_0^+ + R_+)},
     \end{equation}
     in $C^0$-norm. Here, $f_i$ are functions of bounded derivative that go to $\infty$ as $R_0^+,R_0^- \rightarrow \infty$. More precisely, there exists $a_i,b_i> 0$ such that 
    \[
    f_i \geq \min \{a_i R_0^+,b_iR_0^-\}.
    \]
 \end{lem}

 \begin{proof}
 The proof involves careful analysis of the asymptotic behaviour of the flowlines, the special cokernel element $\sigma_0$, and the perturbation sections $\psi_\pm$ and $\psi_0$.
 
 Throughout this proof, we suppress the notation of the chosen gluing parameters $(R_-, R_0^-, R_0^+, R_+)$, with the hope that this leads to greater clarity, rather than the opposite. Also, as norms do not change with a global change of coordinates, we often liberally switch between the untranslated gradient flowlines $u_0$ and $u_\pm$ and the translated gradient flowlines $u_0^\tau$ and $u_\pm^\tau$.
 
By using the asymptotic forms of $u_\pm$ and $u_0$, we have the following upper bounds on the Sobolev norms:
\begin{align}
    \|\tu_0\|_{
    \supp \beta'_0
    } & \lesssim e^{-|\lambda_0^-| (R_0^- + h_-R_-)} + e^{-\lambda_1^+(R_0^+ + h_+ R_+)}, \label{ineq: tail estimates 1 u0}\\
    \|u_-\|_{
    \supp \beta'_-
    } &\lesssim e^{-\lambda_0^+(R_- + h_0^-R_0^-)}, \quad \|\tu_+\|_{
    \supp \beta'_+
    } \lesssim e^{-|\lambda_1^-|(R_+ + h_0^+R_0^+)}.\label{ineq: tail estimates 1}
\end{align}
Similarly, the linearized section has norm with an upper bound,
\begin{align}\label{ineq: norm estimate on linearized section}
    \langle \mathcal{E}_0, \sigma_0 \rangle \lesssim e^{-\lambda_0^+(R_- + R_0^-)} + e^{-|\lambda_1^-|(R_0^+ + R_+)}.
\end{align}
We want to compare $e^{-\lambda_0^+(R_- + R_0^-)} + e^{-|\lambda_1^-|(R_0^+ + R_+)}$ to 
\begin{align}
    \mathcal{R}^\tau(\psi_0) =  \langle \os - \os_0, \sigma_0 \rangle =  \langle\beta'_-\psi_-, \sigma^\tau_0\rangle + \langle \beta'_+\tpsi_+, \sigma^\tau_0\rangle + \langle\mathcal{F}_0^\tau, \sigma_0^\tau \rangle,
\end{align}
where $\mathcal{F}_0^\tau$ denotes all the non-linear (with respect to $\psi_\pm$) terms in $\mathcal{R}^\tau$ and $\sigma_0^\tau$ denotes the translate of $\sigma_0$ that is a section of $(\tu_0)^* TM$.
Let us first estimate the first term $\langle \beta'_-\psi_-^\tau, \sigma_0^\tau\rangle$.
Recall from Proposition~\ref{prop: solving psi + and psi - as functions of psi 0}, we have the norm estimate,
\begin{align}
    \|\tpsi_\pm\| \lesssim \|\tpsi_0\|_{\supp \beta'_0 \cap \supp \beta_\pm} + \|\tu_0\|_{\supp \beta'_0 \cap \supp \beta_\pm}.
\end{align}

Note that even though the Sobolev norm of the section does not change with translation, the translation matters when we restrict the domain over which the norm is taken. We continue to denote translated sections with superscript $\tau$ while remembering that the actual translation depends on the gluing parameters.
As the supports on the right-hand side are restricted, we have additional exponential decay compared to Inequalities~\ref{ineq: tail estimates 1}:
\begin{align}\label{ineq: tail estimates 2}
    \|\tu_0\|_{\supp \beta'_0 \cap \supp \beta_-} &\lesssim e^{-|\lambda_0^-|(R_0^- + h_-R_-)},\\
    \|\tu_0\|_{\supp \beta'_0 \cap \supp \beta_+} & \lesssim e^{-\lambda_1^+(R_0^+ + h_+R_+)}.
\end{align}

However the above estimates for $\|\psi_-^\tau\|$ alone are not enough to bound the term $\langle \sigma_0, \beta_-' \psi_-^{\tau \prime}\rangle$ to the extent we would like. If we input those bounds we would find that $\langle \sigma_0, \beta_-' \psi_-^{\tau \prime}\rangle$ is comparable in size with $\mathfrak{s}_0$. A crucial estimate in \cite{Hutchings-Taubes_2009} is the observation that the part of $\psi_-^\tau$ that contributes to $\mathfrak{s}$ is substantially smaller than the total Sobolev norm of $\psi_-^\tau$. This is done by obtaining further exponential decay estimates for $\psi_-^\tau$ as it approaches to the support of $\beta_-'$ - this is done by observing over certain regions of $u_-$, the equation $\Theta_-=0$ is ``autonomous'' in $\psi_-^\tau$.

\begin{prop} \label{prop:autonomous}
Considered in the domain of $u_-^\tau$ with $s$ as the domain variable, for $s>s_0 = 1+R_--h_0^-R_-$
we have
\[
|\psi_-^\tau\|(s) \leq |\psi_-^\tau |(s_0) e^{-|\lambda_0^+(s-s_0)|}.
\]
\end{prop}
\begin{proof}
For $s>s_0$ the vector field $\psi_-^\tau$ satisfies
\[D_-^\tau \psi_-^\tau =0,
\]
and its exponential decay properties follow from Proposition \ref{Prop:decay_of_kernel}.

\end{proof}

This is still not quite enough to get the estimates on $\langle \sigma_0^\tau, \beta_-'\psi_-\rangle$ that we need\footnote{ In \cite{Hutchings-Taubes_2009} this estimate alone is enough, however, since we are gluing $u_-$ and $u_0$ at a critical point where they decay at different exponential rates, we need to make further improvements. In \cite{Hutchings-Taubes_2009}, this is not required because their analogue of $u_0$ is a branched cover of a trivial cylinder; near the Reeb orbit, it is constant, rather than exponentially decaying to the Reeb orbit.}. Our next step will be improving the overall bounds on the Sobolev norm of $\psi_-^\tau$. For that, we first improve the Sobolev norm of $\psi_0^\tau$ supported near $\beta_-'$.

\begin{prop}The Sobolev norm of $\psi_0^\tau$ satisfies:
\begin{align}\label{ineq: tail estimates 2}
    \|\tpsi_0\|_{\supp \beta'_0 \cap \supp \beta_-} &\lesssim \|\psi_0^\tau\|e^{-|\lambda_0^-|(h_-R_- + h_0^- R_0^-)},\\
     \|\tpsi_0\|_{\supp \beta'_0 \cap \supp \beta_+} &\lesssim \|\psi_0^\tau\|e^{-\lambda_1^+(h_+R_+ + h_0^+ R_0^+)}
\end{align}
\end{prop}
\begin{proof}
The first estimate on $\psi_0$ above are obtained as follows: we observe for $s<1+R_-+h_-R_0^-$ the vector field $\psi^\tau_0$ satisfies the autonomous equation
$D_0^\tau\psi_0^\tau=0$. Proposition~\ref{Prop:decay_of_kernel} gives the required exponential decay when applied to $D_0^\tau \psi_0^\tau = 0$ constrained to the support of $\beta'_0 \cap \supp \beta_-$. The second estimate follows similarly.
\end{proof}
So, we get
\begin{align}\label{eqn: norm estimate 1 psi_-}
    \|\psi_-\| & \lesssim \|\tu_0\|_{\supp \beta'_0 \cap \supp \beta_-} + \|\tpsi_0\|_{\supp \beta'_0 \cap \supp \beta_-}\\
    & \lesssim \|\tu_0\|_{\supp \beta'_0 \cap \supp \beta_-} + \|\tpsi_0\|e^{-|\lambda_0^-|(h_-R_- + h_0^- R_0^-)}\\
    & \lesssim \|\tu_0\|_{\supp \beta'_0 \cap \supp \beta_-} + \bigg(\|\tpsi_+\|_{\supp \beta'_+}\\
    & \quad+ \|\tpsi_-\|_{\supp \beta'_-} + \|\tu_-\|_{\supp \beta'_-} + \|\tu_+\|_{\supp \beta'_+}\bigg)e^{-|\lambda_0^-|(h_-R_- + h_0^- R_0^-)}.
    \end{align}
    This implies,
    \begin{align}
    \|\psi_-\|
    &\lesssim \|\tu_0\|_{\supp \beta'_0 \cap \supp \beta_-} + \bigg(\|\tu_0\|_{\supp \beta'_0 \cap \supp \beta_+} \\
    &\quad+ \|\tu_-\|_{\supp \beta'_-} + \|\tu_+\|_{\supp \beta'_+}\bigg)e^{-|\lambda_0^-|(h_-R_- + h_0^- R_0^-)}\\
    & \lesssim e^{-|\lambda_0^-|(R_0^- + h_-R_-)} + e^{-\lambda_1^+(R_0^+ + h_+R_+)-|\lambda_0^-|(h_-R_- + h_0^- R_0^-)}\\
    & \quad + e^{-(\lambda_0^+ +  |\lambda_0^-| h_-)R_- - (\lambda_0^+  + |\lambda_0^-| )h_0^-R_0^-} \\
    & \quad + e^{-|\lambda_1^-|(R_+ + h_0^+R_0^+) -|\lambda_0^-|(h_-R_- + h_0^- R_0^-)}.
\end{align}
Finally, we estimate
\begin{align}
    \langle \sigma_0^\tau, \beta'_-\psi_-\rangle & \lesssim  \langle \sigma_0^\tau, \psi_-\rangle\\
    &\lesssim \|\psi_-\|e^{-\lambda_0^+(R_0^- + h_-R_-)}\\
    & \lesssim e^{-(|\lambda_0^-| + \lambda_0^+)(R_0^- + h_-R_-)}\\
    & \quad + e^{-\lambda_1^+(R_0^+ + h_+R_+)-(|\lambda_0^-|h_0^- + \lambda_0^+)R_0^- - (|\lambda_0^-| + \lambda_0^+)h_-R_-}\\
    & \quad + e^{-(\lambda_0^+ (1 + h_-)+ |\lambda_0^-|h_-)R_- - (\lambda_0^+ (1 + h_0^-)+ |\lambda_0^-| h_0^-)R_0^-} \\
    & \quad + e^{-|\lambda_1^-|(R_+ + h_0^+R_0^+) -(|\lambda_0^-|h_0^- + \lambda_0^+)R_0^- - (|\lambda_0^-| + \lambda_0^+)h_-R_-}.
\end{align}
The second line used the exponential decay estimates we obtained for $\psi_-$, combined with the exponential decay of $\sigma_0^\tau$.

We now compare the above with the bound \ref{ineq: norm estimate on linearized section} on $\os_0$ term-by-term.
\begin{enumerate}
    \item If $|\lambda_0^-| > \lambda_0^+$, then pick $h_- > 1/2$. Otherwise, pick $|\lambda_0^-| R_0^- > \lambda_0^+ R_-$. In both cases, we get,
    \begin{align}
     e^{-(\lambda_0^+ + |\lambda_0^-|)(R_0^- + h_-R_-)} \ll e^{-\lambda_0^+(R_- + R_0^-)}.
\end{align}
\item If $|\lambda_0^-| > \lambda_0^+$, pick $h_- > 1/2$ and get $(|\lambda_0^-| + \lambda_0^+)h_-R_- > \lambda_0^+ R_-$. Otherwise, take
\begin{align}
    |\lambda_0^-|R_0^- > \lambda_0^+ R_-, \text{ and }\quad h_-, h_0^- > \frac12,
\end{align}
and get
\begin{align}
    \lambda_0^+h_-R_- + |\lambda_0^-| h_- R_0^- > \lambda_0^+ R_-.
\end{align}
In either case, we get
\begin{align}
    e^{-\lambda_1^+(R_0^+ + h_+R_+)-(|\lambda_0^-|h_0^- + \lambda_0^+)R_0^- - (|\lambda_0^-| + \lambda_0^+)h_-R_-} \ll e^{-\lambda_0^+(R_- + R_0^-)}.
\end{align}
\item The third term comes for free.
\begin{align}
    e^{-(\lambda_0^+ (1 + h_-)+ |\lambda_0^-|h_-)R_- - (\lambda_0^+ (1 + h_0^-)h_0^-+ |\lambda_0^-| h_0^-)R_0^-} \ll e^{-\lambda_0^+(R_- + R_0^-)}.
\end{align}
\item The fourth term conditions are the same as the second term.
\end{enumerate}
So we get, for $h_-, h_0^- > 1/2$, and either $|\lambda_0^-| > \lambda_0^+$ or by assuming $|\lambda_0^-| R_0^- > \lambda_0^+ R_-$, we have
\begin{equation}
  \langle \beta'_- \psi_-  , \sigma_0^\tau \rangle \ll \os_{00}.
\end{equation}
Completely analogous arguments give us, for $h_+, h_0^+ > 1/2$, and either $\lambda_1^+ > |\lambda_1^-|$ or by assuming $\lambda_1^+ R_0^+ > |\lambda_1^-| R_+$, we have
\begin{equation}
  \langle \beta'_+ \psi_+^\tau  , \sigma_0^\tau \rangle \ll \os_{00}.
\end{equation}
We are left with the non-linear term whose significant terms are,
\begin{align}
    \mathcal{F}_0^\tau \sim  (\tpsi_0)^2.
\end{align}
Let us first estimate $\|\psi_0\|^2$. Recall from Proposition~\ref{prop:solving for theta}, we have estimates of the form
\begin{align}
    \|\psi_0\| & \lesssim \|\tpsi_+\|_{\supp\beta'_+} + \|\psi_-\|_{\supp\beta'_-} + \|\tu_-\|_{\supp\beta'_-} + \|\tu_+\|_{\supp\beta'_+}\\
    &\lesssim \|\tpsi_+\|e^{-|\lambda_1^-|(h_+ R_+ + h_0^+ R_0^+)} + \|\psi_-\|e^{-\lambda_0^+(h_- R_- + h_0^- R_0^-)}\\
    & \quad + e^{-\lambda_0^+(R_- + h_0^-R_0^-)} +  e^{-|\lambda_1^-|(R_+ + h_0^+R_0^+)}\\
    &\lesssim (\|\psi_0\| + \|\tu_0\|_{\supp \beta'_0 + \supp \beta_+})e^{-|\lambda_1^-|(h_+ R_+ + h_0^+ R_0^+)}\\
    &\quad + (\|\psi_0\| + \|\tu_0\|_{\supp \beta'_0 + \supp \beta_-})e^{-|\lambda_1^-|(h_+ R_+ + h_0^+ R_0^+)} \\
    & \quad + + e^{-\lambda_0^+(R_- + h_0^-R_0^-)} +  e^{-|\lambda_1^-|(R_+ + h_0^+R_0^+)}.
    \end{align}
    The exponential decay estimates for $\psi_+^\tau$ follow analogously to the exponential decay estimates for $\psi_-^\tau$.
    By moving all the $\psi_0$ terms to the left-hand side, we get,
    \begin{align}
    \|\psi_0\| & \lesssim  \|\tu_0\|_{\supp \beta'_0 \cap \supp \beta_+}e^{-|\lambda_1^-|(h_+ R_+ + h_0^+ R_0^+)}\\
    &\quad +  \|\tu_0\|_{\supp \beta'_0 \cap \supp \beta_-}e^{-|\lambda_1^-|(h_+ R_+ + h_0^+ R_0^+)} \\
    & \quad + + e^{-\lambda_0^+(R_- + h_0^-R_0^-)} +  e^{-|\lambda_1^-|(R_+ + h_0^+R_0^+)}\\
    & \lesssim e^{-\lambda_1^+ (R_1^+ + h_+R_+) -|\lambda_1^-|(h_+ R_+ + h_0^+ R_0^+)}\\
    & \quad + e^{-|\lambda_0^-| (R_0^- + h_-R_-) -|\lambda_1^-|(h_+ R_+ + h_0^+ R_0^+)}\\
    & \quad +e^{-\lambda_0^+ (R_- + h_0^-R_0^-)} + e^{-|\lambda_1^-| (R_+ + h_0^+R_0^+)}.
\end{align}
So we again compare
\begin{align}
    \|\psi_0\|^2 & \lesssim e^{-2\lambda_1^+ (R_1^+ + h_+R_+) -2|\lambda_1^-|(h_+ R_+ + h_0^+ R_0^+)}\\
    & \quad + e^{-2|\lambda_0^-| (R_0^- + h_-R_-) -2|\lambda_1^-|(h_+ R_+ + h_0^+ R_0^+)}\\
    & \quad +e^{-2\lambda_0^+ (R_- + h_0^-R_0^-)} + e^{-2|\lambda_1^-| (R_+ + h_0^+R_0^+)},
\end{align}
with the bound on $\langle \mathcal{E}_0, \sigma_0 \rangle$ from Inquality~\ref{ineq: norm estimate on linearized section} and 
see that if $h_0^\pm > 1/2$ and $h_\pm > 1/2$, then
\begin{align}
    \|\psi_0\|^2 \ll \os_{00}.
\end{align}
With all of the above terms,
\begin{align}
    \langle \mathcal{F}_0^\tau, \sigma_0^\tau \rangle \ll \os_{00}.
\end{align}
This concludes the proof of Lemma~\ref{lem: error term is C0 small}.
\footnote{ Without assumption \ref{assumption on setup} for nonlinear terms we would need also 
to estimate terms of the form $\langle \psi_-\tpsi_0, \sigma^\tau_0\rangle_{\supp \beta_-}$. We can proceed by noticing that 
\begin{align}
    \langle \psi_-\tpsi_0, \sigma^\tau_0\rangle_{\supp \beta_-} \lesssim \langle \psi_-, \sigma^\tau_0\rangle_{\supp \beta_-},
\end{align}
and use our previous estimates.

}

Putting these together we get for $h_0^\pm > 1/2$ and $h_\pm > 1/2$, and either $\lambda_1^+ > |\lambda_1^-|$ or by assuming $|\lambda_0^-| R_0^- > \lambda_0^+ R_-$, and $\lambda_1^+ R_0^+ > |\lambda_1^-| R_+$,

 \end{proof}

 \begin{rem}\label{rem:nonlinear_metric}
 We have seen in the above to get the appropriate $C^0$ estimates, we needed exponential decay estimates such as Proposition \ref{prop:autonomous}. The proof of the proposition relied on finding regions in the domain where $D_-\psi_-=0$. The existence of such regions, in turn, is a consequence of Assumption \ref{assumption on setup}. Without this assumption, if we constructed $u_\#$ with any naive pregluing, in the proof of proposition \ref{prop:autonomous} we would instead seen the equation
 \[
 D_-\psi_- +\mathcal{F}(\psi_-,\psi_0^\tau)=E(s)
 \]
for $s\in [1+R_--\gamma h_0^-R_-,1+R_++h_0^-R_0^-]$, where $E(s)$ is a function of $s$ and $\mathcal{F}$ is a quadratic function of its inputs. Since we will have $E(s)$ as a source term, the vector field $\psi_-$ simply will not undergo exponential decay in this region. 

If we don't impose Assumption \ref{assumption on setup}, we still expect to be able to remedy the situation as follows, we first construct the naive preluing $u_\#$, then we perturb it over the region $s\in [1+R_--\gamma h_0^-R_-,1+R_++h_0^-R_0^-]$,  so that it actually becomes a gradient flow segment for the region. We call the perturbed pregluing $\tilde{u}_\#$. Then, we perturb $\tilde{u}_\#$ using vector fields $\psi_*$, and the exponential decay estimates go through as before. The technique for construction $\tilde{u}_\#$ is present in the proof of surjectivity of gluing, in Lemma~\ref{lem: lemma for surjectivity}. In essence, this lemma explains how to construct a finite gradient segment near the critical point (from a segment that almost satisfies the gradient flow equations up to a small error) subject to boundary conditions. Naturally, one needs to be careful about the errors incurred in this process.
 \end{rem}

\subsection{$C^1$-estimates}\label{sec: C1 estimates}
In this section, we show that the obstruction section has the same number of zeros as the linearized obstruction section by showing they are ``$C^1$ close'' to each other.

We recall that we think of $\mathfrak{s}$ taking place in the domain $u_0$, corresponding to the cokernel associated to the equation $\Theta_0$. The obstruction section consists of the $L^2$-pairing of $\sigma$ with the term
\begin{align}
&\beta_-'(s+R_0^-+R_-+2)(u_-(s+R_0^-+R_-+2) +\psi_-(s+R_0^-+R_-+2)) \\
&\quad + \beta_+'(s-(R_0^++R_+))(u_+(s-(R_0^++R^++2))+\psi_+(s-R_0^+-R^++2)) \\
&\quad + Q_0(\psi_0). 
\end{align}

We first recall the setup for taking the derivative of the obstruction section. Recall that, even though we started with pregluing parameters $R_0^\pm$ and $R_\pm$, we set $R_\pm = R_0^\pm/A$ and $R_0^- + R_0^+ = R_0$. We take our independent variables to be $(R_0^-,R_0^+)$. 
 We now explain how to take the derivative of the obstruction section with respect to $R_0^-$, the case for $R_0^+$ is analogous.

The derivative of the linearized obstruction section is directly computable and analyzed in the proof of Theorem \ref{thm: gluing}. The difference 
$\mathfrak{s}-\os_0$ contains many terms that implicitly depend on $R_0^-$; the main terms of concern for us will be how the vector fields $\psi_\pm$ contribute to the nonlinear portion of the obstruction section. We want to show that these contributions are small compared to the terms that show up in the derivative of $\mathfrak{s}_0$.

For most of this section, we examine the $R_0^-$  derivatives of terms $\beta_\pm \psi_\pm$ as they appear in $\os-\os_0$, which are the most difficult to estimate. The same methodology from the previous section applies here as well: we iteratively improve estimates for the $R_0^-$-derivatives of $\psi_\pm$ by identifying regions where various vector fields exhibit exponential decay.

Let us focus on $\psi_-$ for simplicity. Similar considerations will apply to $\psi_+$. Note we already have estimates for the Sobolev norm of $\psi_-$ and its $R_0^-$-derivative from Proposition~\ref{prop: solving psi + and psi - as functions of psi 0}, but we find they are still too large to help us understand the $C^1$ behaviour of $\mathfrak{s}$. As before for the $C^0$-estimates, we will find that the portion of $\psi_-$ and its $R_0^-$-derivative that contributes to $\mathfrak{s}$ is substantially smaller than the norm estimates achieved in Propositions~\ref{prop: solving psi + and psi - as functions of psi 0}. We achieve this by first deriving an exponential decay property of $\frac{d\psi_-}{dR_0^-}$ for $s$ sufficiently large. We then improve the Sobolev norm estimates on $\frac{d\psi_0}{dR_0^-}$ to further improve the Sobolev norm of $\frac{d \psi_0}{dR_0^-}$.

We recall that $\psi_-$ satisfies an equation of the form 
\begin{align}
\Theta_- & = D_-\psi_- +\beta_0' \psi_0(s-(2+R_- -R_0^-)) +\beta_0' u_0^{R_-+R_0^-} + Q_-(\psi_-),
\end{align}
from which we derived norm estimates of the form 
 \[
 \left \| \frac{d\psi_-}{dR_0^-}\right\|\lesssim \frac{1}{R_-}\left(\|\tpsi_0\|_{\supp \beta'_0 \cap \supp \beta_\pm} + \|\tu_0\|_{\supp \beta'_0 \cap \supp \beta_\pm} + \left\|\frac{d\tpsi_0}{dR_0^- }\right\|\right).
 \]
We note this is slightly different from the form in the middle of Proposition~\ref{prop: solving psi + and psi - as functions of psi 0}, since during Proposition~\ref{prop: solving psi + and psi - as functions of psi 0} we have $\psi_-(R_*,\psi_0)$ and only took the partial derivative with respect to the first factor.

We now substantially improve the estimated norm on the part of $ \left \| \frac{d\psi_-}{dR_0^-}\right\|$ that appears in the obstruction section $\mathfrak{s}$. The principle is the same as the improved norm estimates of Section \ref{sec: C0 closeness of linearized section}, where we notice away from the support of $\beta_0'$, the vector field $\frac{d\psi_-}{dR_0^-}$ satisfies a differential equation that forces it to have exponential decay.
\begin{lem}
Let $s$ denote the coordinate in the domain of $u_-$, for $s>s_0 = 1+R_--h_0^-R_-$
\begin{align}
\left|\frac{d}{dR_0^-} \psi_-(s)\right | \leq \left|\frac{d}{dR_0^-}\psi_-\right |(s_0) e^{-|\lambda_0^+(s-s_0)|}.
\end{align}
\end{lem}
\begin{proof}
Due to our assumptions on the Morse function and the metric, away from the support of $\beta_0'$,
the equation
\begin{align}
\Theta_- & = D_-\psi_- +\beta_0' \psi_0(s-(2+R_- -R_0^-)) +\beta_0' u_0^{R_-+R_0^-} + Q_-(\psi_-) = 0,
\end{align}
  reduces to the linear equation $D_-\psi_- =0$. We may differentiate it with respect to $R_0^-$ to obtain 
\[
D_- \frac{d}{dR_0^-} \psi_0=0
\]
from which the exponential decay properties follow.
\end{proof}

In order to get the best bounds on $|d\psi_-(s)/dR_0^- |$ for $s>s_0$, we need an estimate on $|d\psi_-/dR_0^-|(s_0)$. This comes estimating  $\left \|d\psi_-/dR_0^-\right \|$. As we observed, this is upper bounded in part by the Sobolev norm of $d\psi_0/dR_0^-$, constrained to the part where the term $d\psi_-/dR_0^-$ appears in the equation $\Theta_-=0$. Our next step is to improve this term by using additional exponential-decay estimates for $d\psi_0/dR_0^-$.
To this end, examine the section over the middle segment. 

\begin{prop} \label{prop:decay_R_0}
Consider $s$ the variable the domain of $u_0$, for
For $s<s_0 = R_-+h_-R_0^--(2+R_-+R_0^-)$, we have the exponential decay estimates
\[
\left | \frac{d\psi_0}{dR_0^-}(s) \right | \leq \left | \frac{d\psi_0}{dR_0^-}(s_0) \right | e^{-|\lambda_0^- (s-s_0)|}
\]
\end{prop}
\begin{proof}
This region corresponds to the region left of where $\beta_-$ becomes identically equal to $1$, refer to Figure~\ref{fig: cutoff functions}.
For this region, the equation $\Theta_0$ reduces to
\[
D_0\psi_0 = 0.
\]
Differentiating with respect to $R_0^-$ produces the required exponential decay estimates as in Proposition~\ref{Prop:decay_of_kernel}.
\end{proof}

We now have all the ingredients necessary to prove the $C^1$-smallness of the term $\beta_-'(s+R_0^-+R_-+2)\psi_-(s+R_0^-+R_-+2)$ as it appears in $\mathfrak{s}-\os_{0}$. 

\begin{prop}
Consider $\left\langle\sigma, \beta_-^{\tau \prime} \frac{d\psi_-}{dR_0^-} \right\rangle$ that appears in the nonlinear obstruction section. We have
\[
\left\langle\sigma, \beta_-^{\tau \prime}\frac{d\psi_-}{dR_0^-} \right\rangle \ll e^{-\lambda_0^+(R_- + R_0^-)} + e^{-|\lambda_1^-|(R_0^+ + R_+)}.
\]

\end{prop}
\begin{proof}
As in Section \ref{sec: C0 closeness of linearized section}, we begin by combining the estimates 
\begin{align}
            \left \| \frac{d\psi_0}{dR_*}\right \| &\lesssim (R_0^-)^{-1}\left(\|u_-\|_{\supp \beta_-'} + \|\psi_-\|_{\supp \beta_-'} + \left \|\frac{d\psi_-}{dR_*} \right\|_{\supp \beta_-'}\right) \\
            & + (R_0^+)^{-1}\left(\|u^\tau_+\|_{\supp \beta_+'} + \|\tpsi_+\|_{\supp \beta_+'}+\left \|\frac{d\psi_+^\tau}{dR_*} \right\|_{\supp\beta_+'}\right).
\end{align}
and 

 \[
 \left \| \frac{d\psi_-}{dR_0^-}\right\|\lesssim \frac{1}{R_-}\left(\|\tpsi_0\|_{\supp \beta'_0 \cap \supp \beta_-} + \|\tu_0\|_{\supp \beta'_0 \cap \supp \beta_-} + \left\|\frac{d\tpsi_0}{dR_0^- }\right\|_{\supp\beta_0'\cap \supp \beta_-}\right).
 \]
to get
\begin{align}
    \left\|\frac{d\psi_-}{dR_0^-}\right\| & \lesssim   \|\tu_0\|_{\supp \beta'_- \cap \supp \beta_0} + \|\tpsi_0\|_{\supp \beta'_- \cap \supp \beta_0}+\left\|\frac{d\psi_0^\tau}{dR_0^-}\right\|_{\supp \beta'_- \cap \supp \beta_0}\\
    &\lesssim  \|\tu_0\|_{\supp \beta'_- \cap \supp \beta_0} +  \left(\|\tpsi_0\|+\left\|\frac{d\psi_0^\tau}{dR_0^-}\right\|\right)e^{-|\lambda_0^-|(h_-R_- + h_0^- R_0^-)}\\
    &\lesssim \|\tu_0\|_{\supp \beta'_- \cap \supp \beta_0} + \bigg(\|\tu_0\|_{\supp \beta'_+ \cap \supp \beta_0}  \\
    &\quad + \|u_-\|_{\supp \beta_-'} + \|\tu_+\|_{\supp \beta_+'}\bigg)e^{-|\lambda_0^-|(h_-R_- + h_0^- R_0^-)}\\
    & \lesssim e^{-|\lambda_0^-|(R_0^- + h_-R_-)} + e^{-\lambda_1^+(R_0^+ + h_+R_+)-|\lambda_0^-|(h_-R_- + h_0^- R_0^-)}\\
    & \quad + e^{-(\lambda_0^+ +  |\lambda_0^-| h_-)R_- - (\lambda_0^+  + |\lambda_0^-| )h_0^-R_0^-} \\
    & \quad + e^{-|\lambda_1^-|(R_+ + h_0^+R_0^+) -|\lambda_0^-|(h_-R_- + h_0^- R_0^-)}.\\
\end{align}

In the second line above, we used Proposition \ref{prop:decay_R_0}. Next,
in the same way as Lemma~\ref{lem: error term is C0 small}, we combine the exponential decay of $\frac{d\psi_-}{dR_0^-}$ (recall this is a vector field appropriately translated to be viewed in the domain of $u_0$, where we have suppressed the translation as in Equation \ref{eqn:simplified_form}) and the exponential decay of $\sigma$ to obtain:

\begin{align}
    \left\langle \sigma_0, \beta^{\tau \prime}_-\frac{d\psi_-}{dR_-^0}\right\rangle 
    &\lesssim \left\|\frac{d\psi_-}{dR_0^-} \right\|e^{-\lambda_0^+(R_0^- + h_-R_-)}\\
    & \lesssim e^{-(|\lambda_0^-| + \lambda_0^+)(R_0^- + h_-R_-)}\\
    & \quad + e^{-\lambda_1^+(R_0^+ + h_+R_+)-(|\lambda_0^-|h_0^- + \lambda_0^+)R_0^- - (|\lambda_0^-| + \lambda_0^+)h_-R_-}\\
    & \quad + e^{-(\lambda_0^+ (1 + h_-)+ |\lambda_0^-|h_-)R_- - (\lambda_0^+ (1 + h_0^-)h_0^-+ |\lambda_0^-| h_0^-)R_0^-} \\
    & \quad + e^{-|\lambda_1^-|(R_+ + h_0^+R_0^+) -(|\lambda_0^-|h_0^- + \lambda_0^+)R_0^- - (|\lambda_0^-| + \lambda_0^+)h_-R_-}.
\end{align}

Comparing with the exponents of the linearized section and following the recipe in the proof of Lemma~\ref{lem: error term is C0 small},
this concludes the lemma.
\end{proof}

The upshot of the above proposition is that whatever upper bounds we derived for $\psi_-$, they also hold (up to a constant or a factor of $1/R_*$) for the $R_0^-$-derivative of $\psi_-$. We note immediately that an analogous statement holds for estimating the $R_0^-$-derivative of $\psi_+$ as it appears in $\mathfrak{s}$.

An analogous computation to Proposition~\ref{lem: error term is C0 small} gives the following.
\begin{prop}\label{lem: linearized obs section is C1 close}
The nonlinear obstruction section $\mathfrak{s}$ is $C^1$-close to $\mathfrak{s}_0$. By this, we mean that
\[
\left | \frac{d}{dR_0^-}\langle \mathfrak{s}-\mathfrak{s}_0,\sigma\rangle\right | \ll e^{-\lambda_0^+(R_- + R_0^-)} + e^{-|\lambda_1^-|(R_0^+ + R_+)}.
\]
\end{prop}

\begin{proof}
With the terms $\beta_\pm^{\tau \prime} \frac{d\psi_\pm}{dR_0^-}$ taken care of, the rest of the terms are bounded in a similar fashion as in Proposition~\ref{lem: error term is C0 small}: the remaining terms are quadratic in $\psi_0$ and their $R_0^-$-derivatives. We observe after chasing through some inequalities
\begin{align}
\left \| \frac{d\psi_0}{dR_0^-}\right \|^2 & \lesssim e^{-2\lambda_1^+ (R_1^+ + h_+R_+) -2|\lambda_1^-|(h_+ R_+ + h_0^+ R_0^+)}\\
    & \quad + e^{-2|\lambda_0^-| (R_0^- + h_-R_-) -2|\lambda_1^-|(h_+ R_+ + h_0^+ R_0^+)}\\
    & \quad +e^{-2\lambda_0^+ (R_- + h_0^-R_0^-)} + e^{-2|\lambda_1^-| (R_+ + h_0^+R_0^+)}, 
\end{align}
which is the same bound as $\|\psi_0\|^2$ in Proposition~\ref{lem: error term is C0 small}. Hence, we conclude as in Proposition~\ref{lem: error term is C0 small}.
\end{proof}

\subsection{Injectivity and Surjectivity of the Gluing map}
In this section, we provide proofs of injectivity and surjectivity of the gluing map, which go into the proof of Theorem~\ref{thm:gluing map is homeo}.
\begin{lem}[Injectivity of the Gluing map, Section 7.2 of \cite{Hutchings-Taubes_2009}]\label{lem:injectivity of gluing isomorphism}
    If $r$ is sufficiently large and $\delta > 0$ sufficiently small, the restricted gluing map $G$ (\ref{eqn:definition of gluing isomorphism}) is injective.
\end{lem}
\begin{proof}
    We show injectivity by showing that if $r$ is sufficiently large, $\delta > 0$ sufficiently small and $u(R_0^-, R_0^+) \in \tilde G_\delta(u_+, u_0, u_-)$, then $(R_0^-, R_0^+)$ is determined by $u(R_0^-, R_0^+)$. For this, it suffices to prove the following two claims:
    \begin{itemize}
        \item[(i)] If $r$ is sufficiently large and $\delta$ sufficiently small with respect to $r$, then $u(R_0^-, R_0^+) \in \widetilde G_\delta(u_+, u_0, u_+)$ implies $R_0^-, R_0^+ > r$.
        \item[(ii)] For $r$  sufficiently large, if $(R_0^-, R_0^+) \in [r, \infty)^2$ and $u(R_0^-, R_0^+) = u(\tilde R_0^-, \tilde R_0^+)$, then $(R_0^-, R_0^+) = (\tilde R_0^-, \tilde R_0^+)$.
    \end{itemize}
    The proof of (i) more or less follows from the definitions, we have $R_0^\pm\geq C(\frac{1}{\delta} +1)$.

    To see (ii) Choose $p_0$ in the image of $u_0$, and let $B_{\delta_1} (p_0)$ denote a radius $\delta_1$ ball around $p_0$ in $M$. We assume $p_0,\delta$ are chosen $u_0\cap B_\delta(p_0)$ is an interval, which we denote by $B_0$. We further assume that for $\epsilon > 0$ sufficiently small, for any $\psi_0 \in \cB_{\epsilon}$ with $\|\psi_0\|_\infty < \epsilon$ and $\| \nabla \psi_0\|_\infty < \epsilon$, any $\Delta s \in \R$, and any $s, \tilde s \cB_0$,
    \begin{align}\label{eqn: inequaltiy of distance in injectivity proof}
        \mathrm{dist}(\exp_{u_0(s)} (\psi_0 (s)), \exp_{u_0(\tilde s )}  (\psi_0(\tilde s ))) \geq c_0|s-\tilde{s}|
    \end{align}
    for a constant $c_0 > 0$.

     Fix an $r$ such that part (i) is satisfied. Suppose two different pairs of gluing parameters yield the same curve. We let $R_0^-, R_0^+ > r$ and $\tilde{R}_0^-,\tilde{R}_0^->r$ denote the two pairs that produce the same curve. We denote the resulting curve by $u(R_0^-, R_0^+) = u(\tilde{R}_0^-, \tilde{R}_0^+)$. We note these curves are parametrized curves from $\mathbb{R} \rightarrow M$. Let $\tpsi_0$ and $\tilde{\tpsi_0}$ denote sections, respectively from Proposition~\ref{prop:solving for theta} applied to gluing parameters 
    \begin{equation}
    (R_0^-/A, R_0^-,R_0^+, R_0^+/A) \text{ and }(\tilde{R}_0^-/A, \tilde{R}_0^-, \tilde{R}_0^+, \tilde{R}_0^+/A).
    \end{equation} 
    Translate $\tpsi_0$ and $\tilde{\tpsi_0}$ back appropriately to get corresponding $\psi_0$ and $\tilde{\psi}_0$ sections over $u_0$.

    Let $s_0 = u_0^{-1}(p_0)$. Then, as $\exp_{u(s_0)}(\psi_0(s_0))$ is a point on the gluing $u(R_0^-, R_0^+) = u(\tilde R_0^-, \tilde R_0^+)$, then for $\Delta R_0^-  := (R_- + R_0^-) -(\tilde R_- + \tilde R_0^-)$, we have $\tilde{s} = s+ \Delta R_0^- $ with
    \begin{align}\label{eqn: translation equality in injectivity proof}
        \exp_{u_0(s_0)}(\psi_0(s_0)) = \exp_{u_0( s + \Delta R_0^-)}(\tilde\psi_0( s + \Delta R_0^-)).
    \end{align}
     Set $\Delta R_0^+  := (R_0^+ - \tilde R_0^+)$.

    On the other hand, the bounds of the derivatives of $\psi_0$ from Proposition~\ref{prop:solving for theta} imply
    \begin{equation}
    \bigg\|\frac{\partial \psi_0}{ \partial R_0^-}\bigg\| \lesssim e^{-\Lambda r}, \quad \bigg\|\frac{\partial \psi_0}{\partial R_0^+}\bigg\| \lesssim e^{-\Lambda r}
    \end{equation}
    for some $\Lambda>0$. Therefore, 
    \begin{align}
        \mathrm{dist} (\exp_{u_0(s)}(\psi_0(s)), \exp_{u_0(s)}(\tilde \psi_0(s))) \lesssim e^{-\Lambda r}(|\Delta R_0^-| + |\Delta R_0^+|). 
    \end{align}
    Combining the above inequality with \ref{eqn: inequaltiy of distance in injectivity proof} and \ref{eqn: translation equality in injectivity proof}, we get 
    \begin{align}
        |\Delta R_0^-| & \lesssim e^{-\Lambda r}(|\Delta R_0^-| + |\Delta R_0^+|).
    \end{align}
    By a symmetric argument with $p_+ \in \im u_+$ we get
    \begin{align}
       |\Delta R_0^+| & \lesssim e^{- \Lambda r}(|\Delta R_0^-| + |\Delta R_0^+|).
    \end{align}
    This means if $r$ is sufficiently large $\Delta R_0^- = \Delta R_0^+ = 0$, that is,
    $(R_0^-, R_0^+) = (\tilde R_0^-, \tilde R_0^+)$.

\end{proof}

\begin{lem}[Surjectivity of the gluing map, Section 7.3 of \cite{Hutchings-Taubes_2009}]\label{lem:surjectivity of the gluing isomorphism}
    If $r$ is sufficiently large and $\delta > 0$ sufficiently small, the restricted gluing map $G$ (\ref{eqn:definition of gluing isomorphism}) is surjective.
\end{lem}
\begin{proof}
    First, we understand exactly what we need to prove. Let \\$v \in G_\delta(u_-, u_0, u_+)$ and let $v = v_- \star v_0 \star v_+$ be a decomposition as in Definition~\ref{defn:open nbhd around broken u}. We need to show we can find pregluing parameters $(R_0^-,R_0^+)$ and vector fields $(\psi_\pm,\psi_0)$ such that $v$ (up to global reparametrization) equal to the deformation of the pregluing $(u_-,u_0,u_+)$ with pregluing parameters $(R_0^-,R_0^+)$ with the vector fields $(\psi_\pm,\psi_0)$ as given in the gluing construction.

    Given $v\in  G_\delta(u_-, u_0, u_+)$, with standard gluing analysis we can produce pregluing parameters $(R_0^-,R_0^+)$ such that if we let $u_\#$ denote the preglued curve, we can find a vector field $\eta_*\in W^{1,2}(u_\#^*TM)$ with suitably small norm, such that maybe after reparametrizing $v$, we have
    \[
    v(s) = \exp_{u_\#(s)}(\eta_*(s)).
    \]
    With this information, our goal is to slightly adjust the pregluing parameters and find vector fields $(\psi_\pm,\psi_0)$ so that they solve the equations $\Theta_-,\Theta_0,\Theta_+$ and live in the right functional spaces and realize $v$ as being under the image of the gluing map.

    To be precise, let $\beta_*$'s be defined with parameters $(R_-, R_0^-, R_0^+ R_+)$ as in Definition~\ref{defn: cutoff functions}. Outside the intervals 
    \begin{align}
        I_0 &:= [1+R_- - (1+\gamma)h_0^- R_-, 1+R_- + (1+\gamma)h_-R_0^-] \text{ and }\\
        I_1 &:= [3+R_- + R_0^- - (1+\gamma)h_0^+ R_0^+, 3+R_- + R_0 + (1+\gamma) h_+ R_+],
    \end{align}
     where more than one $\beta_*$ is supported, the vector field $\eta_*$ restricted to that region already satisfies Equations~\ref{eqn:theta minus}, \ref{eqn:theta 0}, and $\ref{eqn:theta plus}$. More precisely,
    \begin{align}
        \Theta_-(\eta_*) &= 0 \text{ on } (-\infty, 1 + R_- -(1+\gamma)h_0^-R_-], \\
        \Theta^\tau_0(\eta_*) &= 0 \text{ on } [1+R_- + (1+\gamma)h_-R_0^-, 3+R_- + R_0^- - (1+\gamma)h_0^+ R_0^+],\\
        \Theta^\tau_+(\eta_*) &= 0 \text{ on } [3+R_- + R_0 + (1+\gamma) h_+ R_+, \infty).
    \end{align}
    Note that we have only single inputs for the $\Theta$'s, since only one $\beta_*$ has support on each of the domains, and so only the value of one $\eta_*$ matters. Hence, we define $\psi_\pm^\tau, \psi_0^\tau$ to be equal to $\eta_*$ on the above intervals.
    To show that $v$ is obtained from the gluing construction, we need to extend and modify $\psi_-^\tau$, $\psi_0^\tau$, and $\psi_+^\tau$ on all of $\R$ such that the following properties hold:
    \begin{enumerate}
        \item We call the extended vector fields $\psi_\pm, \psi_0$; \footnote{We shall casually switch between $\psi_*^\tau$ and $\psi_*$ where convenient.}with appropriately chosen pregluing paraemters $(R_0^-,R_0^+)$ the map $v$ is obtained by perturbing the prelguing with the vector fields $(\psi_+,\psi_0,\psi_-
        )$.
        \item Equations~\ref{eqn:theta minus}, \ref{eqn:theta 0}, and \ref{eqn:theta plus} holds for $(\tpsi_+,\tpsi_0,\tpsi_-
        )$ on all of $\R$.
        \item The following sums hold:
        On 
            $\supp \beta_- \cap \supp \beta_0 $,
        \begin{equation}
            \beta_-  \tpsi_- + \beta_0  \tpsi_0  = \eta_* ;
        \end{equation}
            On 
             $\supp \beta_- \cap \supp \beta_0 $,
            \begin{equation}
             \beta_-  \tpsi_- + \beta_0  \tpsi_0  = \eta_*;
             \end{equation}
              On
              $\supp \beta_0 \cap \supp \beta_+ $,
              \begin{equation}
              \beta_0  \tpsi_0 + \beta_+  \tpsi_+  = \eta_*; 
              \end{equation}
             On  $\supp \beta_0 \cap \supp \beta_+ $,
             \begin{equation}
              \beta_0  \tpsi_0 + \beta_+  \tpsi_+ = \eta_*.
              \end{equation}
        \item The extensions have norms $\|\psi_-\|$, $\|\psi_0\|$, and $\|\psi_+\| < \epsilon$, for $\epsilon$ satisfying Propositions~\ref{prop: solving psi + and psi - as functions of psi 0} and \ref{prop:solving for theta}.
        \item The extensions lie in the appropriate spaces, $\psi_- \in \cH_-$, $\psi_0 \in \cH_0$, and $\psi_+ \in \cH_+$.
    \end{enumerate}

    Currently the triple $(\tpsi_+,\tpsi_-,\tpsi_0)$ is only defined on the complement of $I_0$ and $I_1$. We explain step by step how to modify them to satisfy each of $1-5$. After each modification, we will still denote them by $(\tpsi_+,\tpsi_-,\tpsi_0)$  to avoid introducing too many sub/superscripts.

    Let us look for the correct ways to define $\tpsi_-$ and $\tpsi_0$ on \begin{equation}
    \supp \beta_- \cap \supp \beta_0 = \left[1 + R_-  -(1+ \gamma)h_0^-R_-, 1+R_-+(1+\gamma)h_-R_0^-\right].
    \end{equation}
    The extension for $\psi_+$ is analogous. 
    Let $\pi_+$ be the projection on $T_{x_0} M$ to the subspace spanned by all eigenvectors of $\hess_f (x_0)$ that have positive eigenvalues and $\pi_-$ to the subspace of eigenvectors that have negative eigenvalues.

    We note that this makes sense because near the critical point, we have chosen our metric to be Euclidean and the Morse function quadratic, so the Hessian is defined and is non-degenerate at all points in $I_0$. So, $\pi_\pm$ make sense at each point of $I_0$.

    We apply Proposition \ref{prop_semi_infinite_extension}, we take $v_A = \pi_+ \eta_*(1+R_--(1+\gamma)h_0^-R_-)$ and $v_B=\pi_- \eta_*(1+R_-+(1+\gamma)h_-R_0^-)$ to get extensions $\tpsi_-$ and $\tpsi_0$ that satisfy $\Theta^\tau_0=0$ and $\Theta_-^\tau=0$ all the way to $s=-\infty$ and $s=+\infty$ respectively.  We note the constructed solution automatically satisfies (3) by Proposition \ref{lem: lemma for surjectivity}. (4) also follows from uniqueness. Apply this to $\tpsi_0$ and $\tpsi_+$ on $I_1$ gives us the triple $(\tpsi_+,\tpsi_-,\tpsi_0)$ that satisfies (1)-(4).

    Running the above process, we observe for each pregluing parameter $(R_0^{-'},R_0^{+'})$ near the original $(R_0^-,R_0^+)$ we have constructed vector fields $(\psi_+^{\tau'},\psi_0^{\tau'},\psi_-^{\tau'})$ that satisfy (1)-(4). For part (5), we vary the pregluing parameters $(R_0^-,R_0^+)$ (recall these are the actual independent coordinates on the base of the obstruction bundle).

    To be more precise, we need to ensure the vector fields $\psi_\pm,\psi_0$ associated to the pregluing parameters $(R_-^-,R_0^+)$ satisfying properties (1)-(4) are orthogonal to the kernel of $D_\pm, D_0$, respectively. The kernel of $D_*$ is spanned by the vector field that generates reparametrization in the $s$ direction. Let $w_*\in \ker D_*$ denote such vector field. Then, $\psi_*$ is in $\ker D_*^\perp$ if and only if
    \[
    \langle \psi_*,w_*\rangle =0.
    \]

    We next observe that when we change the pregluing parameter $R_0^-$, we are (up to small controlled errors) adding a multiple of $w_0$ to $\eta_0$. Similarly, when we are changing $R_0^+$, we are changing (up to small controlled errors) $\psi_+$ by multiples of $w_+$. Finally, we can add multiples of $w_-$ to $\eta_-$ by globally translating $v$ in the $s$ direction. After doing this carefully (see Step 3 of proof of Lemma~7.5 in \cite{Hutchings-Taubes_2009}), we can find a unique $(R_0^+, R_0^-)$ so that the resulting $\psi_\pm,\psi_0$ satisfy 1-5.
     \end{proof}

\begin{lem}\label{lem: lemma for surjectivity}
\footnote{This is analogous to Lemma~7.6 in \cite{Hutchings-Taubes_2009}}
    Recall that the Morse flow equation is given by
\begin{equation}
F = \frac{d}{ds} + X = 0.
\end{equation}
Let $u_\#$ denote the pregluing given by the pregluing parameters $(R_0^-,R_0^+)$.
Take $A = 1 + R_-  -(1+ \gamma)h_0^-R_-$ and $B =1+R_-+(1+\gamma)h_- R_0^-$.
    There exists $\epsilon_0 > 0$ such that for $\epsilon < \epsilon_0$, and 
    \begin{align}
    v_A \in \pi_+ T_{u(A)} M 
    \text{ and }
    v_B \in \pi_- T_{u(B)} M
    \end{align}
    with $|v_A|, |v_B| < \epsilon$, there exists a unique solution $\eta$ to the equation $F(\exp_{u_\#}\eta) = 0$ on $[A,B]$ satisfying the boundary conditions $\pi_+\eta(A)=v_A$ and $\pi_-\eta(B)=v_B$.

    \end{lem}
    \begin{proof}[Proof of Lemma:]
    Denote by $W^{2,2}[A,B]$ the Sobolev completion of $u_\#^*TM$ restricted to the domain $s\in [A,B]$.
    Define the map 
        \begin{align}
            \mathcal{F}: W^{2,2} [A, B] & \to \pi_+ T_{u_\#(A)} M \times \pi_-T_{u_\#(B)} M  \times W^{1,2} [A,B]\\
            \eta & \mapsto (\pi_+\eta(A), \pi_- \eta(B), F(exp_{u_\#}(\eta)).
        \end{align}
        We show that $\mathcal{F}$ is an isomorphism when restricted to a sufficiently small ball of its domain.

        We note that the operator $F$ on $W^{2,2}[A,B]$ is a linear operator $F=\frac{d}{ds} +A$ where $A$ is the Hessian of $f$ at the critical point. This means we can solve this problem using Fourier series expansions. 
        
        Let us show $\mathcal{F}$ is injective. Consider $\eta \in W^{1,2}[A,B]$ $F(\eta)=0$. Then, we can write $\eta = \sum_{i=1}^n a_i e^{-\lambda_i s} v_i$, where $v_i$ are eigenvectors of $A$ with eigenvalues $\lambda_i$. The constants $a_i$ are all equal to $0$ because $\pi_+\eta (A) =\pi_-\eta(B) =0$.

        To show surjectivity of $\mathcal{F}$, suppose $\xi \in W^{1,2} [A,B]$, then we can write $\xi = \sum_{i=1}^n c_i(s) v_i$. If we set $\eta =\sum a_i(s) v_i$, we can solve for $a_i$ satisfying the ODE
        \[
        a_i'(s) + \lambda_i a_i =c_i(s).
        \]
        This ensures the condition $F(\eta)=\xi$. The condition $(\pi_+\eta(A),\pi_-\eta(B))=(v_A,v_B)$ is ensured by adding a multiple of $\sum_{i=1}^n a_i e^{-\lambda_i s} v_i$. Doing this carefully also shows that the norm of the inverse of $F$ is bounded above by a constant independent of the pregluing parameters $R_0^\pm$.
    \end{proof}
    
    Using similar ideas as above, we prove the following proposition.
    \begin{prop}\label{prop_semi_infinite_extension}\footnote{This proposition is analogous to Lemma~7.7 in \cite{Hutchings-Taubes_2009}}
    Take $A = 1 + R_-  -(1+ \gamma)h_0^-R_-$ and $B =1+R_-+(1+\gamma)h_- R_0^-$. Given  
    \begin{align}
    v_A \in \pi_+ T_{u(A)} M 
    \text{ and }
    v_B \in \pi_- T_{u(B)} M
    \end{align}
    with $|v_A|, |v_B| < \epsilon$, there exists unique $\tpsi_0 \in W^{2,2}(u_0^{\tau *}TM)$ restricted to $s<B$ and $\psi_-\in W^{2,2}(u_-^{\tau *}TM)$ restricted to $s>A$, both with norm less than $C\epsilon$ so that
    \[
    \tpsi_-(A) = v_A, \quad \tpsi_0(B) = v_B.
    \]
    For $s<B$ we have
    \[
    \Theta^\tau_0(\tpsi_-,\tpsi_0)=0
    \]
    and for $s>A$ we have
    \[
    \Theta^\tau_-(\psi_-,\tpsi_0)=0
    \]
    \end{prop}
    \begin{proof}[Sketch of proof]
    The idea of the proof is to define 
    \begin{align}
    \mathcal{F}:&W^{2,2}(u_-|_{[A,\infty)}^*TM)\times W^{2,2}(u_0^\tau|_{(-\infty,B]}^*TM)\\ 
    &\rightarrow \pi_+ T_{u_\#(A)} M \times \pi_-T_{u_\#(B)} M\\
    &\quad\quad\quad\times  W^{1,2}(u_-|_{[A,\infty)}^*TM)\times W^{1,2}(u_0^\tau|_{(-\infty,B]}^*TM) 
    \end{align}
    by
    \[\mathcal{F}(\psi_-,\tpsi_0) = (\pi_+\psi_-(A),\pi_-\tpsi_0(B), \Theta_-, \Theta_0^\tau)\] 
    and show this $\mathcal{F}$ is an isomorphism by a Fourier series argument as in the proof of Lemma~\ref{lem: lemma for surjectivity}.
    \end{proof}


\section{Obstruction Bundle Gluing with perturbation}

In this section, we use the same techniques as before to examine the case of Morse but not Smale gradient vector fields, and what can happen to broken flowlines after perturbing the metric in a 1-parameter family. In particular, we examine (under certain assumptions) the glue-ability of $2$-component flowlines over a $1$-parameter family of metrics. We refer to this gluing informally as {\bf ``$t$-gluing"}. Here, $t$ refers to the perturbation. Our main purpose is to give an expository account of how the technology can be implemented, rather than repeating detailed proofs that are all of the same flavour as those we previously worked out. Hence, we will state the setup and the relevant theorems precisely, but will not go into the proofs in detail.

We restrict ourselves to particular one-parameter perturbations $\{g_t\}_{t \in (0,\epsilon)}$ of the metric $g$ that are
 defined as follows. We borrow this construction from \cite[Theorem 2.2.5 (Smale Theorem)]{Audin-Damian_2014}. Assume, for simplicity, that on the entire manifold $M$ there is only one (unparametrized) flowline $u_0$ with a non-trivial cokernel for the pair $(f,g)$. We assume the cokernel is 1-dimensional. The more general case would be considering bifurcations of broken flowlines with multiple non-transverse components.

Let $y = u_0(0) \in M$. Recall from Equation~\ref{eqn: identification of coker with perp}, we can identify 
\begin{align}\label{eqn: identification of coker with perp t-gluing}
    \coker D_{u_0} \cong \R \langle v \rangle,
\end{align}
for some $0 \neq v \in (T_y W^u + T_y W^s)^\perp$. We let $\sigma$ denote the element in the cokernel that corresponds to $v\in T_{u(0)}M$.

Perturb $g$ to $g_t$ in a neighbourhood of $u_0(0)$ away from all the critical points and index $1$ flowlines. We choose the perturbation so that
\begin{equation} \label{eqn_perturbation_of_f}
\nabla_{g_t} f(u_0(s)) = \nabla_{g} f + tV +O(t^2)
\end{equation}
such that $\int_{u_0(s)}\langle \sigma, V\rangle ds >0$.

Then, it can be checked that for all $t \in (0, \epsilon)$ the pairs $(f, g_t)$ are Morse-Smale (in particular, the flowline $u_0$ disappears for $t\neq 0$). Let $\ms(x_{-1}, x_1; g_t)$ denote the set of flowlines for metric the $g_t$, namely, $u: \R \to M$ satisfying
\begin{align}
    \frac{d u}{ds} = -\nabla_{g_t} f \circ u
\end{align}
with $u(\pm \infty) = x_{\pm 1}$.
Figure~\ref{fig:t-gluing} shows the kind of bifurcation for gradient flowlines that can happen for $t\neq 0$. It is precisely this kind of phenomenon that we wish to describe using obstruction bundle gluing techniques.
\begin{figure}[h]
    \centering
    \includegraphics[width=\linewidth]{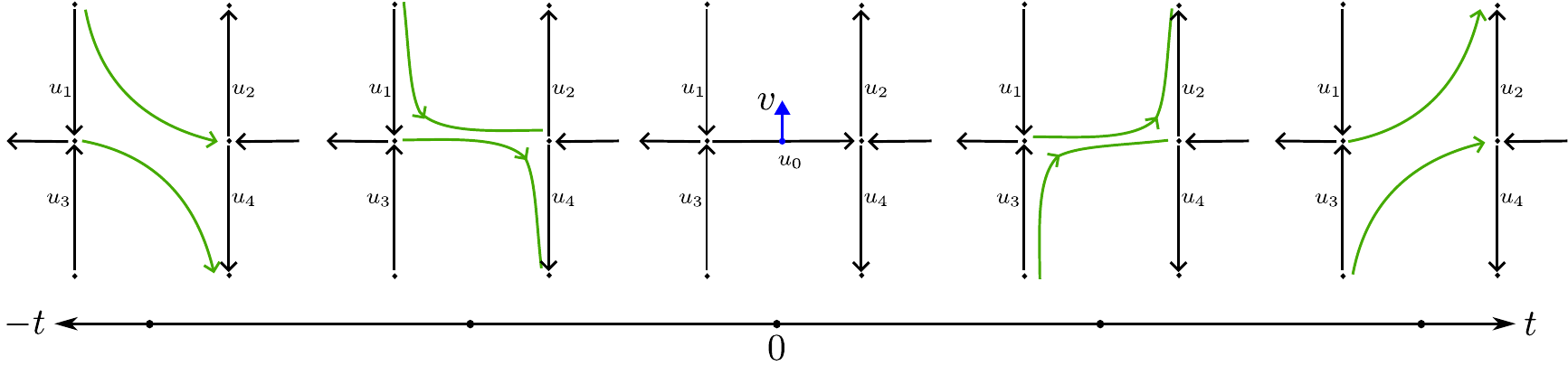}
    \caption{Different pairs of flowlines are gluable for $t>0$ ($(u_0, u_2)$ and $(u_3, u_0)$) and for $t<0$ ($(u_1, u_0)$ and $(u_0, u_4)$) }
    \label{fig:t-gluing}
\end{figure}
We again work the Assumptions~\ref{assumption on setup} on the form of the metric $g_t$ near the critical points to simplify our analysis.
\begin{thm}\label{thm: t-gluing}
    For $(f,g)$ a pair of a Morse function and a metric satisfying Assumptions~\ref{assumption on setup}, consider the perturbation $(f, g_t)$ given as above. For $x_{-1}, x_0, x_1 \in \crit (f)$ with
    \begin{align}
        \ind(x_{-1}) = k + 1, \ind(x_0) = \ind(x_1) = k,
    \end{align}
    let
    \begin{align}
        (u_-,u_0) \in \ms(x_{-1}, x_0) \times \ms(x_0, x_1).
    \end{align}
    
    Let $\lambda_0^+$ be the smallest positive eigenvalue of $\hess_{x_0} f$ and $\lambda_0^-$ the largest (least negative) negative eigenvalue of $\hess_{x_0} f$.
 Denote the cokernel element corresponding to $v$ under the identification \ref{eqn: identification of coker with perp} by $\sigma_0$. 
Assume there exists a nonzero $b_- \in T_{x_0} M$ and a constant $c > 0$ such that
\begin{align}
 \sigma_0 = 
         e^{\lambda_0^+ s}b_-  + \sum_{v_+} e^{\lambda_+ s} v_+ \text{ for } s< -1
\end{align}
Here, $v_+$ are eigenvectors of the Hessian with eigenvalue $\lambda_+$. By assumption we have $|\lambda_+|>|\lambda_0^+|$ for every $\lambda_+$ that appears in the sum.
Similarly,  near the critical point, the gradient flowline $u_-$ can be written as 

\begin{align}
 u_-& =  e^{-\lambda_0^+ s}a_-  + \sum_{v_-} e^{-\lambda_+ s} v_- & s> 1
     \end{align}
for a vector $a_- \in T_{x_0}M$ which is an eigenvector of the Hessian. Assume that 
\begin{align}
    \langle a_-, b_- \rangle \neq 0, \quad 
\end{align}
Then, if 
\begin{equation}
    \langle a_-, b_- \rangle > 0 \quad (\text{resp., } \langle a_-, b_- \rangle < 0),
\end{equation}
    there exists a unique one-parametric family 
    \begin{align}
        u_t \in \ms(x_{-1}, x_1; g_t) \text{ for } t > 0 \quad (\text{resp., } t < 0).
    \end{align}
    that degenerates into the broken gradient flowline $(u_-,u_0)$ at $t=0$.
    Conversely if $\langle a_-, b_- \rangle > 0$ (resp., $ \langle a_-, b_- \rangle < 0$), no 1-parameter family degenerates to $(u_-,u_0)$ from $t<0$ (resp. $t<0$).

    An analogous statement holds for $(u_+, u_0) \in \ms(x_{-1}, x_0) \times \ms(x_0, x_1) $ with
    \begin{align}
        \ind(x_{-1}) = \ind(x_0) = k + 1,  \ind(x_1) = k.
    \end{align}
    
\end{thm}

As in the $0$-gluing case, we first discuss an Example that we recommend the reader keep in mind throughout the proof.
\begin{ex}\label{ex: t-gluing on torus}
\begin{figure}[h]
    \centering
    \includegraphics[scale = 0.5]{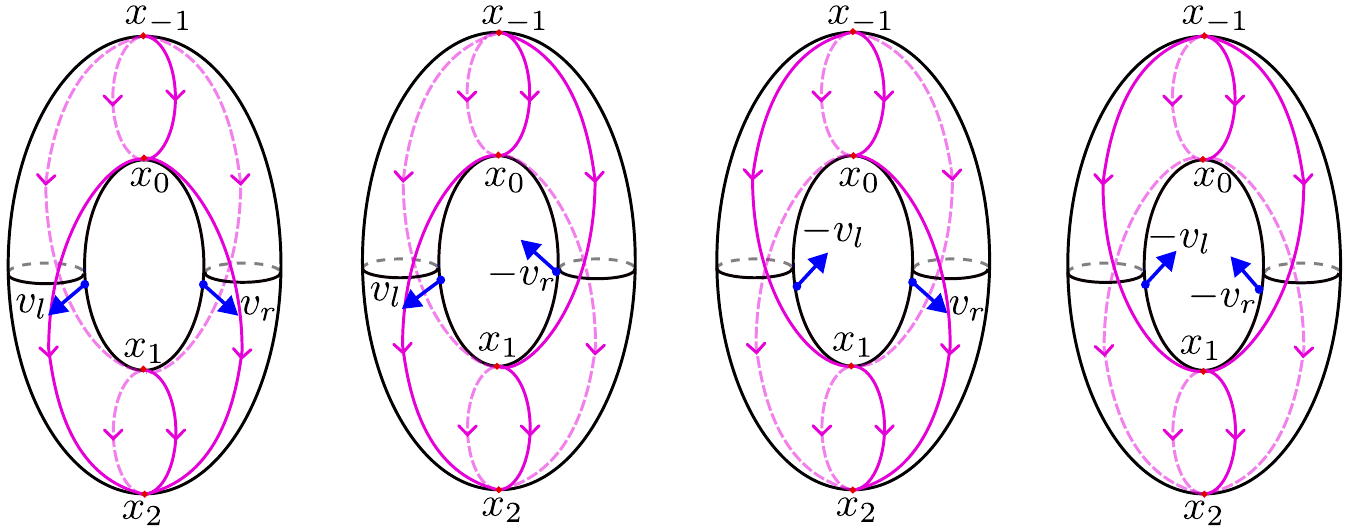}
    \caption{Example of $t$-gluing on the torus. Different choices of the vector as the generator of the cokernel give a different combination of glueable flowlines. Let $V^r$ and $V^l$ denote the first-order perturbation of the gradient flow equations \ref{eqn_perturbation_of_f} over $u_0^l$ and $u_0^r$, respectively. The bifurcations above are all for $t<0$.
    The choices of perturbations from left to right are given by (1)
    $\langle V^l,\sigma_0^l\rangle>0, \langle V^r,\sigma_0^r\rangle>0$; (2) $\langle V^l,\sigma_0^l\rangle>0, \langle V^r,\sigma_0^r\rangle<0$; (3) $\langle V^l,\sigma_0^l\rangle<0, \langle V^r,\sigma_0^r\rangle>0$; (4)$l\langle V^l,\sigma_0^l\rangle<0, \langle V^r,\sigma_0^r\rangle<0$.}
    \label{fig: different t-gluings torus example}
\end{figure}
   Consider the upright torus with Morse function given by the height function. Just as in Example~\ref{ex: 0-gluing on upright torus}, the flowlines $u_0^l$ and $u_0^r$ have $1$-dimensional cokernels. Denote the vector $v^l : = (1,0,0) \in T_{p^l}T^2$ and $v^r : = (1,0,0) \in T_{p^r}T^2$. We take the cokernels $\sigma_0^l$ and $\sigma_0^r$ of $u_0^l$ and $u_0^r$ respectively to be given by the vectors $v_l$ and $v_r$ respectively.  We can perturb the metric over $u_0^l$ and $u_0^r$ independently, and different choices give different gluable pairs as illustrated in Figure~\ref{fig: different t-gluings torus example}.

   With a fixed choice of perturbation of the metric around $u_0^l$ and $u_0^r$, we can define the Morse complex even without the Smale condition. The generators of the complexes remain critical points, graded by their Morse indices. The differential now counts broken flowlines of total index $1$ (there can be an index $0$ flowline as a component of the broken flowline) that is ``$t$-gluable" with the choices we have made. Theorem~\ref{thm: t-gluing} implies that the complex is the same as the Morse complex for a choice of Morse-Smale pair $(f,g)$. Hence, this definition recovers the usual Morse complex.
\end{ex}

\begin{figure}[h]
    \centering
    \includegraphics[scale=0.5]{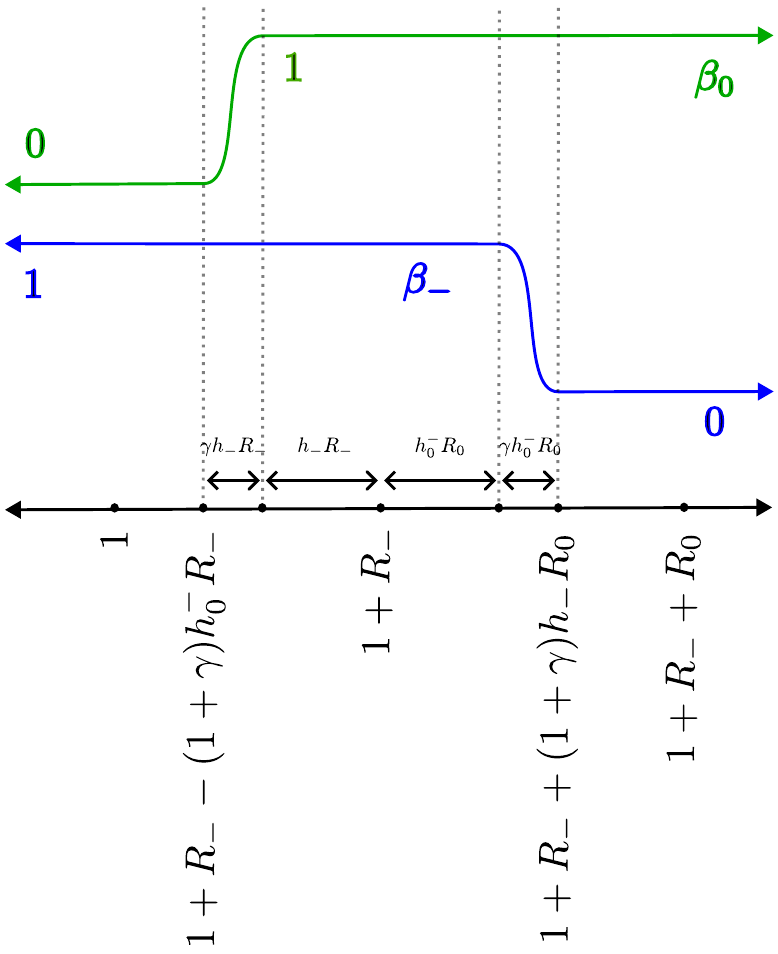}
    \caption{Cutoff functions for $t$-gluing}
    \label{fig:cutoff functions for t-gluing}
\end{figure}
We start by defining the pregluing, refer Figure~\ref{fig:cutoff functions for t-gluing}.
Choose gluing parameters $R_-$ and $R_0 > 0$. For $\beta: \R \to [0,1]$ as in Definition~\ref{defn: cutoff functions}, $0<h<1$ and $0\ll \gamma \ll 1$, define two cutoff functions
\begin{align}
    \beta_-(s) & = \beta\left(\frac{-s + (1+R_- + h_0(1+\gamma)R_0)}{\gamma h_0 R_0 }\right),\\
    \beta_0(s) & = \beta\left( \frac{s - (1+R_- - h_-(1+ \gamma)R_-)}{\gamma h_- R_-}\right).
\end{align}
Similar to the previous section, using the fact that the metric is the constant metric near the critical points, define the pregluing $u_\#: \R \to M$ by
\begin{align}
    u_\#(s) = \beta_-(s) u_-(s) + \beta_0(s) u_0^{R_- + R_0}.
\end{align}
To deform the pregluing, consider the pullback bundles
\begin{align}
    u_-^*(TM) &\text{ on }  (-\infty, 1+R_-],\\
    (u^{R_-+R_0}_0)^*(TM) &\text{ on } =[1+R_-, \infty).
\end{align}
Pick sections $\psi_-$ and $\tpsi_0$ of $ u_-^*(TM)$ and $(u^{R_-+R_0}_0)^*(TM)$, respectively, and deform $u_\#$ to get $u(\psi_-, \tpsi_0)$ given by
\begin{align}\label{eqn: t-gluing pregluing}
    s \mapsto &\exp_{u_\#(s)}(\beta_-\psi_- + \beta_0 \tpsi_0)(s)\\
    & = \beta_-(s)\exp_{u_-(s)}\psi_-(s) + \beta_0(s)\exp_{\left(u_0^{R_- + R_0}(s)\right)}\tpsi_0(s).
\end{align}
Up to this point, the pregluing and the deformation are exactly like in the $3$-component $0$-gluing case. The main change here is that the operator is now different. The base space for the gradient flow operator is now
\begin{align}
    \mathcal{B} \times (-\rho, \rho) := \mathcal{P}^{1,2}_{x_{-1}, x_1} TM \times (-\rho, \rho)
\end{align}
and the operator is given by
\begin{align}\label{eqn: operator for t-gluing}
    F(u,t) = \dot{u} + \nabla_{g_t} f \circ u =  \dot{u} + \nabla_g f \circ u + tV\circ u +O(t^2)
\end{align}

The deformed pregluing $u(\psi_-, \tpsi_0)$ is a flowline for $\nabla_{g_t} f$ if and only if it is a zero of $F$ given in Equation~\ref{eqn: operator for t-gluing}. One can expand the equation $F(u(\psi_-, \tpsi_0), t) = 0$ just as in Section~\ref{sec: equation for deformation to be a flowline} to get the following lemma.
\begin{lem}\label{lem: splitting the 2level t gluing eqn to thetas}
    There exist functionals $\Theta^\tau_-$ and $\Theta^\tau_0$ given by 
    \begin{align}
        \Theta^\tau_- (\psi_-, \tpsi_0) &= D_-\psi_- + \beta'_0 \psi_0 + \beta'_0 \tu_0 + Q_-(\psi_-),\label{defn: theta minus t-gluing}\\
        \Theta^\tau_0 (\psi_-, \tpsi_0,t) &= D^\tau_0 \tpsi_0 + \beta'_-\psi_- +\label{defn: theta 0 t-gluing} \\
        & \quad \quad + \beta'_-u_- + t V\circ \tu_0 +  Q_0(t,\psi_0^\tau)
    \end{align}
    where $D_*$ are the respective linearized operators (of the unperturbed operator $\frac{\partial}{\partial s} + \nabla f = 0$) and $Q_*$ are ``quadratic" (or higher order) functions of its input variables. Note for $Q_0$, the terms involving $t$ are supported only in the region where we perturbed the metric.
    Then we have $u(\psi_-, \tpsi_0)$ is a gradient flowline of $\nabla_t f$, that is,
    \begin{equation}\label{eqn: gradient flow equation for family of metrics gt}
        F(u(\psi_-, \tpsi_0), t) =0
    \end{equation}
    if and only if
    \begin{align}
        \beta_-\Theta^\tau_-(\psi_-, \tpsi_0) + \beta_0\Theta^\tau_0(\psi_-, \tpsi_0,t) = 0.
    \end{align}
    The superscripts $\tau$ always denote an appropriate translation as earlier.
\end{lem}
As in the $0$-gluing case, our strategy is to solve the two equations
\begin{align}
    \Theta^\tau_-(\psi_-, \tpsi_0) = 0  \label{eqn: theta minus t-gluing}, \text{ and }\\
    \Theta^\tau_0(\psi_-, \tpsi_0,t) = 0
    \label{eqn: theta 0 t-gluing}
\end{align}
iteratively. Let $\mathcal{H}_-$ denote the orthogonal complement of $\ker(D_-)$ in $H^{1,2}(u_-^* TM)$ and $\mathcal{H}^\tau_0$ denote the orthogonal complement of $\ker(D^\tau_0)$ in $H^{1,2}((u_0^{\tau})^* TM)$. We will solve Equations~\ref{eqn: theta minus t-gluing} and \ref{eqn: theta 0 t-gluing} for $\psi_- \in \cH_-$ and $\tpsi_0 \in \cH_0$. Let $\cB_{\epsilon, *}^\tau \subset \cH^\tau_*$ denote the $\epsilon$-ball for $* \in \{-,0\}$.

First, just as in Section~\ref{sec: solving for psi + and psi -}, we solve for $\psi_-$ as a function of $\psi_0$. The techniques are identical. So, we only state the analogous proposition.
\begin{prop}\label{prop: solving for theta minus t-gluing}
    For $\epsilon > 0$ and $R_-$ large enough, the following holds:
    \begin{enumerate}
        \item Given any $\tpsi_0 \in \cB_{\epsilon,0}$, there exists a unique vector field $\psi_-  \in \cB_{\epsilon, -}$ such that $\psi_- = \psi_-(\tpsi_0)$ solves \ref{eqn: theta minus t-gluing}.
        \item We get bounds on the Sobolev norm of $\psi_-$
        \begin{align}
            \|\psi_-\| \lesssim R_-^{-1}(\|\tpsi_0\|_{\supp \beta'_0} + \|\tu_0\|_{\supp \beta'_0})
        \end{align}
        \item The derivative of $\psi_-$ at a point $\tpsi_0 \in \cB_\epsilon$ defines a bounded linear functional $\mathcal{D}: \cH_0 \to \cH_-$ satisfying
        \begin{align}
            \|\mathcal{D}\eta\| \lesssim R_-^{-1} \|\eta\|.
        \end{align}
        \item The untranslated solutions $\psi_-(\psi_0) \in \cH_-$ depend implicitly on the gluing parameters $(R_-, R_0)$. When we wish to make this dependence explicit, we shall write $\psi_-(\psi_0,R_0,R_-)$.
        The derivative of $\psi_-$ with respect to $R_* \in \{R_-, R_0\}$ satisfy
        \begin{equation}
            \left\| \frac{\partial\psi_-}{\partial R_*} \right\| \lesssim \frac{1}{R_-}\left( \|\psi_0^\tau\|_{\supp \beta'_0 \cap \supp \beta_-} + \|\tu_0\|_{\supp \beta'_0 \cap \supp \beta_-} \right).
        \end{equation}
    \end{enumerate}
\end{prop}

 The next step is to solve Equation~\ref{eqn: theta 0 t-gluing} for $\tpsi_0 \in \cB_\epsilon$ after substituting $\psi_- = \psi_-(\tpsi_0)$ we just obtained in Proposition~\ref{prop: solving for theta minus t-gluing}. Let us rewrite $\Theta^\tau_0$ in Equation~\ref{eqn: theta 0 t-gluing} as 
 \begin{equation}
     D^\tau_0 \tpsi_0 + F^\tau_0(\tpsi_0) = 0
 \end{equation}
 where $F^\tau_0$ consists of all the terms other than $D^\tau_0$ in $\Theta^\tau_0$, refer Equation~\ref{defn: theta 0 t-gluing}, giving
 \begin{align}
     F^\tau_0(\tpsi_0) & := \beta'_-\psi_-(\tpsi_0) 
         + \beta'_-u_- + t\beta_0 V\circ \tu_0 +   Q_0(t,\psi_0^\tau),
 \end{align}
where we consider $\psi_-$ to be the function of $\tpsi_0$ obtained in Proposition~\ref{prop: solving for theta minus t-gluing}.

Just as in Section~\ref{sec: solving for psi 0}, $D^\tau_0$ is not invertible, so we cannot directly use a contraction mapping theorem. We introduce a choice of $L^2$-orthogonal projection $\Pi$ from $L^2(u_0^* TM)$ onto $\ker D_0^* \cong \coker D_0$ (its translated version is denoted by $\Pi^\tau$). Then, to solve Equation~\ref{eqn: theta 0 t-gluing}, it is sufficient to solve the following two equations simultaneously,
\begin{align}
    D^\tau_0 \tpsi_0 + (1 - \Pi^\tau) F^\tau_0(\tpsi_0) &=0, \label{eqn: theta 0 image part t-gluing} \text{ and }\\
    \Pi^\tau F^\tau_0(\psi^\tau_0) &= 0. \label{eqn: theta 0 cokernel part t-gluing}
\end{align}
Let $\psi_*$ denote the appropriate translations of $\psi^\tau_*$ so that they are vector fields over the untranslated flowlines $u_*$. Then $\psi_*$ satisfy the translated equations
\begin{align}
    D_0 \psi_0 + (1 - \Pi) F_0(\psi_0) &=0, \label{eqn: translated theta 0 image part t-gluing} \text{ and }\\
    \Pi F_0(\psi_0) &= 0. \label{eqn: translated theta 0 cokernel part t-gluing}
\end{align}
The first equation (either in the translated version $\tpsi_0$ or the untranslated version $\psi_0)$ can be solved by our now-familiar method of creating a contraction map, namely,
\begin{equation}
    \tpsi_0 \mapsto -(D^\tau_0)^{-1}(1 - \Pi^\tau)^\tau F_0(\tpsi_0),
\end{equation}
where $(D^\tau_0)^{-1}$ denotes the right inverse of $D^\tau_0$ when restricted to $\cH^\tau_0 \to \im D^\tau_0 = \im (1 - \Pi^\tau)$. We get the following theorem, whose proof is again analogous to that of Proposition~\ref{prop:solving for theta}; hence, we omit it here.
\begin{prop}\label{prop: solving for psi 0 t-gluing}
    For each $t > 0$, the following are true for $\epsilon > 0$ small enough and $R_+, R_0$ large enough.
    \begin{enumerate}
        \item There exists a unique $\psi_0 \in \cB_{\epsilon,0}$ satisfying Equation~\ref{eqn: translated theta 0 image part t-gluing}.
        \item This $\psi_0$ satisfies, for the $\psi_-$ obtained in Proposition~\ref{prop: solving for theta minus t-gluing},
        \begin{equation}
            \|\psi_0\| \lesssim R_0^{-1}(\|\psi_-\|_{\supp \beta'_-} + \|u_-\|_{\supp \beta'_-}) + t.
        \end{equation}
        \item $\psi_0$ defines a smooth section of $(u_0)^* TM$. Additionally, $\psi_-(\psi_0)$ obtained from Proposition~\ref{prop: solving for theta minus t-gluing} a smooth section of $(u_-^*)TM$.
        \item The vector fields $\psi_0$ and $\psi_-(\psi_0)$ depend implicitly on the gluing parameters $(R_-, R_0, t)$. These dependences are smooth.

        For $R_* \in \{R_-, R_0\}$ we have
        \begin{align}
            \left\|\frac{d\psi_0}{dR_*}\right\| \lesssim (R_0)^{-1}\left( \|u_-\|_{\supp \beta_-'} + \|\psi_-\|_{\supp \beta_-'} + \left\|\frac{d\psi_-}{dR_*}\right\|_{\supp\beta_-'}\right).
        \end{align}
        \begin{align}
            \left\|\frac{d\psi_-}{dR_*}\right\| \lesssim (R_-)^{-1}\left( \|u_0^\tau\|_{\supp \beta_0'} + \|\tpsi_0\|_{\supp \beta_0'} + \left\|\frac{d\tpsi_0}{dR_*}\right\|_{\supp\beta_0'}\right).
        \end{align}
        For the $t$ derivatives, we have
        \begin{align}
        \left \|\frac{d\psi_0}{dt}\right \| \leq 1 + \frac{1}{R_0} \left\|\frac{d\psi_-}{dt}\right\|_{\supp \beta_-'}
        \end{align}
        \begin{align}
        \left \|\frac{d\psi_-}{dt}\right \| \leq \frac{1}{R_-} \left\|\frac{d\psi_0}{dt}\right\|_{\supp \beta_0'}
        \end{align}
    \end{enumerate}
\end{prop}
\begin{rem}
In contrast to the $0$-gluing case, taking $t$-derivatives yields terms of order $1$ rather than terms that go to zero. So, as the pregluing parameters go to $\infty$, the $t$ derivative of $\psi_0$ is of order 1.
\end{rem}

We now move on to Equation~\ref{eqn: theta 0 cokernel part t-gluing}. As in Section~\ref {sec: obstruction section and gluing map}, we observe that to find a solution of Equation~\ref{eqn: operator for t-gluing}, it is enough to find a zero of Equation~\ref{eqn: theta 0 cokernel part t-gluing}. So, we define this as the ``obstruction section" and find its zeroes. The gluing map, as we now define, restricted to the zeroes of the obstruction section, will define the required ``gluing" and conclude the proof of Theorem~\ref{thm: t-gluing}.

As before, we first get rid of the redundancy of the two pregluing parameters $(R_-,R_0)$ by setting $R_-=R_0/A$ for large enough $A$. In particular, in light of the analogous estimates in the $0$-gluing section, we should set $\lambda_0^+R_0>\lambda_1^-R_-$.

Let $r$ be larger than the minimum values of $R_-$ and $R_0$ given by Propositions~\ref{prop: solving for theta minus t-gluing} and \ref{prop: solving for psi 0 t-gluing}.
To look at $\Pi F_0(\psi_0)$ from Equation~\ref{eqn: translated theta 0 cokernel part t-gluing} as a section of an appropriate bundle, define the {\bf obstruction bundle}, $\ob \to [r, \infty) \times t$ as the trivial bundle where the fiber over any $(R_0, t) \in [r, \infty) \times (-\rho, \rho)$ is
\begin{align}
    \ob_{(R_0, t)} = \hom(\coker(D_{u_0}), \R).
\end{align}
We are now ready to define the obstruction section, which is really a different perspective on Equation~\ref{eqn: theta 0 cokernel part t-gluing}.
\begin{defn}
    Define a section $\os:[r, \infty) \times (-\rho, \rho) \to \ob$, call the {\bf obstruction section}, as
    \begin{align}
        \os(R_0, t)(\sigma_0) := \langle \sigma_0, \Pi F_0(\psi_0(R_-, R_0, t))\rangle \text{ for all } \sigma \in \coker(D_{u_0}),
    \end{align}
    where $\psi_0(R_-, R_0, t)$ is the solution to Equation~\ref{eqn: theta 0 image part t-gluing} obtained from Propositin~\ref{prop: solving for psi 0 t-gluing} for the parameters $(R_-, R_0, t) = (R_0/A, R_0,t)$ for a fixed large integer $A \in \Z$ and $F_0$ is the corresponding term in Equation~\ref{eqn: theta 0 cokernel part t-gluing}.
\end{defn}
The obstruction section is smooth just like in Proposition~\ref{prop:smootheness of obstruction section}, except we need to restrict the perturbation parameter to either positive or negative. Similar to Lemma~\ref{lem: transversality of obstruction section}, the obstruction sections will also be transverse to the zero section. 
\begin{prop}\label{prop: smoothness and transversality of obstruction section t-gluing}
Let $\os_+ := \os|_{[r,\infty) \times (0,\rho)}$ and $\os_- := \os|_{[r,\infty) \times (-\rho,0)}$ denote two restrictions of the obstruction section.
    The sections \begin{align}
        \os_+:[r, \infty) \times (0, \rho) \to \ob \text{ and }
        \os_-:[r, \infty) \times (-\rho, 0) \to \ob
    \end{align} are smooth sections. The sections $\mathfrak{s}_\pm$ are also transverse to the zero sections.
\end{prop}

The fact that $\mathfrak{s}_\pm$ are transverse to zero comes directly from showing that its $t$-derivative is bounded away from zero. We still need to count how many zeroes $\mathfrak{s}_\pm$ has given a fixed $t$.
Nonetheless, Proposition~\ref{prop: smoothness and transversality of obstruction section t-gluing} implies that $\os_\pm^{-1}(0)$ are manifolds. So, we can define a  ``gluing" map by Definition~\ref{defn: t-gluing maps} on $\os_\pm^{-1}(0)$.
\begin{defn}\label{defn: t-gluing maps}
Define the {\bf $\mathbf{(R_-, R_0, t)}$-gluing}, denoted by $u(R_-, R_0;t)$ to be the deformed pregluing~\ref{eqn: t-gluing pregluing} when the $\psi_-$ and $\psi_0$ are those obtained with the parameters $(R_-, R_0)$ and perturbation parameter $t$ in Propositions~\ref{prop: solving for theta minus t-gluing} and \ref{prop: solving for psi 0 t-gluing}.
    Define two {\bf gluing maps} 
    \begin{align}
        G_\pm : \os_\pm^{-1}(0) \to \ms_{x_{-1}, x_2}, \quad (R_0, t) \mapsto u(R_0/A, R_0;t).\label{defn: t-gluing map}
    \end{align}
\end{defn}

We now want to show that the gluing maps above capture all the flowlines ``close to breaking" to the broken flowline $(u_-, u_0)$. To do this, we adapt definitions from the previous section rather than rewrite similar ones for brevity. Analogous to Definition~\ref{defn:open nbhd around broken u}, define the {\bf space of paths close to $(u_-, u_0)$}, $\tilde G_\delta(u_+, u_0)$, to be concatenated paths $v_- \star v_0$ satisfying analogous ``closeness" properties. Let the {\bf space of $g_t$-flowlines close to $(u_-, u_0)$} be the subset $G^-_\delta(u_-, u_0) \subset \tilde G_\delta(u_-, u_0) \times (-\rho, 0)$ or $G^+_\delta(u_-, u_0) \subset \tilde G_\delta(u_-, u_0) \times ( 0, \rho)$ consisting of tuples $((v_-, v_0), t)$ such that $v_- \star v_0$ is a flowline of $-\nabla_t f$. Given $\delta > 0$, denote the space of paths close to breaking to $(u_-, u_0)$ that we obtain in the image of the gluing map as $U_\delta^\pm = G^{-1}_\pm (G_\delta^\pm)$. Let $U_\delta = U^+_\delta \cup U^-_\delta$. We now have the parametrization result analogous to Theorem~\ref{thm:gluing map is homeo}. The proof contains similar ideas to those in the proof of Theorem~\ref{thm:gluing map is homeo}, so we omit redoing them.
\begin{thm}If $r$ is sufficiently large and $\rho$ is sufficiently small, then
    \begin{enumerate}
        \item[(a)] 
        the entire base space
            $[r,\infty) \times ((-\rho, 0) \cup (0, \rho)) \subset U_\delta$, and        
        \item[(b)] the gluing maps \ref{defn: t-gluing maps} restrict to homeomorphisms
        \begin{align}
            G_\pm: \os^{-1}_\pm (0) \cap U_\delta^\pm \to  G_\delta^\pm(u_+, u_0)/\mathbb{R}.
        \end{align}
    \end{enumerate}
\end{thm}
\subsection{The linearized section $\os_0$}
Just as in the unperturbed case in Section~\ref{sec: obg 0-gluing}, we would like to ``count" the zeroes of the obstruction sections $\os_\pm$, but counting them directly is difficult. So, we define similar ``linearized" obstruction sections.

\begin{defn}
    Define the linearized section by defining how it pairs with the element $\sigma_0$ as, 
    \begin{align}
        \os^-_{0}: [r, \infty) \times (-\rho, 0) &\to \ob, \quad \os^+_0: [r, \infty) \times ( 0, \rho) \to \ob\\
        \os^\pm_0(R_0, t)(\sigma_0) &:= T_{2 + R_- + R_0} \langle \beta'_- u_- + t\beta_0 V\circ \tu_0, \sigma^\tau_0\rangle. 
    \end{align}
\end{defn}
Having defined the linearized section, we are ready to complete the proof of Theorem~\ref{thm: t-gluing}.
\begin{proof}[Proof of Theorem~\ref{thm: t-gluing}]

    We can define $\os_{00}$ as
    \begin{equation}
        \os_{00}(R_0, t) := - \langle b_-, a_-\rangle e^{-\lambda_0^+ { ( R_- + R_0)}} + t\langle V, \sigma_0 \rangle.
    \end{equation} The same argument as Proposition~\ref{thm: gluing} shows it suffices to compare $\mathfrak{s}$ with $\mathfrak{s}_{00}$ instead of $\mathfrak{s}_0$ and show these two are ``$C^1$-close'' or ``$C^0$-close''. Given $R_-$ and $R_0$, $\os_{00}(R_0, t) = 0$ if and only if
    \begin{equation}\label{eqn: t solving linearized os}
        t = \frac{ \langle b_-, a_-\rangle e^{-\lambda_0^+ { ( R_- + R_0)}}}{ \langle X, \sigma \rangle}. 
    \end{equation}
    This immediately tells us, as $\langle V, \sigma_0\rangle > 0$ from our choices, that we get a one-parameter family of solutions given by Equation~\ref{eqn: t solving linearized os} for $\os_0^+$ and $(\os_0^-)^{-1}(0) = \emptyset$ for $\langle b_-, a_-\rangle > 0$ and vice-versa for $\langle b_-, a_-\rangle < 0$. 

     The next step of the proof is to show that for $|t|$ sufficiently small, the linearized section and the obstruction section, both viewed as functions of $R_0$, have the same number of zeroes. In the case of $0$-gluing we achieved this by showing the two are ``$C^1$-close'' to each other. Here, the setup is slightly different, so we sketch the strategy.

    In the case $\mathfrak{s}_{00}$ does not have any zeroes, the proof follows by showing all the other terms that appear in $\mathfrak{s}$ are much smaller than $e^{-\lambda_0^+  ( R_- + R_0)} + t$ by exponential factors.
    In particular, we need to estimate the norms of the terms 
    \[
    \langle \sigma_0^\tau, \beta_-'\psi_-\rangle, \quad \langle Q_0(t,\tpsi_-), \sigma_0^\tau\rangle.
    \]
    The same exponential decay estimates in Section \ref{sec: C0 closeness of linearized section} also show the nonlinear section $\mathfrak{s}$ does not have zeroes.

    In the case where $\mathfrak{s}_{00}$ has a unique zero, after setting all the appearing constants to $1$, the full obstruction section takes the form
    \[
    \mathfrak{s} = e^{-\lambda_0^+(1+1/A)R_0} -t +L(t,R_0) +Q(t,R_0)
    \]
    where $L(t,R_0) =\langle \sigma_0^\tau, \beta_-'\psi_-\rangle $ and $Q(t,R_0) =\langle Q_0(t,\tpsi_-), \sigma_0^\tau\rangle $ are smooth functions of $(t,R_0)$.
    
    If we take the $R_0$ derivative of $\os_{00}$ we see it does not change sign, so the zero of $\os_{00}$ is unique.

    Running the same estimates, we note that we have
    \[
    L\leq te^{-\eta R_0} +e^{-\eta R} e^{-\lambda_0^+(1+1/A)R_0}, \quad 
    \frac{dL}{dR_0} \leq t e^{-\eta R_0}+ e^{-\eta R} e^{-\lambda_0^+(1+1/A)R_0}\] for some $\eta>0$.
    We also have
    \begin{align}
    Q(t,R_0) &\leq t^2 + e^{-\eta R_0} e^{-\lambda_0^+(1+1/A)R_0} +te^{-\eta R_0},\\ \frac{dQ}{dR_0} &\leq e^{-\eta R_0} e^{-\lambda_0^+(1+1/A)R_0} + te^{-\eta R_0} +t^2
    \end{align}
    
    It could be the case that $\eta<(1+1/A)\lambda_0^+$, so it's not a priori obvious that for every value of $R_0$, the derivative of the second term $F$ or the third term $G$ is much smaller than the first term.

    This is remedied by our key observation that to show the zero of $\mathfrak{s}$ is unique, it suffices that its derivative at any of its zeroes has the same sign as the derivative of $\os_{00}$ (which is nonvanishing). To be more precise, for $t$ very small, we need to show 
    \[
    \left|\frac{dL}{dR_0}\right|, \left |\frac{dQ}{dR_0}\right |\ll e^{-\lambda_0^+(1+1/A)R_0}
    \]
    only for $e^{-\lambda_0^+(1+1/A)R_0} \in [(1-\epsilon)t,(1+\epsilon)t]$ since the zero must appear\footnote{The correct phrasing is for any $\epsilon >0$, if $t$ is sufficiently small the zero of $\os$ must occur in the interval $e^{-\lambda_0^+(1+1/A)R_0} \in [(1-\epsilon)t,(1+\epsilon)t]$} in this range of $R_0$, but in this range
    \[
    \frac{d\os_{00}}{dR_0} \sim e^{-\lambda_0^+(1+1/A)R_0},\quad 
    \left|\frac{dL}{dt}\right| \sim e^{-\lambda_0^+(1+1/A)R_0} e^{-\eta R_0} 
    \]
    Similar exponential decay estimates also show that
    \[
    \left |\frac{dQ}{dR_0}\right |\leq e^{-\lambda_0^+(1+1/A)R_0}e^{-\eta R_0},
    \]
    and our conclusion follows.
\end{proof}
The above $t$-gluing can be extended to multiple-component flowlines. Such flowlines can appear in the compactification on moduli spaces of flowlines as seen in Lemma~\ref{lem: compactification}. Unfortunately, the asymptotic relations no longer look as nice as in Theorem~\ref{thm: t-gluing}. We get one equation for each non-tranversely cutout flowline.

We first describe a prototypical example, and then state a Theorem.
\begin{figure}[h]
    \centering
    \includegraphics[scale=0.5]{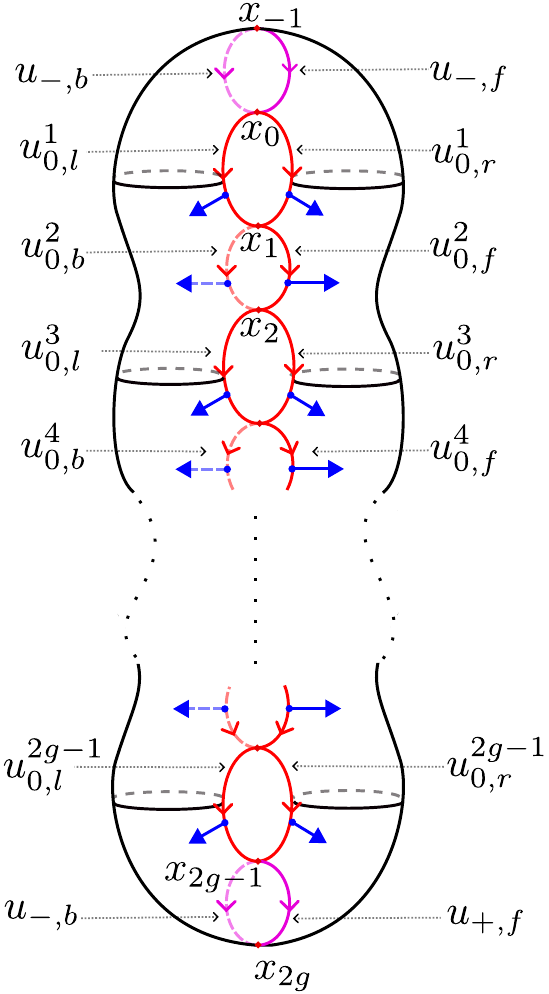}
    \caption{Genus $g$ surface embedded in $\R^3$ symmetrically with respect to reflection $x \mapsto -x$. Red flowlines are not transversely cut out. Blue vectors represent the choice of cokernel elements.}
    \label{fig:height function on genus g surface}
\end{figure}
\begin{ex}
    Consider the genus $g$ surface $\Sigma_g$ embedded in $\R^3$ symmetric with respect to the reflection $x \mapsto -x$ as shown in Figure~\ref{fig:height function on genus g surface}. Let the height function, that is, the projection to the $z$-coordinate, be the Morse function $f$ and consider the metric $g$ obtained from restricting the standard Euclidean metric of $\R^3$. In keeping with the simplifications of this paper, we actually slightly modify the metric to make it Euclidean in each Morse neighbourhood of the critical point. 
    
    We get $2g+2$ critical points, of which the maximum $x_0$ is of index $2$, the minimum $x_{2g+2}$ is of index $0$, and the rest all have index equal to $1$. Notice we have $4g-2$ non-transversely cut out flowlines. 

    For each $j = 1, \dots, 2g - 1$, some broken flowlines will be $t$-gluable depending on the choices of the perturbation of the metric near each of the non-transversely cut out flowlines.
\end{ex}

We consider the bifurcation analysis of a broken flowline built from a single transverse flowline followed by $m$ consecutive non-transverse gradient flowlines. We write down the combinatorial criteria that predict whether this broken flowline glues after the perturbation or disappears.  We still have to restrict to the case when the maximum dimension of any cokernel is $1$. We leave this as a Theorem without proof, but only remark that the proof would be analogous to that of Theorem~\ref{thm: t-gluing}\footnote{The analogue of this theorem in the case of circle valued Morse theory is discussed in \cite{Hutchings_blog}.}.
\begin{thm}\label{thm: multi level t-gluing}
    For $(f,g)$ a pair of a Morse function and a metric, let 
    \begin{align}
        u_- \in \ms(x_{-1}, x_0), \quad u_0^j \in \ms(x_{j-1}, x_j) \text{ for } j = 1, \dots, m,
    \end{align}
    with 
    \begin{align}
        \ind(x_{-1}) = k + 1, \quad \ind(x_j) = k \text{ for } j = 0, \dots, m.
    \end{align}
    For $j = 0, \dots, k-1$, let $\lambda_j^+$ be the smallest positive eigenvalue and $\lambda_j^-$ be the largest (least negative) negative eigenvalue of $\hess_{x_j} f$. For $j = 1, \dots, m$, fix $v_j \in (T_{u_0^j(0)} W^u(x_{j-1}) \cap T_{u_0^j(0)} W^s(x_j))^\perp$ and denote by $\sigma_0^j \in \coker D_{u_0^j}$ the corresponding cokernel element to $v_j$ under the identification $(T_{u_0^j(0)} W^u(x_{j-1}) \cap T_{u_0^j(0)} W^s(x_j))^\perp \cong \coker D_{u_0^j}$. 

    Let $g_t$ denote a $ t$-dependent perturbation of the metric supported away from the critical points and transversely cut out index $1$ gradient flowlines. We assume 
    \begin{align}
        \nabla_{g_t} f = \nabla_g f + tV +O(t^2).
    \end{align}

    Assume there exists $b^-_{j-1} \in T_{x_{j-1}} M$ and $b^+_j \in T_{x_{j}} M$  such that
    \begin{align}
        \sigma_0^j &=  e^{\lambda_{j-1}^+ s}b^-_{j-1} + \text{higher order terms}&\text{ for } s < -1,\\
        \sigma_0^j &=  e^{\lambda_{j}^- s}b^+_{j} +\text{higher order terms} &\text{ for } s > 1.
    \end{align}
    Similarly, assume we have $a^+_0\in T_{x_0}M$, $a_j^\pm \in T_{x_j}M$ such that 
    \begin{align}
        u_- &= e^{-\lambda_0^+ s}a_0^+ + \text{higher order terms}&\text{ for } s > 1, \\
        u_0^j &=  e^{-\lambda_{j-1}^- s}a_{j-1}^- + \text{higher order terms} &\text{ for } s < -1, , j = 1, \dots, m.\\
        u_0^j &= e^{-\lambda_j^+ s} a_j^+ +  \text{higher order terms}&\text{ for } s > 1, j = 1, \dots, m.
    \end{align}
Assume that $\langle a_0^+, b_0^-\rangle$, $\langle a_j^-, b_j^+\rangle$, and $\langle a_{j}^+, b_j^- \rangle$ are non-zero for $j = 0, \dots, m$.
    Then, there exists a one-parametric family $u_t \in \ms(x_{-1}, x_m; g_t)$ if and only  if there exists $\rho > 0$ small enough and $r$ sufficiently large such that for all $t \in (-\rho, 0)$ or $t \in (0, \rho)$, there exist $R_1,...,R_m>r$ satisfying all of the following equations:
    \begin{align}
        -e^{-\lambda_0^+ R_1}\langle a_0^+, b_0^-\rangle + e^{-\lambda_1^- R_2}\langle a_1^-, b_1^+\rangle + t\langle V,\sigma_0^1\rangle  &= 0;\\
        -e^{-\lambda_{j}^+ R_{j+1}}\langle a_{j}^+, b_j^-\rangle + e^{-\lambda_j^- R_{j+2}}\langle a_{j+1}^-, b_{j+1}^+\rangle + t\langle V, \sigma_0^{j+1}\rangle &= 0 \text{ for all } j = 1, \dots, m-2;\\
        e^{-\lambda_{m-1}^+ R_m}\langle a_{m-1}^+, b_{m-1}^-\rangle + t\langle V,\sigma_0^m\rangle &= 0.
    \end{align}      
\end{thm}
We can obtain analogous statements for broken flowlines of the form $$(u_0^1, \dots, u_0^m, u_+)$$ where each $u_0^j$ has a $1$-dimensional cokernel and $u_+$ is transversely cut out with Fredholm index $1$. Once we have proved these theorems, we can define Morse differentials by counting broken flowlines. To define the differentials, first fix a perturbation of the metric $t$.
The differential would then be a count of total index $1$ broken flowlines that are $t$-gluable for $t>0$.

\bibliographystyle{alpha}
\bibliography{references.bib}

\end{document}